\documentclass[11pt,reqno]{amsart}
\usepackage{a4wide}

\numberwithin{equation}{section}
\usepackage{mathrsfs}
\usepackage{amsfonts}
\usepackage{amsmath}
\usepackage{stmaryrd}
\usepackage{amssymb}
\usepackage{amsthm}
\usepackage{mathrsfs}
\usepackage{url}
\usepackage{amsfonts}
\usepackage{amscd}
\usepackage{indentfirst}
\usepackage{enumerate}
\usepackage{amsmath,amsfonts,amssymb,amsthm}
\usepackage{amsmath,amssymb,amsthm,amscd}
\usepackage{graphicx,mathrsfs}
\usepackage{appendix}

\usepackage[numbers,sort&compress]{natbib}

\usepackage[colorlinks,linkcolor=blue,anchorcolor=blue,citecolor=blue]{hyperref}
\usepackage{color}

\newcommand{\R}{\mathbb{R}}

\newcommand{\N}{\mathbb{N}}
\renewcommand{\theequation}{\arabic{section}.\arabic{equation}}

\newtheorem{Thm}{Theorem}[section]
\newtheorem{Lem}[Thm]{Lemma}

\newtheorem{Prop}[Thm]{Proposition}

\newtheorem{Rem}[Thm]{Remark}

\begin{document}

\title[Brezis-Nirenberg problem]
{The number of positive solutions  to\\ the Brezis-Nirenberg problem}
\author[D. Cao, P. Luo and S. Peng]{Daomin Cao, Peng Luo and Shuangjie Peng}

\address[Daomin Cao]{Institute of Applied Mathematics, AMSS, The Chinese Academy of Sciences, Beijing 100190, P.R.China and University of Chinese Academy of Sciences, Beijing 100049,  P.R. China}
\email{dmcao@amt.ac.cn}

\address[Peng Luo]{School of Mathematics and Statistics and Hubei Key Laboratory of Mathematical Sciences, Central China Normal University, Wuhan 430079, China}
\email{pluo@mail.ccnu.edu.cn}

\address[Shuangjie Peng]{School of Mathematics and Statistics and Hubei Key Laboratory of Mathematical Sciences, Central China Normal University, Wuhan 430079, China}
\email{sjpeng@mail.ccnu.edu.cn}

\begin{abstract}
In this paper we are concerned with the well-known  Brezis-Nirenberg problem
\begin{equation*}
\begin{cases}
-\Delta u= u^{\frac{N+2}{N-2}}+\varepsilon u, &{\text{in}~\Omega},\\
u>0, &{\text{in}~\Omega},\\
u=0, &{\text{on}~\partial \Omega}.
\end{cases}
\end{equation*}
The existence of multi-peak solutions  to the above problem for small $\varepsilon>0$ was obtained in \cite{Musso1}. However, the uniqueness or the exact  number of positive solutions to the above problem is still unknown.
Here we focus on the local uniqueness of  multi-peak solutions and  the exact  number of positive solutions to the above problem for small $\varepsilon>0$.

By using various  local Pohozaev identities and blow-up analysis, we first detect the relationship between the profile of the blow-up solutions and Green's function of the domain $\Omega$ and then obtain a type of
 local uniqueness results  of blow-up solutions. Lastly  we give a description of the number of positive solutions for small positive $\varepsilon$, which depends also on Green's function.
\end{abstract}

\date{\today}
\maketitle

{\small\small
\keywords {\noindent {\bf Keywords:}
 {Critical Sobolev exponent, Local Pohozaev identity, Existence of solutions, \\ Exact number of solutions, Green's function}
\smallskip
\newline
\subjclass{\noindent {\bf 2010 Mathematics Subject Classification:} 35A02 $\cdot$ 35B09
$\cdot$ 35J05 $\cdot$
35J08 $\cdot$ 35J60}
}
\section{Introduction and main results}
\setcounter{equation}{0}
In this paper, we consider the following Brezis-Nirenberg problem
\begin{equation}\label{1.1}
\begin{cases}
-\Delta u= u^{\frac{N+2}{N-2}}+\varepsilon u, &{\text{in}~\Omega},\\
 u>0, &{\text{in}~\Omega},\\
u=0, &{\text{on}~\partial \Omega},
\end{cases}
\end{equation}
where $N\geq 3$, $\varepsilon>0$ is a small parameter, $\Omega$ is a smooth and bounded domain in $\R^N$.

In 1983, Brezis and Nirenberg  proved in their celebrated paper \cite{Brezis} that if $N\geq 4$,  problem \eqref{1.1} has a solution for
  $\varepsilon \in (0,\lambda_1)$, where $\lambda_1$ denotes the first eigenvalue of $-\Delta$ with
0-Dirichlet boundary condition on $\partial\Omega$.
Also it is well known in \cite{Pohozaev} that problem \eqref{1.1} admits no solutions when $\Omega$ is star-shaped and  $\varepsilon=0$.
On the other hand, Bahri and  Coron \cite{Bahri1} gave an existence result of a positive solution to  problem \eqref{1.1} for $\Omega$ with a nontrivial topology and
 $\varepsilon=0$. Since then a lot of attention has been paid to the limiting behavior  of the solutions $u_\varepsilon$
 of \eqref{1.1} as $\varepsilon\rightarrow 0$. To state such type of results, we introduce  some facts on Green's function.

The Green's function $G(x,\cdot)$ is the solution of
\begin{equation*}
\begin{cases}
-\Delta G(x,\cdot)= \delta_x, &{\text{in}~\Omega}, \\
G(x,\cdot)=0, &{\text{on}~\partial\Omega},
\end{cases}
\end{equation*}
where $\delta_x$ is the Dirac function. For $G(x,y)$, we have the following form
\begin{equation*}
G(x,y)=S(x,y)-H(x,y), ~(x,y)\in \Omega\times \Omega,
\end{equation*}
where $S(x,y)=\frac{1}{(N-2)\omega_N|y-x|^{N-2}}$ is the singular part and $H(x,y)$ is the regular part of $G(x,y)$, $\omega_N$ is a measure of the unit sphere of $\R^N$. For any $x\in \Omega$, we denote $R(x):=H(x,x) $, which is called the Robin function.

Rey \cite{Rey2} proved that if a solution $u_\varepsilon$ of \eqref{1.1}  satisfies
\begin{equation}\label{caa}
|\nabla u_{\varepsilon}|^2\rightharpoonup S^{N/2} \delta_{x_0},~\mbox{as}~\varepsilon \rightarrow 0,
\end{equation}
with $S$  the best Sobolev constant  defined by $$
S=\inf \Big\{\displaystyle\int_{\Omega}|\nabla u|^2~~\big| ~~u\in H^1_0(\Omega),~\displaystyle\int_{\Omega}|u|^{2^*}=1\Big\},$$
then $x_0$ is a  critical point of $R(x)$. Conversely  if $x_0$ is a nondegenerate critical point of $R(x)$ and $N\geq 5$, then \eqref{1.1} has a solution $u_\varepsilon$  satisfying \eqref{caa}.
Similar results are also proved in \cite{Han}. Later, Glangetas \cite{Glangetas} proved that
the solution $u_\varepsilon$ of \eqref{1.1} satisfying \eqref{caa} is unique for $\varepsilon$ small enough
 under some additional conditions.

A natural question is whether \eqref{1.1} has a solution $u_\varepsilon$
concentrated at multi-points.
 In this aspect, Musso and Pistoia \cite{Musso1} gave an affirmative answer. To state their results, we need to introduce  some notations.
As is well-known, the equation $-\Delta u= u^{\frac{N+2}{N-2}} ~~\mbox{in}~\R^N$ has a family of solutions
\begin{equation*}
U_{x,\lambda}(y)=C_N\frac{\lambda^{(N-2)/2}}{(1+\lambda^2|y-x|^2)^{(N-2)/2}},
\end{equation*}
where $x\in\R^N$, $\lambda\in \R^+$ and $C_N=\big(N(N-2)\big)^{(N-2)/{4}}$.
Set
\begin{equation}\label{a1}
A=\displaystyle\int_{\R^N}U^{\frac{N+2}{N-2}}_{0,1},
~B=\displaystyle\int_{\R^N}U^{2}_{0,1}.
\end{equation}
Let $\Psi_{k}: \Omega^{k}\times (\R^+)^{k}\rightarrow \R$ be defined by
\begin{equation*}
\Psi_{k}(x,\lambda)= A^2 \Big(M_k(x)\lambda^{(N-2)/2},\lambda^{(N-2)/2}\Big)-B\sum^{k}_{j=1}\lambda^{2}_{j},
\end{equation*}
where $\lambda^{(N-2)/2}=\Big(\lambda_1^{(N-2)/2},\cdots,\lambda_k^{(N-2)/2}\Big)^T$, the matrix $M_k(x)=\Big(m_{ij}(x)\Big)_{1\leq i,j\leq k}$ is defined by
\begin{equation*}
m_{ii}(x)=R(x_i),~m_{ij}(x)=-G(x_i,x_j),~\mbox{if}~i\neq j.
\end{equation*}
 Musso and Pistoia  \cite{Musso1}  proved  that there exists a family of solutions $u_{\varepsilon}$ to \eqref{1.1} satisfying
\begin{equation}\label{4-6-1}
|\nabla u_{\varepsilon}|^2\rightharpoonup S^{N/2} \sum^k_{i=1}\delta_{a_i},~\mbox{as}~\varepsilon \rightarrow 0,
\end{equation}
if $N\geq 5$ and  $(a^k,\Lambda^k)$ is a  nondegenerate critical  point of $\Psi_{k}$ with
$a^k=(a_1,\cdots,a_{k})$ and some $\Lambda^k=(\lambda_{1},\cdots,
\lambda_{k})$.

On the other hand, for any given $f\in H^1(\Omega)$, let $P$ denote the projection from
$H^1(\Omega)$ onto $H^{1}_0(\Omega)$, i.e., $u=Pf$ is the solution of
\begin{equation*}
\begin{cases}
\Delta u=\Delta f, &{\text{in}~\Omega}, \\
u=0, &{\text{on}~\partial\Omega}.
\end{cases}
\end{equation*}
Now for any $x\in\Omega$ and $\lambda\in \R^+$, we define
\begin{equation*}
\begin{split}
E_{x,\lambda}=\Big\{v\in H^1_0(\Omega)\Big|& ~\Big\langle \frac{\partial PU_{x,\lambda}}{\partial \lambda},v \Big\rangle=\Big\langle \frac{\partial PU_{x,\lambda}}{\partial x_i},v\Big\rangle=0,~\mbox{for}~i=1,\cdots,N\Big\},
\end{split}
\end{equation*}
where $\|\cdot\|$  denotes the basic norm in the Sobolev space $H^1_0(\Omega)$ and $\langle\cdot,\cdot\rangle$  means the corresponding inner product.
Then our first result is on the structure of the blow-up solutions of  \eqref{1.1}.
\begin{Thm}\label{th1-1}
Let $N\ge 5$ and suppose that  $u_{\varepsilon}(x)$ is a solution of \eqref{1.1} with \eqref{4-6-1}. Then  $M_k(a^k)$ is a non-negative matrix  with
$a^k=(a_1,\cdots,a_{k})$ and
 $u_{\varepsilon}(x)$ can be written as
\begin{equation*}
u_{\varepsilon}=\sum^k_{j=1} PU_{x_{j,\varepsilon}, \lambda_{j,\varepsilon}}+w_{\varepsilon},
\end{equation*}
satisfying, for $j=1,\cdots,k$,
 $\lambda_{j,\varepsilon}=
\big(u_\varepsilon(x_{j,\varepsilon})\big)^{\frac{2}{N-2}}$,
\begin{equation*}
 x_{j,\varepsilon}\rightarrow a_j, ~\lambda_{j,\varepsilon}\rightarrow +\infty,~ \|w_{\varepsilon}\|=o(1)~\mbox{and}~w_\varepsilon\in \bigcap^k_{j=1}E_{x_{j,\varepsilon},\lambda_{j,\varepsilon}}.
\end{equation*}
Moreover if  $M_k(a^k)$ is a positive matrix,  then there exist two constants $C_1,C_2$ such that
\begin{equation*}
  0<C_1\leq \varepsilon ^{\frac{1}{N-4}}\lambda_{j,\varepsilon}\leq C_2<+\infty.
\end{equation*}
Furthermore if we denote (by choosing subsequence)
\begin{equation*}
\lambda_j:=\lim_{\varepsilon\rightarrow 0}\big(\varepsilon ^{\frac{1}{N-4}}\lambda_{j,\varepsilon}\big)^{-1},~\mbox{for}~j=1,\cdots,k,
\end{equation*}
then $(a^k,\Lambda^k)$ is a  critical point of $\Psi_{k}$ with
$a^k=(a_1,\cdots,a_{k})$ and $\Lambda^k=(\lambda_{1},\cdots,
\lambda_{k})$.
\end{Thm}

When $\Omega$ is a convex domain, it is known from \cite{Grossi} that $\Psi_k(x,\lambda)$ has no critical points in $\Omega^k\times(\R^+)^k$  for $k\ge 2$. Hence combining  Theorem \ref{th1-1}, we conclude  that \eqref{1.1} has no solutions
blowing-up at multiple points on convex domains.
On the other hand, from \cite{Caffarelli,Cardaliaguet,G2002}, we know that Robin function $R(x)$ has a unique critical point on convex
domains, which is also non-degenerate under some conditions.
Therefore, considering the uniqueness result of Glangetas \cite{Glangetas}, we see that problem \eqref{1.1} has a unique  solution  for $\varepsilon$ small enough
and a convex domain $\Omega$.

Next, to study the number of concentrated solutions, for any given $a^k=(a_1,\cdots,a_{k})$ satisfying $\nabla_x \Psi_k(a^k,\Lambda^k)=0$ for some $\Lambda^k\in (\R^+)^k$, we define
$$S_k= \Big\{\Lambda^k=(\lambda_1,\cdots,\lambda_k),
\nabla_x \Psi_k(a^k,\Lambda^k)=0,~\nabla_\lambda \Psi_k(a^k,\Lambda^k)=0\Big\}.$$
Now we can count the number of solutions to \eqref{1.1} satisfying \eqref{4-6-1},
which can be stated as follows.
\begin{Thm}\label{th1.1}
For $N\geq 6$ and any given $a^k=(a_1,\cdots,a_{k})$, suppose that  $M_k(a^k)$ is a positive matrix and
 $(a^k,\Lambda^k)$ is a nondegenerate critical point of $\Psi_{k}$ for any $\Lambda^k\in S_k$.
 Then for $\varepsilon>0$ sufficiently small,
$$
\mbox{the number of solutions to
\eqref{1.1} satisfying \eqref{4-6-1}}~=\sharp S_k,$$
where $\sharp S_k$ is the number of the elements in the set $S_k$.
\end{Thm}

In Theorem \ref{th1.1}, the existence and  non-degeneracy of  critical points to $\Psi_{k}$ play a crucial role. In fact,
the existence of critical points to $\Psi_{k}$ and  their  non-degeneracy are very important topics.
Musso and Pistoia  \cite{Musso1}  constructed a class of $\Omega_{\delta}$ for small $\delta$
and proved
the existence of stable critical points of $\Psi_{k}$ on $(\Omega_\delta)^k\times (\R^+)^k$ for some domain $\Omega_\delta$.
Specially, let $\Omega_0=\bigcup^k_{i=1} \Omega_i$, where $\Omega_1,\cdots,\Omega_k$ are $k$ smooth bounded domains such that $\Omega_i\bigcap \Omega_j=\emptyset$ if $i\neq j$. Then the function $\Psi_{k}$ has a strict minimum point in the connected component $\Omega_1\times \cdots \times \Omega_k\times(\R^+)^k$ of the set $(\Omega_0)^k\times (\R^+)^k$.
Moreover, assume that
$$\Omega_i\subset \big\{(x_1,x')\in \R\times \R^{N-1}|~a_i\leq x_1\leq b_i\big\}, ~\mbox{with}~ b_i<a_{i+1},~i=1,\cdots,k.$$
 For any $\delta>0$, let $C_\delta=\big\{(x_1,x')\in \R\times \R^{N-1}|~x_1\in (a_1,b_k),|x'|\leq \delta\big\}$
 and $\Omega_\delta$ be a smooth connected domain such that $\Omega_0\subset \Omega_\delta\subset \Omega_0\bigcup C_\delta$. Then if $\delta$ is small enough, the function $\Psi_{k}$ has a strict minimum point on $(\Omega_\delta)^k\times (\R^+)^k$, which is stable.

Very recently, Bartsch, Micheletti and Pistoia \cite{Bartsch1} proved that all critical points of $\Psi_{k}$ are non-degenerate  for most domains. Specially, for a bounded domain $\Omega\subset \R^N$ of class $C^{m+2,\alpha}, m\geq 0, 0<\alpha<1$, $\psi\in C^{m+2,\alpha}$, the set
\begin{equation*}
\Omega_\psi:=(id+\psi)(\Omega)=\big\{x+\psi(x):x\in \Omega\big\}
\end{equation*}
is again a bounded domain of class $C^{m+2,\alpha}$ provided $\|\psi\|_{C^1}<\rho(\Omega)$ is small. Setting
\begin{equation*}
\mathcal{B}^{m+2,\alpha}(\Omega):=\big\{\psi\in C^{m+2,\alpha}(\bar\Omega,\R^N):~\|\psi\|_{C^1}<\rho(\Omega)\big\},
\end{equation*}
then the set
$$\mathcal{M}^{m+2,\alpha}(\Omega):=\big\{\psi\in
\mathcal{B}^{m+2,\alpha}(\Omega):~\mbox{all critical points of $\Psi_{k}$ are non-degenerate on}~(\Omega_\psi)^k\times (\R^+)^k\big\}$$
is a dense subset of $\mathcal{B}^{m+2,\alpha}(\Omega)$.

The above results give us that   there exist some domains $\Omega$ such that  $\Psi_{k}$ possesses some critical points and all these  critical points are non-degenerate on $(\Omega)^k\times (\R^+)^k$.
We  can also refer to \cite{Bartsch,Micheletti} and  the references
therein.

\smallskip

Furthermore, to obtain the exact number of solutions to  \eqref{1.1}, we need to impose some assumption on the domain $\Omega$.
The following one will be used later.

\smallskip

\noindent\textbf{Assumption A:}\emph{ The problem
\begin{equation}\label{4-18-21}
\begin{cases}
-\Delta u= u^{\frac{N+2}{N-2}},~u>0,  &{\text{in}~\Omega},\\
u=0, &{\text{on}~\partial \Omega},
\end{cases}
\end{equation}
has no solutions.}

\smallskip

It follows from \cite{Cerqueti,Li} and Theorem~\ref{th1-1}  that all blow-up points of \eqref{1.1} are simple and isolated.
Also from the well-known  results in \cite{Bahri2}, we find that  the number of the  blow-up points to \eqref{1.1} are finite. Now we denote  the largest number of blow-up points by $k_0$ and define
$$T_k= \Big\{(a^k,\Lambda^k)=(a_1,\cdots,a_k,\lambda_1,\cdots,\lambda_k),
\nabla_x \Psi_k(a^k,\Lambda^k)=0,~\nabla_\lambda \Psi_k(a^k,\Lambda^k)=0\Big\}.$$
Then  the following result confirms the number of solutions to problem \eqref{1.1}.
\begin{Thm}\label{th1.2}
Let $N\geq 6$. For any integer $k\in [1, k_0]$, suppose that $M_k(a^k)$ is a positive matrix,
$(a^k,\Lambda^k)$ is a nondegenerate critical point of $\Psi_{k}$
for any $(a^k,\Lambda^k)\in T_k$  and the domain $\Omega$ satisfies \textbf{Assumption A}. Then for $\varepsilon>0$ sufficiently small,
$$
\mbox{the number of solutions to \eqref{1.1}}~~=\displaystyle\sum^{k_0}_{k=1}\sharp T_k,$$
where $\sharp T_k$ is the number of the elements in the set $T_k$.\end{Thm}

\begin{Rem}
From the above statements after Theorem \ref{th1.1}(see also \cite{Bartsch1, Musso1}), we know that there are some non-convex domains such that $\Psi_{k}$ admits some critical points and all critical points of $\Psi_{k}$ are non-degenerate.   A special example on which the function $\Psi_{k}$ has at least two critical points is as follows. Let $\Omega_0=\bigcup^{k+1}_{i=1} \Omega_i$, where $\Omega_1,\cdots,\Omega_{k+1}$ are $k+1$ smooth bounded domains such that $dist\{\Omega_i,\Omega_j\}$ is large if $i\neq j$,
then  the function $\Psi_{k}$ has two strict minimum points in $\Omega_1\times \cdots \times \Omega_k\times(\R^+)^k$ and $\Omega_1\times \cdots \times \Omega_{k-1}\times \Omega_{k+1}\times(\R^+)^k$ correspondingly.
Moreover, assume that
$$\Omega_i\subset \big\{(x_1,x')\in \R\times \R^{N-1}|~a_i\leq x_1\leq b_i\big\}, ~\mbox{with}~ b_i<a_{i+1},~i=1,\cdots,k+1.$$
 For any $\delta>0$, let $C_\delta=\big\{(x_1,x')\in \R\times \R^{N-1}|~x_1\in (a_1,b_{k+1}),|x'|\leq \delta\big\}$
 and $\Omega_\delta$ be a smooth connected domain such that $\Omega_0\subset \Omega_\delta\subset \Omega_0\bigcup C_\delta$. Then if $\delta$ is small enough, the function $\Psi_{k}$ has at least two strict minimum points on $(\Omega_\delta)^k\times (\R^+)^k$, which are non-degenerate.
Hence problem
\eqref{1.1} admits at least two solutions concentrated at $k$ points for above domain $\Omega_\delta$.
And Theorem \ref{th1.1} gives us the description on the exact number of solutions concentrated at $k$ points for above domain $\Omega_\delta$.

On the other hand, it is known in \cite{Pohozaev}  that
\textup{Assumption A} is satisfied for a star-shaped domain.
And from \cite{Carpio}, we can also find some non star-shaped domains on which \textup{Assumption A} holds. However whether there exists a non-convex domain such that \textup{Assumption A} holds and the function $\Psi_{k}$ admits  non-degenerate critical points simultaneously seems to be interesting and difficult. Since the function $\Psi_{k}$ will depend on Green's function on $\Omega$ and
we know little information on Green's function. The properties of the critical points of $\Psi_{k}$
will be a substantive and important project.
A known example concerning that $\Psi_{k}$ admits some non-degenerate critical points
is above domain constructed in  \cite{Musso1}. But it seems to be not easy to determine
whether \textup{Assumption A} holds in this case.

Here we would like to point out that without \textup{Assumption A}, we can relax the result in Theorem \ref{th1.2} into
``the number of concentrated solutions to \eqref{1.1} $= \sum^{k_0}_{k=1}\sharp T_k$".

\end{Rem}

Whether or not theorems~\ref{th1.1} and \ref{th1.2} are true for $N=5$ are not clear due to our methods, which can be found in Remark \ref{Rem-luo1}  below for more details.
To prove our main results, the crucial step is to prove a local uniqueness result of blow-up solutions. To this end,
a widely used method  is
to reduce into finite dimensional problems and count
the local degree. We refer to \cite{Cao1,Glangetas} for examples. However, for the multi-peak solution of \eqref{1.1}, it is extremely complicated  to calculate the corresponding degree.
Here inspired by \cite{Deng,GPY},  our proofs mainly depend on the local Pohozaev type identities:
\begin{equation}\label{clp-1}
-\int_{\partial \Omega'}\frac{\partial u_\varepsilon}{\partial \nu}\frac{\partial u_\varepsilon}{\partial x_i}
+\frac{1}{2}\int_{\partial \Omega'}|\nabla u_\varepsilon|^2\nu_i=\frac{N-2}{2N}\int_{\partial \Omega'}  u_\varepsilon^{\frac{2N}{N-2}}\nu_i+ \frac{\varepsilon}{2}\int_{\partial \Omega'} u_\varepsilon^2\nu_i,
\end{equation}
and
\begin{equation}\label{clp-10}
\begin{split}
-&\int_{\partial \Omega'}\frac{\partial u_{\varepsilon}}{\partial\nu}
\big\langle x-x_{j,\varepsilon},\nabla u_{\varepsilon}\big\rangle
+\frac{1}{2}\int_{\partial \Omega'}
|\nabla  u_{\varepsilon}|^2
\big\langle x-x_{j,\varepsilon},\nu\big\rangle
+\frac{2-N}{2}\int_{\partial \Omega'}\frac{\partial u_{\varepsilon}}{\partial\nu}
  u_{\varepsilon} \\&=
  \frac{N-2}{2N}\int_{\partial \Omega'} u^{\frac{2N}{N-2}}_\varepsilon \big\langle x-x_{j,\varepsilon},\nu\big\rangle +\frac{\varepsilon}{2
  }\int_{\partial \Omega'}u_{\varepsilon}^{2}\big\langle x-x_{j,\varepsilon},\nu\big\rangle-\varepsilon\int_{ \Omega'} u^2_\varepsilon,
\end{split}
\end{equation}
where $\Omega'\subset \Omega$ is a smooth domain and $\nu(x)=\big(\nu_{1}(x),\cdots,\nu_N(x)\big)$ is the outward unit normal of $\partial \Omega'$. The local Pohozaev identities \eqref{clp-1} and \eqref{clp-10} can be deduced by  multiplying $\frac{\partial u_\varepsilon}{\partial x_i}$  and $\langle x-x_{j,\varepsilon},\nabla u_{\varepsilon}\rangle$ on both sides of \eqref{1.1} and integrating on $\Omega'$ respectively.
With the absence of potential function in \eqref{1.1},
 only surface integrals appear in the local  Pohozaev identities \eqref{clp-1} and \eqref{clp-10}. So we need to
study carefully each surface integral to determine which one dominates all the others.
The concentrated points of \eqref{1.1} depend on Green's function of $\Omega$, which causes new difficulties in the estimates of each term in local Pohozaev identities. Here inspired by  \cite{Cao}, we establish some new entire estimates to overcome these difficulties caused by Green's function. Last but not least,  since any solution of \eqref{1.1}  with \eqref{4-6-1} decays algebraically, we need to estimate the order of each terms in the local Pohozaev identities precisely. Here we also point out that
the interaction between the bumps
must be taken into careful consideration.

This paper is organized as follows.  In Section \ref{s2}, we establish some basic estimates of the solutions with concentration and give the proof of Theorem \ref{th1-1}.
In Section \ref{s3}, we estimate the regularization of difference between  two solutions. Then combining these calculations and the local Pohozaev identities, we prove
Theorem \ref{th1.1} and Theorem \ref{th1.2} in Section \ref{s4}. In Section \ref{s5}, we give the proofs of some crucial estimates involving the Green's function.
In order that we can give a clear line of our framework, we list some basic  estimates
and calculations in  Appendix \ref{App-A}.

Throughout our paper, we use the same $C$ to denote various generic positive constants independent of $\varepsilon$. We will use $\partial$ or $\nabla$ to denote the partial derivative for any function $h(y,x)$ with respect to $y$, while we will use $D$ to denote the partial derivative for any function $h(y,x)$ with respect to $x$.

\section{Some estimates on blow-up solutions and Proof of Theorem \ref{th1-1}}\label{s2}
\setcounter{equation}{0}
In this section, we obtain some basic estimates for solutions of  \eqref{1.1} satisfying
\eqref{4-6-1}. These estimates are crucial for discussions in next sections. We start with
the following decomposition result concerning with solutions of  \eqref{1.1}.
\begin{Prop}
Let $N\geq 5$. Suppose that  $u_{\varepsilon}(x)$ is a solution of \eqref{1.1} satisfying
\eqref{4-6-1}. Then  $u_{\varepsilon}$ can be written as
\begin{equation}\label{4-18-11}
u_{\varepsilon}=\sum^k_{j=1} PU_{x_{j,\varepsilon}, \lambda_{j,\varepsilon}}+w_{\varepsilon},
\end{equation}
satisfying, for $j=1,\cdots,k$,
 $\lambda_{j,\varepsilon}=
\Big(u_\varepsilon(x_{j,\varepsilon})\Big)^{\frac{2}{N-2}}$,
\begin{equation}\label{4-18-12}
 x_{j,\varepsilon}\rightarrow a_j, ~\lambda_{j,\varepsilon}\rightarrow +\infty,~ \|w_{\varepsilon}\|=o(1)~\mbox{and}~w_\varepsilon\in \bigcap^k_{j=1}E_{x_{j,\varepsilon},\lambda_{j,\varepsilon}}.
\end{equation}
\end{Prop}
\begin{proof}
Since $u_{\varepsilon}(x)$ is a solution of \eqref{1.1} satisfying
\eqref{4-6-1}, we find that $u_{\varepsilon}(x)$ blows up at $a_1,\cdots,a_{k}$. Then there exist $x_{j,\varepsilon}\in \Omega$  for $j=1,\cdots,k$  satisfying
$$x_{j,\varepsilon}\rightarrow a_j~\mbox{and}~u_\varepsilon(x_{j,\varepsilon})\rightarrow +\infty.$$
Let $v_{1,\varepsilon}=\lambda_{1,\varepsilon}^{-(N-2)/2}
u_{\varepsilon}\big(\frac{x}{\lambda_{1,\varepsilon}}+x_{1,\varepsilon}\big)$,
then
\begin{equation*}
-\Delta v_{1,\varepsilon} =v_{1,\varepsilon}^{2^*-1}
+\frac{\varepsilon}{\lambda_{1,\varepsilon}^2}v_{1,\varepsilon},
~\mbox{in}~\R^N.
\end{equation*}
For any fixed small $d$, $\displaystyle\max_{B_{d\lambda_{1,\varepsilon}}
(0)}v_{1,\varepsilon}=1,$
which means that
\begin{equation*}
u_{\varepsilon}=
PU_{x_{1,\varepsilon}, \lambda_{1,\varepsilon}}+u_{1,\varepsilon},~\mbox{with}~\int_{B_d(x_{1,\varepsilon})}\big[|\nabla u_{1,\varepsilon}|^2+u^2_{1,\varepsilon}\big]=o(1).
\end{equation*}
Repeating the above process and setting  $
w_{\varepsilon}(x):=u_{\varepsilon}-\displaystyle\sum^k_{j=1} PU_{x_{j,\varepsilon}, \lambda_{j,\varepsilon}}$,
 we get
 $$\int_{\bigcup^k_{j=1}
B_d(x_{j,\varepsilon})}\big(|\nabla w_{\varepsilon}|^2+w_{\varepsilon}^2\big)=o(1).
$$
This and \eqref{4-6-1} imply $\|w_{\varepsilon}\|=o(1)$. Then we find
\begin{equation*}
 \Big\langle \frac{\partial PU_{x_{j,\varepsilon},\lambda_{j,\varepsilon}}}{\partial \lambda},w_{\varepsilon}\Big\rangle=\Big\langle \frac{\partial PU_{x_{j,\varepsilon},\lambda_{j,\varepsilon}}}{\partial x_i},w_{\varepsilon}\Big\rangle=o(1).
\end{equation*}
Now we can move $x_{j,\varepsilon}$ a bit(still denoted by $x_{j,\varepsilon}$), so that the error term $w_\varepsilon\in \displaystyle\bigcap^k_{j=1}E_{x_{j,\varepsilon},\lambda_{j,\varepsilon}}$.
\end{proof}

\begin{Prop}\label{prop2.1}
Let $u_{\varepsilon}$ be  a  solution of \eqref{1.1} with \eqref{4-6-1}, then for any small fixed  $d>0$, it holds
\begin{equation}\label{clp-2}
u_{\varepsilon}(x)=A\Big(\sum^k_{j=1}\frac{G(x_{j,\varepsilon},x)}{
(\lambda_{j,\varepsilon})^{(N-2)/2}} \Big)
+O\Big(\frac{1}{\lambda_\varepsilon^{(N+2)/2}}
+\frac{\varepsilon}{\lambda_\varepsilon^{(N-2)/2}}\Big)
+o\Big(\frac{\varepsilon}{\lambda_\varepsilon^{2}}\Big),~\mbox{in}~  C^1
\Big(\Omega\backslash\bigcup^k_{j=1}B_{2d}(x_{j,\varepsilon})\Big),
\end{equation}
where $A$ is the constant in \eqref{a1} and $
  \lambda_\varepsilon:
  =\min\big\{\lambda_{1,\varepsilon},\cdots,\lambda_{k,\varepsilon}\big\}$.
\end{Prop}
\begin{proof}
First  for $x\in \Omega\backslash\displaystyle\bigcup^k_{j=1}B_{2d}(x_{j,\varepsilon})$, we have
\begin{equation}\label{aaaat3}
\begin{split}
u_\varepsilon(x)=& \int_{\Omega}G(y,x)
\big(u_{\varepsilon}^{\frac{N+2}{N-2}}(y)+\varepsilon u_{\varepsilon} (y)\big) dy\\=&
\sum^k_{j=1}\int_{B_d(x_{j,\varepsilon})}G(y,x)
u_{\varepsilon}^{\frac{N+2}{N-2}}(y)dy+\int_{\Omega\backslash\bigcup^k_{j=1}B_{d}(x_{j,\varepsilon})}G(y,x)
u_{\varepsilon}^{\frac{N+2}{N-2}}(y)dy\\&
+ \varepsilon \sum^k_{j=1}\int_{B_d(x_{j,\varepsilon})}G(y,x)u_{\varepsilon} (y)dy+\varepsilon \int_{\Omega\backslash\bigcup^k_{j=1}B_{d}(x_{j,\varepsilon})}G(y,x)u_{\varepsilon} (y)dy.
\end{split}
\end{equation}
And by Taylor's expansion, we know
\begin{equation}\label{tian-1}
\begin{split}
\int_{B_d(x_{j,\varepsilon})} &G(y,x) u_{\varepsilon}^{\frac{N+2}{N-2}}(y) dy\\=&
  G(x_{j,\varepsilon},x)\int_{B_d(x_{j,\varepsilon})}
  u_{\varepsilon}^{\frac{N+2}{N-2}}+\sum^N_{i=1}D_{x_i} G(x_{j,\varepsilon},x)
\int_{B_d(x_{j,\varepsilon})}\big(y_i-x_{j,\varepsilon,i}\big)  u_{\varepsilon}^{\frac{N+2}{N-2}}(y) dy\\&
+\sum^N_{i=1}\sum^N_{m=1}D^2_{x_ix_m} G(x_{j,\varepsilon},x) \int_{B_d(x_{j,\varepsilon})}\big(y_i-x_{j,\varepsilon,i}\big)
\big(y_m-x_{j,\varepsilon,m}\big)  u_{\varepsilon}^{\frac{N+2}{N-2}}(y) dy\\&
+O\Big(\int_{B_d(x_{j,\varepsilon})} |y-x_{j,\varepsilon}|^3 u_{\varepsilon}^{\frac{N+2}{N-2}}(y) dy\Big).
\end{split}
\end{equation}
Also from the symmetry and  the fact that $\displaystyle\sum^N_{i=1}D^2_{x_ix_i} G(x_{j,\varepsilon},x) =0$ for $x\in \Omega\backslash B_d(x_{j,\varepsilon})$, we get
\begin{equation}\label{tian-3}
\begin{split}
\sum^N_{i=1}\sum^N_{m=1}D^2_{x_ix_m} G(x_{j,\varepsilon},x) \int_{B_d(x_{j,\varepsilon})}\big(y_i-x_{j,\varepsilon,i}\big)
\big(y_m-x_{j,\varepsilon,m}\big)  U^{\frac{N+2}{N-2}}_{x_{j,\varepsilon}, \lambda_{j,\varepsilon}}
 =0.
\end{split}
\end{equation}
Next for $x\in \Omega\backslash\displaystyle\bigcup^k_{j=1}B_{2d}(x_{j,\varepsilon})$, from \eqref{A--1}--\eqref{lp21}, it holds
\begin{equation}\label{at6}
\begin{split}
 \varepsilon\int_{B_d(x_{j,\varepsilon})}G(y,x) u_{\varepsilon} (y) dy &=O\Big(\varepsilon
 \int_{B_d(x_{j,\varepsilon})} u_{\varepsilon} (y) dy\Big)=O\Big(\frac{\varepsilon}{\lambda_\varepsilon^{(N-2)/2}}
 \Big)+o\Big(\frac{\varepsilon}{\lambda_\varepsilon^{2}}\Big).
\end{split}
\end{equation}
Then  \eqref{aaaat3}--\eqref{at6} and \eqref{at3}--\eqref{8-17-6} imply
\begin{equation*}
u_{\varepsilon}(x)=A\Big(\sum^k_{j=1}\frac{G(x_{j,\varepsilon},x)}{
(\lambda_{j,\varepsilon})^{(N-2)/2}} \Big)
+O\Big(\frac{1}{\lambda_\varepsilon^{(N+2)/2}}
+\frac{\varepsilon}{\lambda_\varepsilon^{(N-2)/2}}\Big)+o\Big(\frac{\varepsilon}{\lambda_\varepsilon^{2}}\Big),~\mbox{in}~
\Omega\backslash\bigcup^k_{j=1}B_{2d}(x_{j,\varepsilon}).
\end{equation*}

On the other hand,  from  \eqref{A--1}, for $x\in \Omega\backslash\displaystyle\bigcup^k_{j=1}B_{2d}(x_{j,\varepsilon})$, we have
\begin{equation}\label{at-3}
\begin{split}
\frac{\partial u_\varepsilon(x)}{\partial x_i}=& \int_{\Omega}D_{x_i}G(y,x)
\Big(u_{\varepsilon}^{\frac{N+2}{N-2}}(y)+\varepsilon u_{\varepsilon} (y)\Big) dy\\=&
\sum^k_{j=1}\int_{B_d(x_{j,\varepsilon})}D_{x_i}G(y,x)
\Big(u_{\varepsilon}^{\frac{N+2}{N-2}}(y)+\varepsilon u_{\varepsilon} (y)\Big) dy+O\Big(\frac{1}{\lambda^{(N+2)/2}_{\varepsilon}}\Big).
\end{split}
\end{equation}
Similar to the above estimates, for $x\in \Omega\backslash\displaystyle\bigcup^k_{j=1}B_{2d}(x_{j,\varepsilon})$
and $j=1,\cdots,k$, we can prove
\begin{equation}\label{at--3}
\begin{split}
 \int_{B_d(x_{j,\varepsilon})}D_{x_i}G(y,x)
\Big(u_{\varepsilon}^{\frac{N+2}{N-2}}(y)+\varepsilon u_{\varepsilon} (y)\Big) dy=
 \frac{A}{
(\lambda_{j,\varepsilon})^{(N-2)/2}} D_{x_i} G(x_{j,\varepsilon},x)+O\Big(\frac{1}{\lambda^{(N+2)/2}_{\varepsilon}}\Big).
\end{split}
\end{equation}
Then \eqref{at-3} and \eqref{at--3} imply
\begin{equation*}
\frac{\partial u_\varepsilon(x)}{\partial x_i}=A\Big(\sum^k_{j=1}\frac{D_{x_i}G(x_{j,\varepsilon},x)}{
(\lambda_{j,\varepsilon})^{(N-2)/2}} \Big)
+O\Big(\frac{1}{\lambda_\varepsilon^{(N+2)/2}}
+\frac{\varepsilon}{\lambda_\varepsilon^{(N-2)/2}}\Big)+o\Big(\frac{\varepsilon}{\lambda_\varepsilon^{2}}\Big),~\mbox{in}~
\Omega\backslash\bigcup^k_{j=1}B_{2d}(x_{j,\varepsilon}).
\end{equation*}
\end{proof}
\begin{Prop}\label{Prop-4-26-1}
Let $u_{\varepsilon}$ be  a  solution of \eqref{1.1} with \eqref{4-6-1},  then  $M_k(a^k)$ is a non-negative matrix.
Moreover if  $M_k(a^k)$ is a positive matrix, it holds
\begin{equation}\label{4-16-1}
0<C_1\leq \varepsilon ^{\frac{1}{N-4}}\lambda_{j,\varepsilon}\leq C_2<+\infty,~\mbox{for}~j=1,\cdots,k,
\end{equation}
and
\begin{equation}\label{4-18-13}
\nabla_\lambda \Psi_k(a^k,\Lambda^k)=0,~\mbox{with}~a^k=(a_1,\cdots,a_{k})~\mbox{and}~ \Lambda^k=(\lambda_{1},\cdots,\lambda_{k}).
\end{equation}
Here we denote (by subsequence)
$\lambda_j:=\displaystyle\lim_{\varepsilon\rightarrow 0}\big(\varepsilon ^{\frac{1}{N-4}}\lambda_{j,\varepsilon}\big)^{-1}$ for $j=1,\cdots,k$.
\end{Prop}
\begin{proof}
We define the following quadratic form
\begin{equation*}
\begin{split}
P(u,v)=&- \theta\int_{\partial B_\theta(x_{j,\varepsilon})}
\big\langle \nabla u ,\nu\big\rangle
\big\langle \nabla v,\nu\big\rangle
+\frac{\theta}{2}\int_{\partial B_\theta(x_{j,\varepsilon})}
\big\langle \nabla u , \nabla v \big\rangle
\\&
+\frac{2-N}{4}\int_{\partial B_\theta(x_{j,\varepsilon})}
\big\langle \nabla u ,  \nu \big\rangle v
+\frac{2-N}{4}\int_{\partial B_\theta(x_{j,\varepsilon})}
\big\langle \nabla v ,  \nu \big\rangle u.
\end{split}
\end{equation*}
Note that if $u$ and $v$ are harmonic in $ B_d(x_{j,\varepsilon})\backslash \{x_{j,\varepsilon}\}$, then $P(u,v)$ is independent of $\theta>0$.
Let  $\Omega'=B_{\theta}(x_{j,\varepsilon})$ in \eqref{clp-10}, then from  \eqref{clp-2} and  \eqref{lp21}, we have
\begin{equation}\label{lt1}
\begin{split}
 \sum^k_{l=1}\sum^k_{m=1}\frac{P\big(G(x_{m,\varepsilon},x),G(x_{l,\varepsilon},x)\big)}
{\lambda^{(N-2)/2}_{m,\varepsilon}\lambda^{(N-2)/2}_{l,\varepsilon}} =-\frac{ B\varepsilon }{A^2\lambda^{2}_{j,\varepsilon}}+
O\Big(\frac{1}{\lambda^N_{\varepsilon}}+
\frac{\varepsilon}{\lambda^{N-2}_{\varepsilon}}\Big)
+o\Big(\frac{\varepsilon}{\lambda^{(N+2)/2}_{\varepsilon}}+\frac{\varepsilon^2}{\lambda^{4}_{\varepsilon}}\Big),
\end{split}
\end{equation}
where $A,B$ are the constants  in \eqref{a1}.

Next we have the following estimate for which the  proof  is left in Section 5:
\begin{equation}\label{luo2}
P\Big(G(x_{m,\varepsilon},x), G(x_{l,\varepsilon},x)\Big)=
\begin{cases}
-\frac{(N-2)R(x_{j,\varepsilon})}{2},~&\mbox{for}~l,m=j.\\[1mm]
\frac{(N-2)G(x_{j,\varepsilon},x_{l,\varepsilon})}{4},~
&\mbox{for}~m=j, ~l\neq j. \\[1mm]
\frac{(N-2)G(x_{j,\varepsilon},x_{m,\varepsilon})}{4},~&\mbox{for}~m\neq j, ~l=j.\\[1mm]
0, ~&\mbox{for}~l,m\neq j.
\end{cases}
\end{equation}
Then \eqref{lt1} and \eqref{luo2} imply
\begin{equation}\label{cc5}
 \frac{ R(x_{j,\varepsilon})}{\lambda^{N-2}_{j,\varepsilon}}
-\sum^k_{l\neq j}\frac{G(x_{j,\varepsilon},x_{l,\varepsilon})}{\lambda^{(N-2)/2}_{j,\varepsilon}\lambda^{(N-2)/2}_{l,\varepsilon}}
 =\frac{2B\varepsilon}{A^2(N-2)\lambda^{2}_{j,\varepsilon}}+
O\Big(\frac{1}{\lambda^N_{\varepsilon}}+
\frac{\varepsilon}{\lambda^{N-2}_{\varepsilon}}\Big)
+o\Big(\frac{\varepsilon}{\lambda^{(N+2)/2}_{\varepsilon}}+\frac{\varepsilon^2}{\lambda^{4}_{\varepsilon}}\Big).
\end{equation}
Let $\Lambda_{j,\varepsilon}:=\Big(\varepsilon ^{\frac{1}{N-4}}\lambda_{j,\varepsilon}\Big)^{-1}$. From
Corollary 3.7 in \cite{Cerqueti}, we find $$\Lambda_{j,\varepsilon}\geq C>0,~\mbox{for}~j=1,\cdots,k.$$
Now we define
$\Lambda^k_{\varepsilon}=\max\big\{\Lambda_{j,\varepsilon}, j=1,\cdots,k\big\}$. Then
 \begin{equation}\label{lt3}
 \begin{split}
 \Lambda^{N-2}_{j,\varepsilon} R(x_{j,\varepsilon})
-\sum^k_{l\neq j}\Lambda^{(N-2)/2}_{j,\varepsilon}\Lambda^{(N-2)/2}_{l,\varepsilon}G(x_{j,\varepsilon},x_{l,\varepsilon})
=\frac{2B}{A^2(N-2)}\Lambda^{2}_{j,\varepsilon}+o\Big((\Lambda_{\varepsilon}^k)^{N-2}\Big).
\end{split}
\end{equation}
Since $\frac{1}{k}\displaystyle\sum^k_{l=1}\Lambda^{N-2}_{l,\varepsilon}\leq (\Lambda^k_{\varepsilon})^{N-2}\leq \displaystyle\sum^k_{l=1}\Lambda^{N-2}_{l,\varepsilon}$, \eqref{lt3} gives us
\begin{equation}\label{lla}
 \Big(M_k(x_\varepsilon)+o(1)\Big)
\vec{\mu}_{k,\varepsilon}^T=\frac{2B}{A^2(N-2)} \big(\Lambda^{\frac{6-N}{2}}_{1,\varepsilon},\cdots,
\Lambda^{\frac{6-N}{2}}_{k,\varepsilon}\big)^T,
\end{equation}
where $\vec{\mu}_{k,\varepsilon}=\big(\Lambda_{1,\varepsilon}^{(N-2)/2},\cdots,\Lambda_{k,\varepsilon}^{(N-2)/2}
\big)$ and
$x_{\varepsilon}=(x_{1,\varepsilon},\cdots,x_{k,\varepsilon})$.
 Now we  recall that the first eigenvector of a symmetric matrix may be chosen with all its components strictly positive (see also Appendix A in \cite{Bahri2}). So if $\rho(a^k)$ is
   the first eigenvalue of $M_k(a^k)$, then there exists a first eigenvector $\overrightarrow{\chi}(a^k)$ of $M_k(a^k)$  such that all its components are strictly positive.
   Then $\eqref{lla}$ gives us that
  \begin{equation*}
 \overrightarrow{\chi}(a^k)\Big(M_k(x_\varepsilon)+o(1)\Big)
\vec{\mu}_{k,\varepsilon}^T=\frac{2B}{A^2(N-2)} \overrightarrow{\chi}(a^k) \big(\Lambda^{\frac{6-N}{2}}_{1,\varepsilon},\cdots,
\Lambda^{\frac{6-N}{2}}_{k,\varepsilon}\big)^T> 0.
\end{equation*}
Also we know
  \begin{equation*}
 \rho(a^k) \Big(\overrightarrow{\chi}(a^k)
\vec{\mu}_{k,\varepsilon}^T\Big)= \overrightarrow{\chi}(a^k) M_k(a^k)
\vec{\mu}_{k,\varepsilon}^T~\mbox{and}~\overrightarrow{\chi}(a^k)
\vec{\mu}_{k,\varepsilon}^T>0.
\end{equation*}
Then these mean that  $ \rho(a^k)\geq 0$ and $M_k(a^k)$ is a non-negative matrix.
Moreover, if $M_k(a^k)$ is a positive matrix, we find  $\Lambda_{j,\varepsilon}$ is bounded for $j=1,\cdots,k$. And then these imply \eqref{4-16-1}.
Moreover letting $\varepsilon \rightarrow 0$ in \eqref{lt3}, we find \eqref{4-18-13}.
\end{proof}

\begin{Prop}
Under the conditions in Proposition \ref{Prop-4-26-1}, it holds
\begin{equation}\label{aclp-2}
u_{\varepsilon}(x)=A\Big(\sum^k_{j=1}\frac{G(x_{j,\varepsilon},x)}{
(\lambda_{j,\varepsilon})^{(N-2)/2}} \Big)
+\begin{cases}
 O\Big(\frac{1}{\lambda_\varepsilon^{5/2}}\Big),~&N=5,\\[2mm]
 O\Big(\frac{1}{\lambda_\varepsilon^{(N+2)/2}}\Big),~&N\geq 6,
\end{cases}
~\mbox{in}~  C^1
\Big(\Omega\backslash\bigcup^k_{j=1}B_{2d}(x_{j,\varepsilon})\Big).
\end{equation}
\end{Prop}
\begin{proof}
The estimate \eqref{aclp-2} can be deduced by \eqref{clp-2} and \eqref{4-16-1}.
\end{proof}
\begin{Prop}\label{prop--1}
Let $u_{\varepsilon}$ be  a  solution of \eqref{1.1} with \eqref{4-6-1} and  $M_k(a^k)$ be a positive matrix. Then
\begin{equation}\label{4-18-1}
\nabla_x \Psi_k(a^k,\Lambda^k)=0,~\mbox{with}~
 a^k=(a_1,\cdots,a_{k})~ \mbox{and}~\Lambda^k=(\lambda_{1},\cdots,\lambda_{k}),
 \end{equation}
 where $\lambda_j:=\displaystyle\lim_{\varepsilon\rightarrow 0}\big(\varepsilon ^{\frac{1}{N-4}}\lambda_{j,\varepsilon}\big)^{-1}$ for $j=1,\cdots,k$.
Moreover if $(a^k,\Lambda^k)$ is a nondegenerate critical point of $\Psi_{k}$, then
for $j=1,\cdots,k$, it follows
\begin{equation}\label{clp-03}
\big|x_{j,\varepsilon}-a_j\big|=
 \begin{cases}
 O\Big(\frac{1}{\lambda_{\varepsilon}}\Big),&~\mbox{if}~N=5,\\[1.5mm]
 O\Big(\frac{1}{\lambda^2_{\varepsilon}}\Big),&~\mbox{if}~N\geq 6,
 \end{cases}~\mbox{and}~~~~\,\,\,
\big|\lambda_j- \big(\varepsilon ^{\frac{1}{N-4}}\lambda_{j,\varepsilon}\big)^{-1}\big|=
 \begin{cases}
 O\Big(\frac{1}{\lambda_{\varepsilon}}\Big),&~\mbox{if}~N=5,\\[1.5mm]
 O\Big(\frac{1}{\lambda^2_{\varepsilon}}\Big),&~\mbox{if}~N\geq 6.
 \end{cases}
\end{equation}
\end{Prop}
\begin{proof}
First, we define the following quadratic form
\begin{equation*}
Q(u,v)=-\int_{\partial B_\theta(x_{j,\varepsilon})}\frac{\partial v}{\partial \nu}\frac{\partial u}{\partial x_i}-
\int_{\partial B_\theta(x_{j,\varepsilon})}\frac{\partial u}{\partial \nu}\frac{\partial v}{\partial x_i}
+\int_{\partial B_\theta(x_{j,\varepsilon})}\big\langle \nabla u,\nabla v \big\rangle \nu_i.
\end{equation*}
Note that if $u$ and $v$ are harmonic in $ B_d(x_{j,\varepsilon})\backslash \{x_{j,\varepsilon}\}$, then $Q(u,v)$ is independent of $\theta\in (0,d]$.
Letting  $\Omega'=B_{\theta}(x_{j,\varepsilon})$ in \eqref{clp-1} and using  \eqref{aclp-2},  we have
\begin{equation}\label{luo-tian1}
 \sum^k_{l=1}\sum^k_{m=1}\frac{Q\big(G(x_{m,\varepsilon},x),G(x_{l,\varepsilon},x)\big)}{
 \lambda_{m,\varepsilon} ^{(N-2)/2}
 \lambda_{l,\varepsilon} ^{(N-2)/2}}=
 \begin{cases}
 O\Big(\frac{1}{\lambda^4_{\varepsilon}}\Big),&~\mbox{if}~N=5,\\[1.5mm]
 O\Big(\frac{1}{\lambda^N_{\varepsilon}}\Big),&~\mbox{if}~N\geq 6.
 \end{cases}
\end{equation}
Next we have the following estimate for which  the proof  is left in Section 5:
\begin{equation}\label{luo1}
Q\Big(G(x_{m,\varepsilon},x),G(x_{l,\varepsilon},x)\Big)=
\begin{cases}
-\frac{\partial R(x_{j,\varepsilon})} {\partial x_i},~&\mbox{for}~l,m=j,\\[1mm]
D_{x_i}G(x_{m,\varepsilon},x_{j,\varepsilon}),
 ~&\mbox{for}~m\neq j,~l=j,\\[1mm]
D_{x_i}G(x_{l,\varepsilon},x_{j,\varepsilon}),
~&\mbox{for}~m=j,~l\neq j,\\[1mm]
0,~&\mbox{for}~l,m\neq j.
\end{cases}
\end{equation}
Then \eqref{luo-tian1} and \eqref{luo1} imply
\begin{equation}\label{luo-tian}
\frac{1}{2\lambda_{j,\varepsilon}^{N-2}}\frac{\partial R(x_{j,\varepsilon})}{
\partial{x_i}}-\displaystyle\sum^k_{l=1,l\neq j}
\frac{1}{\lambda^{(N-2)/2}_{j,\varepsilon}\lambda^{(N-2)/2}_{l,\varepsilon}}
\frac{\partial G(x_{j,\varepsilon},x_{l,\varepsilon})}{\partial x_i}=
 \begin{cases}
 O\Big(\frac{1}{\lambda^4_{\varepsilon}}\Big),&~\mbox{if}~N=5,\\[1.5mm]
 O\Big(\frac{1}{\lambda^N_{\varepsilon}}\Big),&~\mbox{if}~N\geq 6.
 \end{cases}
\end{equation}
Let $\Lambda_{j,\varepsilon}:=\Big(\varepsilon ^{\frac{1}{N-4}}\lambda_{j,\varepsilon}\Big)^{-1}$,
 we can  rewrite \eqref{luo-tian}  as follows:
\begin{equation}\label{clp-01}
 \frac{\Lambda_{j,\varepsilon}^{N-2}}{2}\frac{\partial R(x_{j,\varepsilon})}{
\partial{x_i}}-\displaystyle\sum^k_{l=1,l\neq j}
\Lambda_{j,\varepsilon}^{(N-2)/2}\Lambda_{l,\varepsilon}^{(N-2)/2}
\frac{\partial G(x_{j,\varepsilon},x_{l,\varepsilon})}{\partial x_i}=
 \begin{cases}
 O\Big(\frac{1}{\lambda_{\varepsilon}}\Big),&~\mbox{if}~N=5,\\[1.5mm]
 O\Big(\frac{1}{\lambda^2_{\varepsilon}}\Big),&~\mbox{if}~N\geq 6.
 \end{cases}
\end{equation}
Then taking  $\varepsilon\rightarrow 0$ in \eqref{clp-01}, we find \eqref{4-18-1}.
Moreover by the assumption that $(a^k,\Lambda^k)$ is a nondegenerate critical point of $\Psi_{k}$,
 we get \eqref{clp-03} from \eqref{4-16-1}, \eqref{cc5} and \eqref{clp-01}.
\end{proof}

\begin{proof}[\textbf{Proof of Theorem \ref{th1-1}}]
Theorem \ref{th1-1} can be deduced by
\eqref{4-18-11}, \eqref{4-18-12}, \eqref{4-16-1}, \eqref{4-18-13} and
\eqref{4-18-1}.
\end{proof}

\section{Regularization and blow-up analysis}\label{s3}
\setcounter{equation}{0}

To estimate the number of concentrated solutions to \eqref{1.1}, we need first to obtain local uniqueness of such type of solutions. To this end, we need to estimate the difference between   two solutions concentrating
at the same points.

Let $u^{(1)}_{\varepsilon}(x)$, $u^{(2)}_{\varepsilon}(x)$ be two different solutions of $\eqref{1.1}$ satisfying \eqref{4-6-1}. Under the assumption  that  $M_k(a^k)$ is a positive matrix, we find from Theorem \ref{th1-1} that
$u^{(l)}_{\varepsilon}(x)$ can be written as
\begin{equation*}
u^{(l)}_{\varepsilon}=\sum^k_{j=1} PU_{x^{(l)}_{j,\varepsilon}, \lambda^{(l)}_{j,\varepsilon}}+w^{(l)}_{\varepsilon},
\end{equation*}
satisfying, for $j=1,\cdots,k$, $l=1,2$,
 $\lambda^{(l)}_{j,\varepsilon}=
\big(u_\varepsilon(x^{(l)}_{j,\varepsilon})\big)^{\frac{2}{N-2}}$,
\begin{equation*}
 x^{(l)}_{j,\varepsilon}\rightarrow a_j, ~\Big(\varepsilon ^{\frac{1}{N-4}} \lambda^{(l)}_{j,\varepsilon}\Big)^{-1}\rightarrow \lambda_j,~ \|w^{(l)}_{\varepsilon}\|=o(1)~\mbox{and}~w^{(l)}_\varepsilon\in \bigcap^k_{j=1}E_{x^{(l)}_{j,\varepsilon},\lambda^{(l)}_{j,\varepsilon}}.
\end{equation*}
Let $Q_\varepsilon$ be a quadratic form on $H^1_0(\Omega)$ given by
\begin{equation*}
 \big\langle Q_\varepsilon u,v\big\rangle=\big\langle u,v \big\rangle-\int_{\mathbb R^N}\Big[(2^*-1) \Big(\sum^k_{j=1}PU_{x^{(1)}_{j,\varepsilon},\lambda^{(1)}_{j,\varepsilon}}\Big)^{2^*-2}+\varepsilon\Big] uv,~~ \forall \; u,v \in  H^1_0(\Omega).
\end{equation*}
\begin{Prop}\label{Prop-Luo2}
For any $\varepsilon>0$ sufficiently small, there exists a constant $\rho>0$ such that
\begin{equation*}
\big\langle Q_\varepsilon v,v\big\rangle\geq \rho \|v \|^2, \quad \forall \;  v \in  \bigcap^k_{j=1}E_{x^{(1)}_{j,\varepsilon},\lambda^{(1)}_{j,\varepsilon}}.
\end{equation*}
 \end{Prop}
\begin{proof}
This is standard and  can be found in Lemma 1.7 of \cite{Musso1}. Also one can  refer to Proposition B.1
in \cite{Cao2} and Proposition 2.4.3 in \cite{Cao3}.
\end{proof}
\smallskip

Now we define  $\bar{\lambda}_\varepsilon:=
\min
\Big\{\lambda^{(1)}_{1,\varepsilon},\cdots,
\lambda^{(1)}_{k,\varepsilon},\lambda^{(2)}_{1,\varepsilon},\cdots,
\lambda^{(2)}_{k,\varepsilon}\Big\}$.
  \begin{Prop}\label{aaprop-A.2}
For $N\geq 6$, it holds
\begin{equation}\label{ddlp21}
\|w^{(1)}_{\varepsilon}-w^{(2)}_{\varepsilon}\|=
o\Big(\frac{1}{\bar{\lambda}_\varepsilon^{(N+2)/2}}\Big).
\end{equation}
\end{Prop}
\begin{proof}
First we define $\overline{w}_{\varepsilon}:=w^{(1)}_{\varepsilon}-w^{(2)}_{\varepsilon}$, then
 \begin{equation*}
Q_\varepsilon \overline{w}_{\varepsilon}
    =R_\varepsilon^{(1)}\big(w^{(1)}_{\varepsilon}\big)-R_\varepsilon^{(2)}\big(w^{(2)}_{\varepsilon}\big)+ l_\varepsilon,
 \end{equation*}
where
\begin{equation*}
\begin{split}
R^{(l)}_\varepsilon(w^{(l)}_\varepsilon)=&
\Big(\sum^k_{i=1}PU_{y^{(l)}_{i,\varepsilon},\lambda^{(l)}_{i,\varepsilon}}
+w^{(l)}_\varepsilon\Big)^{\frac{N+2}{N-2}}
-
\Big(\sum^k_{i=1}PU_{y^{(l)}_{i,\varepsilon},\lambda^{(l)}_{i,\varepsilon}}\Big)^{\frac{N+2}{N-2}}
-\big(\frac{N+2}{N-2}\big)
\Big(\sum^k_{i=1}PU_{y^{(l)}_{\varepsilon,i},\lambda^{(l)}_{\varepsilon,i}}\Big)^{\frac{4}{N-2}}
w^{(l)}_\varepsilon,
\end{split}\end{equation*}
and
 \begin{equation*}
l_\varepsilon= (2^*-1) \left[\Big(\sum^k_{j=1}PU_{x^{(1)}_{j,\varepsilon},\lambda^{(1)}_{j,\varepsilon}}\Big)^{2^*-2}-
\Big(\sum^k_{j=1}PU_{x^{(2)}_{j,\varepsilon},\lambda^{(2)}_{j,\varepsilon}}\Big)^{2^*-2}\right]
w^{(2)}_{\varepsilon}.
\end{equation*}
Now we write $\overline{w}_{\varepsilon}=\overline{w}_{\varepsilon,1}+\overline{w}_{\varepsilon,2}$ with
$\overline{w}_{\varepsilon,1}\in \displaystyle\bigcap^k_{j=1}E_{x^{(1)}_{j,\varepsilon},\lambda^{(1)}_{j,\varepsilon}}$ and $\overline{w}_{\varepsilon,2}\bot \displaystyle\bigcap^k_{j=1}E_{x^{(1)}_{j,\varepsilon},\lambda^{(1)}_{j,\varepsilon}}$.  Then
\begin{equation*}
\overline{w}_{\varepsilon,1}(x)=\overline{w}_{\varepsilon}(x) -\sum_{i=1}^k\Big(\alpha_{\varepsilon,i,0}\frac{\partial PU_{x^{(1)}_{i,\varepsilon},\lambda^{(1)}_{i,\varepsilon}}}
 {\partial \lambda}+\sum_{j=1}^N\alpha_{\varepsilon,i,j}
 \frac{\partial PU_{x^{(1)}_{i,\varepsilon},\lambda^{(1)}_{i,\varepsilon}}}
 {\partial x_j}\Big),
 \end{equation*}
for some constants $\alpha_{\varepsilon,i,j}$ with $i=1,\cdots,k$ and $j=0,\cdots,N$.
 Then
\[
\begin{split}
 \alpha_{\varepsilon,i,0}&\Big\|\frac{\partial PU_{x^{(1)}_{i,\varepsilon},\lambda^{(1)}_{i,\varepsilon}}}
 {\partial \lambda}\Big\|^2
    \\=& \Bigl\langle \overline{w}_{\varepsilon}(x),\frac{\partial PU_{x^{(1)}_{i,\varepsilon},\lambda^{(1)}_{i,\varepsilon}}}
 {\partial \lambda} \Bigr\rangle -\sum_{l=1,l\neq i}^k\alpha_{\varepsilon,l,0}\Big\langle \frac{\partial PU_{x^{(1)}_{i,\varepsilon},\lambda^{(1)}_{i,\varepsilon}}}
 {\partial \lambda},\frac{\partial PU_{x^{(1)}_{l,\varepsilon},\lambda^{(1)}_{l,\varepsilon}}}
 {\partial \lambda}\Big\rangle\\&+
 \sum_{l=1}^k\sum_{j=1}^N\alpha_{\varepsilon,l,j}\Big\langle \frac{\partial PU_{x^{(1)}_{i,\varepsilon},\lambda^{(1)}_{i,\varepsilon}}} {\partial \lambda},
 \frac{\partial PU_{x^{(1)}_{l,\varepsilon},\lambda^{(1)}_{l,\varepsilon}}}
 {\partial x_j}\Big\rangle  \\=&-\Bigl\langle w^{(2)}_{\varepsilon},
 \frac{\partial  PU_{x^{(1)}_{j,\varepsilon},\lambda^{(1)}_{j,\varepsilon}}}{\partial \lambda}-
 \frac{\partial  PU_{x^{(2)}_{j,\varepsilon},\lambda^{(2)}_{j,\varepsilon}}}{\partial \lambda}
    \Bigr\rangle+o\Big(\sum_{l=1,l\neq i}^k\frac{|\alpha_{\varepsilon,l,0}|}{\bar{\lambda}^2_\varepsilon}+ \sum_{l=1}^k\sum_{j=1}^N|\alpha_{\varepsilon,l,j}|\Big)\\=&
    O\Big( \big|x_{\varepsilon,i}^{(1)}-x^{(2)}_{\varepsilon,i}\big|+
\frac{\big|\lambda_{\varepsilon,i}^{(1)}-\lambda^{(2)}_{\varepsilon,i}\big|}{\bar{\lambda}^2_\varepsilon}
\Big)\|w^{(2)}_{\varepsilon}\|
    +o\Big(\sum_{l=1,l\neq i}^k\frac{|\alpha_{\varepsilon,l,0}|}{\bar{\lambda}^2_\varepsilon}+ \sum_{l=1}^k\sum_{j=1}^N|\alpha_{\varepsilon,l,j}|\Big).\end{split}
\]
Also from \eqref{clp-03}, we know $\big|x^{(1)}_{i,\varepsilon}-x^{(2)}_{i,\varepsilon}\big|=O\Big(\frac{1}{\bar \lambda^2_{\varepsilon}}\Big)$ and $\big|\lambda^{(1)}_{i,\varepsilon}-\lambda^{(2)}_{i,\varepsilon}\big|=O\big(1\big)$ for $N\geq 6$. Then
\begin{equation}\label{8-1-1}
\begin{split}
 \alpha_{\varepsilon,i,0}=
    O\Big(
\|w^{(2)}_{\varepsilon}\|\Big)
    +o\Big(\sum_{l=1,l\neq i}^k |\alpha_{\varepsilon,l,0}| + \sum_{l=1}^k\sum_{j=1}^N\bar{\lambda}^2_\varepsilon|\alpha_{\varepsilon,l,j}|\Big).\end{split}
\end{equation}
Similarly, we find
\begin{equation}\label{8-1-2}
\begin{split}
 \bar{\lambda}^2_\varepsilon|\alpha_{\varepsilon,i,j}|=
    O\Big(\|w^{(2)}_{\varepsilon}\|\Big)
    +o\Big(\sum_{i=1}^k|\alpha_{\varepsilon,i,0}| + \sum_{i=1}^k\sum_{m=1,m\neq j}^N \bar{\lambda}^2_\varepsilon|\alpha_{\varepsilon,i,m}|\Big).\end{split}
\end{equation}
Hence it follows from \eqref{8-1-1} and \eqref{8-1-2} that
\begin{equation}\label{8-1-5}
\sum_{i=1}^k|\alpha_{\varepsilon,i,0}| + \sum_{i=1}^k\sum_{j=1}^N \bar{\lambda}^2_\varepsilon|\alpha_{\varepsilon,i,j}|=O\Big(\|w^{(2)}_{\varepsilon}\|\Big).
\end{equation}
Then by the definition of $Q_{\varepsilon}$ and the fact $\overline{w}_{\varepsilon,1}\in \displaystyle\bigcap^k_{j=1}E_{x^{(1)}_{j,\varepsilon},\lambda^{(1)}_{j,\varepsilon}}$, we have
\begin{equation}\label{1-15-8}
\begin{split}
 &\left\langle Q_{\varepsilon}\,\Big(\alpha_{\varepsilon,i,0}\frac{\partial PU_{x^{(1)}_{i,\varepsilon},\lambda^{(1)}_{i,\varepsilon}}}
 {\partial \lambda}+\sum_{j=1}^N\alpha_{\varepsilon,i,j}
 \frac{\partial PU_{x^{(1)}_{i,\varepsilon},\lambda^{(1)}_{i,\varepsilon}}}
 {\partial x_j}\Big),\overline{w}_{\varepsilon,1}\right\rangle
 \\=&O\left(\int_{\Omega}\big|W(x)\big|\cdot \Big|\alpha_{\varepsilon,i,0}\frac{\partial PU_{x^{(1)}_{i,\varepsilon},\lambda^{(1)}_{i,\varepsilon}}}
 {\partial \lambda}+\sum_{j=1}^N\alpha_{\varepsilon,i,j}
 \frac{\partial PU_{x^{(1)}_{i,\varepsilon},\lambda^{(1)}_{i,\varepsilon}}}
 {\partial x_j}\Big|\cdot \big|\overline{w}_{\varepsilon,1}\big|\right)\\&+
 O\left(\varepsilon\int_{\Omega} \Big|\alpha_{\varepsilon,i,0}\frac{\partial PU_{x^{(1)}_{i,\varepsilon},\lambda^{(1)}_{i,\varepsilon}}}
 {\partial \lambda}+\sum_{j=1}^N\alpha_{\varepsilon,i,j}
 \frac{\partial PU_{x^{(1)}_{i,\varepsilon},\lambda^{(1)}_{i,\varepsilon}}}
 {\partial x_j}\Big|\cdot\big|\overline{w}_{\varepsilon,1}\big|\right),\end{split}
 \end{equation}
 where $W(x):= \Big(\displaystyle\sum^k_{l=1}PU_{x^{(1)}_{l,\varepsilon},\lambda^{(1)}_{l,\varepsilon}}\Big)^{2^*-2}-
\Big(\displaystyle\sum^k_{l=1}PU_{x^{(2)}_{l,\varepsilon},\lambda^{(2)}_{l,\varepsilon}}\Big)^{2^*-2}$.
Also we compute
\begin{equation}\label{07-11-11}
\begin{split}
\int_{\Omega}&\big|W(x) \big|\cdot\Big|\alpha_{\varepsilon,i,0}\frac{\partial PU_{x^{(1)}_{i,\varepsilon},\lambda^{(1)}_{i,\varepsilon}}}
 {\partial \lambda} \Big|\cdot\big| \overline{w}_{\varepsilon,1}\big|\\=&
\int_{\Omega} \Big| \sum^k_{l=1}\big(PU_{x^{(1)}_{l,\varepsilon},\lambda^{(1)}_{l,\varepsilon}} -
 PU_{x^{(2)}_{l,\varepsilon},\lambda^{(2)}_{l,\varepsilon}} \big) \Big|^{2^*-2}\cdot\Big|\alpha_{\varepsilon,i,0}\frac{\partial PU_{x^{(1)}_{i,\varepsilon},\lambda^{(1)}_{i,\varepsilon}}}
 {\partial \lambda} \Big|\cdot\big| \overline{w}_{\varepsilon,1}\big|\\=&
 O\left(\sum^k_{i=1} \big|\alpha_{\varepsilon,i,0} \big|\cdot\Big(|\lambda^{(1)}_{i,\varepsilon}-
 \lambda^{(2)}_{i,\varepsilon}|^{2^*-2}
 \int_{\Omega}  \Big|\frac{\partial PU_{x^{(1)}_{i,\varepsilon},\lambda^{(1)}_{i,\varepsilon}}}
 {\partial \lambda} \Big|^{2^*-1} \big|\overline{w}_{\varepsilon,1}\big|\Big)\right)\\&+
 O\left(\sum^k_{i=1} \big|\alpha_{\varepsilon,i,0} \big|\cdot\Big(|x^{(1)}_{i,\varepsilon}-x^{(2)}_{i,\varepsilon}|^{2^*-2}
 \int_{\Omega}  \Big|\frac{\partial PU_{x^{(1)}_{i,\varepsilon},\lambda^{(1)}_{i,\varepsilon}}}
 {\partial \lambda} \Big| \cdot \Big|\frac{\partial PU_{x^{(1)}_{i,\varepsilon},\lambda^{(1)}_{i,\varepsilon}}}
 {\partial x_j} \Big|^{2^*-2} \big|\overline{w}_{\varepsilon,1}\big|\Big)\right).
 \end{split}\end{equation}
By H\"older's inequality, we have
\begin{equation}\label{07-11-4}
\int_{\Omega}  \Big|\frac{\partial PU_{x^{(1)}_{i,\varepsilon},\lambda^{(1)}_{i,\varepsilon}}}
 {\partial \lambda} \Big|^{2^*-1}   \big| \overline{w}_{\varepsilon,1} \big|
 =O\Big(\frac{1}{\big(\lambda^{(1)}_{i,\varepsilon}\big)^{2^*-1}}
 \big\|PU_{x^{(1)}_{i,\varepsilon},\lambda^{(1)}_{i,\varepsilon}}\big\|^{2^*-1} \big\| \overline{w}_{\varepsilon,1} \big\| \Big)
 =O\Big(\frac{1}{\bar{\lambda}^{2^*-1}_{\varepsilon}}\|\overline{w}_{\varepsilon,1}\|\Big),
 \end{equation}
 and
\begin{equation}\label{07-11-5}
\begin{split}
\int_{\Omega}  &\Big|\frac{\partial PU_{x^{(1)}_{i,\varepsilon},\lambda^{(1)}_{i,\varepsilon}}}
 {\partial \lambda} \Big| \cdot \Big|\frac{\partial PU_{x^{(1)}_{i,\varepsilon},\lambda^{(1)}_{i,\varepsilon}}}
 {\partial x_j} \Big|^{2^*-2} \big| \overline{w}_{\varepsilon,1} \big|
 \\=& O\Big(\frac{1}{\big(\lambda^{(1)}_{i,\varepsilon}\big)^{3-2^*}}
 \big\|PU_{x^{(1)}_{i,\varepsilon},\lambda^{(1)}_{i,\varepsilon}}\big\|^{2^*-1} \big\| \overline{w}_{\varepsilon,1} \big\| \Big)
 =O\Big(\frac{1}{\bar{\lambda}^{3-2^*}_{\varepsilon}}\|\overline{w}_{\varepsilon,1}\|\Big).
 \end{split}
 \end{equation}
Hence using \eqref{clp-03}, \eqref{8-1-5} and \eqref{07-11-11}--\eqref{07-11-5}, we can deduce
\begin{equation}\label{07-11-1}
\begin{split}
\int_{\Omega}&\big|W(x)\big|\cdot \Big|\alpha_{\varepsilon,i,0}\frac{\partial PU_{x^{(1)}_{i,\varepsilon},\lambda^{(1)}_{i,\varepsilon}}}
 {\partial \lambda} \Big|\cdot \big| \overline{w}_{\varepsilon,1}\big|
 =O\Big(\frac{1}{\bar{\lambda}^{2^*-1}_{\varepsilon}}\|w^{(2)}_{\varepsilon}
\|\cdot\|\overline{w}_{\varepsilon,1}\|\Big).
 \end{split}
\end{equation}
And similar to the estimate of \eqref{07-11-1}, we can also find
\begin{equation}\label{07-11-2}
\begin{split}
 \int_{\Omega}\big|W(x)\big|\cdot \Big|\sum_{j=1}^N\alpha_{\varepsilon,i,j}
 \frac{\partial PU_{x^{(1)}_{i,\varepsilon},\lambda^{(1)}_{i,\varepsilon}}}
 {\partial x_j}\Big|\cdot\big|\overline{w}_{\varepsilon,1}  \big| =O\Big(\frac{1}{\bar{\lambda}^{2^*-1}_{\varepsilon}}\|w^{(2)}_{\varepsilon}
\|\cdot\|\overline{w}_{\varepsilon,1}\|\Big).
 \end{split}
\end{equation}
Next using \eqref{8-1-5} and H\"older's inequality, we have
\begin{equation}\label{07-11-3}
\begin{split}
\varepsilon\int_{\Omega} \Big|\alpha_{\varepsilon,i,0}\frac{\partial PU_{x^{(1)}_{i,\varepsilon},\lambda^{(1)}_{i,\varepsilon}}}
 {\partial \lambda}+\sum_{j=1}^N\alpha_{\varepsilon,i,j}
 \frac{\partial PU_{x^{(1)}_{i,\varepsilon},\lambda^{(1)}_{i,\varepsilon}}}
 {\partial x_j}\Big|\cdot\big|\overline{w}_{\varepsilon,1}\big|=O\Big(\frac{\varepsilon}
 {\bar{\lambda}_{\varepsilon}}  \|w^{(2)}_{\varepsilon}
\|\cdot   \|\overline{w}_{\varepsilon,1}\|\Big).
 \end{split}
\end{equation}
Hence from \eqref{1-15-8} and \eqref{07-11-1}--\eqref{07-11-3}, we find
 \begin{equation}\label{a1-15-8}
\begin{split}
 \left\langle  Q_{\varepsilon}\,\Big(\alpha_{\varepsilon,i,0}\frac{\partial PU_{x^{(1)}_{i,\varepsilon},\lambda^{(1)}_{i,\varepsilon}}}
 {\partial \lambda}+\sum_{j=1}^N\alpha_{\varepsilon,i,j}
 \frac{\partial PU_{x^{(1)}_{i,\varepsilon},\lambda^{(1)}_{i,\varepsilon}}}
 {\partial x_j}\Big),\overline{w}_{\varepsilon,1}\right\rangle=
O\Big(\frac{1}{\bar{\lambda}^{2^*-1}_{\varepsilon}}\|w^{(2)}_{\varepsilon}\|\Big).
\end{split}
 \end{equation}
 From $ \overline{w}_{\varepsilon,1}\in \displaystyle\bigcap^k_{j=1}E_{x^{(1)}_{j,\varepsilon},\lambda^{(1)}_{j,\varepsilon}}$ and \eqref{a1-15-8}, we obtain
\begin{equation}\label{2-15-8}
\begin{split}
\Big \langle Q_{\varepsilon}\overline{w}_{\varepsilon}, \overline{w}_{\varepsilon,1}\Big \rangle = &
    \Big \langle Q_{\varepsilon} \overline{w}_{\varepsilon,1},\overline{w}_{\varepsilon,1} \Big \rangle+ \left\langle Q_{\varepsilon}\,\Big(\alpha_{\varepsilon,i,0}\frac{\partial PU_{x^{(1)}_{i,\varepsilon},\lambda^{(1)}_{i,\varepsilon}}}
 {\partial \lambda}+\sum_{j=1}^N\alpha_{\varepsilon,i,j}
 \frac{\partial PU_{x^{(1)}_{i,\varepsilon},\lambda^{(1)}_{i,\varepsilon}}}
 {\partial x_j}\Big),\overline{w}_{\varepsilon,1}\right\rangle\\
 \ge & \rho\|\overline{w}_{\varepsilon,1}\|^2-
 \frac{C}{\bar{\lambda}^{2^*-1}_{\varepsilon}}\|w^{(2)}_{\varepsilon}
\|\cdot\|\overline{w}_{\varepsilon,1}\|.
\end{split}
 \end{equation}
Also, we have
\begin{equation}\label{3-15-8}
\Big \langle R_\varepsilon^{(1)}\big(w^{(1)}_{\varepsilon}\big)-R_\varepsilon^{(2)}\big(w^{(2)}_{\varepsilon}\big), \overline{w}_{\varepsilon,1}\Big \rangle
=O\Big( \big\|\sum^2_{l=1}R_\varepsilon^{(l)}\big(w^{(l)}_{\varepsilon}\big)\big\|
\cdot\|\overline{w}_{\varepsilon,1}\|\Big) =O\Big(\sum^2_{l=1} \|  w^{(l)}_{\varepsilon}\|^{\frac{N+2}{N-2}}\cdot\|\overline{w}_{\varepsilon,1}\|\Big).
 \end{equation}
And
 \begin{equation}\label{4-15-8}
\Big \langle l_{\varepsilon}, \overline{w}_{\varepsilon,1}\Big \rangle=O\Big(
{\bar{\lambda}_\varepsilon}\sum^k_{i=1}{\big|x_{i,\varepsilon}^{(1)}-x^{(2)}_{i,\varepsilon}\big|}
 +\sum^k_{i=1}\frac{\big|\lambda_{i,\varepsilon}^{(1)}-\lambda^{(2)}_{i,\varepsilon}\big|}{\bar{\lambda}_\varepsilon}\Big)
\|w^{(2)}_{\varepsilon}\|\cdot\|\overline{w}_{\varepsilon,1}\|.
 \end{equation}
Combining  \eqref{clp-03} and \eqref{2-15-8}--\eqref{4-15-8}, we obtain
 \begin{equation}\label{abc4-15-8}
\|\overline{w}_{\varepsilon,1}\| =O\Big(\sum^2_{l=1} \|  w^{(l)}_{\varepsilon}\|^{\frac{N+2}{N-2}}\Big)+O\Big(
\frac{\|w^{(2)}_{\varepsilon}\|}{\bar{\lambda}_\varepsilon}
\Big).
 \end{equation}
Then from \eqref{8-1-5} and \eqref{abc4-15-8}, we see
 \begin{equation*}
 \begin{split}
\|\overline{w}_{\varepsilon}\|=&O\left( \|\overline{w}_{\varepsilon,1 }\|+\Big\|\sum^k_{i=1}\Big(\alpha_{\varepsilon,i,0}\frac{\partial PU_{x^{(1)}_{j,\varepsilon},\lambda^{(1)}_{j,\varepsilon}}}
 {\partial \lambda}+\sum_{j=1}^N\alpha_{\varepsilon,i,j}
 \frac{\partial PU_{x^{(1)}_{j,\varepsilon},\lambda^{(1)}_{j,\varepsilon}}}
 {\partial x_j}\Big)\Big\|\right)\\ =& O\Big(\sum^2_{l=1} \|  w^{(l)}_{\varepsilon}\|^{\frac{N+2}{N-2}}\Big)  +O\Big(
\frac{\|w^{(2)}_{\varepsilon}\|}{\bar{\lambda}_\varepsilon}
\Big) +
O\left(\frac{1}{\bar{\lambda}_\varepsilon} \Big(\sum^k_{i=1}\big|\alpha_{\varepsilon,i,0}\big|+ \sum_{i=1}^k\sum_{j=1}^N \bar{\lambda}^2_\varepsilon \big|\alpha_{\varepsilon,i,j}\big|\Big)\right)\\=&
O\Big(\sum^2_{l=1} \|  w^{(l)}_{\varepsilon}\|^{\frac{N+2}{N-2}}\Big)  +O\Big(
\frac{\|w^{(2)}_{\varepsilon}\|}{\bar{\lambda}_\varepsilon}
\Big),
  \end{split}
 \end{equation*}
which and \eqref{lp21} give \eqref{ddlp21}.
\end{proof}
\begin{Prop}\label{aaprop-A.2}
For $N\geq 6$ and  $j=1,\cdots,k$, it holds
\begin{equation}\label{aaaclp-03}
\big|x^{(1)}_{j,\varepsilon}-x^{(2)}_{j,\varepsilon}\big|=o\Big(\frac{1}{\bar \lambda^2_{\varepsilon}}\Big)~\mbox{and}~\big|\lambda^{(1)}_{j,\varepsilon}-
\lambda^{(2)}_{j,\varepsilon}\big|=o\Big(\frac{1}{\bar\lambda^2_{\varepsilon}}\Big).
\end{equation}
\end{Prop}
\begin{proof}
First for $N\geq 6$, from \eqref{clp-03}, we know $\big|x^{(1)}_{j,\varepsilon}-x^{(2)}_{j,\varepsilon}\big|=O\Big(\frac{1}{\bar \lambda^2_{\varepsilon}}\Big)$. Also by direct calculations, we find
\begin{equation*}
 PU_{x^{(l)}_{j,\varepsilon},\lambda^{(l)}_{j,\varepsilon}}(y)
 =\frac{G\big(x^{(l)}_{j,\varepsilon},y\big)}{(\lambda^{(l)}_{j,\varepsilon})^{\frac{N-2}{2}}}
 +O\big(\frac{1}{\bar{\lambda}_{\varepsilon}^{\frac{N+2}{2}}}\big),~~\mbox{in}~\Omega\setminus B_\theta(x^{(l)}_{j,\varepsilon}).
\end{equation*}
Now we define the following quadratic form
\begin{equation*}
Q_1(u,v)=-\int_{\partial B_\theta(x^{(1)}_{j,\varepsilon})}\frac{\partial v}{\partial \nu^{(1)}}\frac{\partial u}{\partial x_i}-
\int_{\partial B_\theta(x^{(1)}_{j,\varepsilon})}\frac{\partial u}{\partial \nu^{(1)}}\frac{\partial v}{\partial x_i}
+\int_{\partial B_\theta(x^{(1)}_{j,\varepsilon})}\big\langle \nabla u,\nabla v \big\rangle \nu^{(1)}_i.
\end{equation*}
where $\nu^{(1)}(x)=\big(\nu^{(1)}_{1}(x),\cdots,\nu^{(1)}_N(x)\big)$ is the outward unit normal of $\partial B_{\theta}(x^{(1)}_{j,\varepsilon})$.
Note that if $u$ and $v$ are harmonic in $ B_d(x^{(1)}_{j,\varepsilon})\backslash \{x^{(1)}_{j,\varepsilon}\}$, then $Q_1(u,v)$ is independent of $\theta\in (0,d]$.
 Let $u_\varepsilon=u^{(q)}_\varepsilon$  with $q=1,2$ and
 $\Omega'=B_{\theta}(x^{(1)}_{j,\varepsilon})$ in \eqref{clp-1}. Then from  \eqref{clp-2},  we have
\begin{equation*}
\begin{split}
 \sum^k_{l=1}\sum^k_{m=1}&\frac{Q_1\big(G(x^{(1)}_{m,\varepsilon},x),G(x^{(1)}_{l,\varepsilon},x)\big)}{
 (\lambda^{(1)}_{m,\varepsilon}) ^{(N-2)/2}
 (\lambda^{(1)}_{l,\varepsilon})^{(N-2)/2}} - \sum^k_{l=1}\sum^k_{m=1}\frac{Q_1\big(G(x^{(2)}_{m,\varepsilon},x),G(x^{(2)}_{l,\varepsilon},x)\big)}{
 (\lambda^{(2)}_{m,\varepsilon}) ^{(N-2)/2}
 (\lambda^{(2)}_{l,\varepsilon})^{(N-2)/2}}\\=&
 - \sum^k_{m=1}\frac{Q_1\big(G(x^{(1)}_{m,\varepsilon},x),w_\varepsilon^{(2)}\big)}{
 (\lambda^{(1)}_{m,\varepsilon})^{(N-2)/2} }
 - \sum^k_{m=1}\frac{Q_1\big(w_\varepsilon^{(1)},G(x^{(2)}_{m,\varepsilon},x)\big)}{
 (\lambda^{(2)}_{m,\varepsilon})^{(N-2)/2} }\\&
 +\int_{\partial B_\theta\big(x^{(1)}_{j,\varepsilon}\big)} \Big[\Big(\sum^k_{j=1}
 PU_{x^{(1)}_{j,\varepsilon},\lambda^{(1)}_{j,\varepsilon}}\Big)^{2^*}-
\Big(\sum^k_{j=1}PU_{x^{(2)}_{j,\varepsilon},\lambda^{(2)}_{j,\varepsilon}}\Big)^{2^*}\Big]
 +o\Big(\frac{1}{\bar\lambda^N_{\varepsilon}}\Big)
 \\=&
O\Big(\frac{\|w_\varepsilon^{(1)}-w_\varepsilon^{(2)}\|}{\bar\lambda^{(N-2)/2}_{\varepsilon}}\Big)
 +o\Big(\frac{1}{\bar\lambda^N_{\varepsilon}}\Big)= o\Big(\frac{1}{\bar\lambda^N_{\varepsilon}}\Big).
\end{split}\end{equation*}
Let $\Lambda^{(q)}_{j,\varepsilon}:=\Big(\varepsilon ^{\frac{1}{N-4}}\lambda^{(q)}_{j,\varepsilon}\Big)^{-1}$ with $q=1,2$.
Then similar to \eqref{clp-01}, we find
\begin{equation}\label{aaaaaclp-01}
\nabla_x\Psi_k(x,\lambda)\big|_{(x,\lambda)
=(x^{(1)}_\varepsilon,\Lambda^{(1)}_\varepsilon)}-
\nabla_x\Psi_k(x,\lambda)\big|_{(x,\lambda)
=(x^{(2)}_\varepsilon,\Lambda^{(2)}_\varepsilon)}=
o\Big(\frac{1}{\bar{\lambda}_\varepsilon^{2}}\Big),
\end{equation}
where $x^{(1)}_\varepsilon=\big(x^{(1)}_{1,\varepsilon},
\cdots,x^{(1)}_{k,\varepsilon}\big)$ and $\Lambda^{(1)}_\varepsilon=\big(\Lambda^{(1)}_{1,\varepsilon},
\cdots,\Lambda^{(1)}_{k,\varepsilon}\big)$.
And similar to the estimate of \eqref{aaaaaclp-01}, we conclude
\begin{equation}\label{aaaabclp-01}
\nabla_\lambda\Psi_k(x,\lambda)\big|_{(x,\lambda)
=(x^{(1)}_\varepsilon,\Lambda^{(1)}_\varepsilon)}-
\nabla_\lambda\Psi_k(x,\lambda)\big|_{(x,\lambda)
=(x^{(2)}_\varepsilon,\Lambda^{(2)}_\varepsilon)}=
o\Big(\frac{1}{\bar{\lambda}_\varepsilon^{2}}\Big).
\end{equation}
Then \eqref{aaaclp-03} follows by  \eqref{aaaaaclp-01} and \eqref{aaaabclp-01}.
\end{proof}
Now we set
\begin{equation}\label{3.1}
\xi_{\varepsilon}(x)=\frac{u_{\varepsilon}^{(1)}(x)-u_{\varepsilon}^{(2)}(x)}
{\|u_{\varepsilon}^{(1)}-u_{\varepsilon}^{(2)}\|_{L^{\infty}(\Omega)}},
\end{equation}
then $\xi_{\varepsilon}(x)$ satisfies $\|\xi_{\varepsilon}\|_{L^{\infty}(\Omega)}=1$ and
\begin{equation}\label{3.2}
- \Delta \xi_{\varepsilon}(x)=C_{\varepsilon}(x)\xi_{\varepsilon}(x)+\varepsilon\xi_{\varepsilon}(x),
\end{equation}
where
\begin{equation*}
C_{\varepsilon}(x)=\Big(\frac{N+2}{N-2}\Big)\int_{0}^1
\Big(tu_{\varepsilon}^{(1)}(x)+(1-t)u_{\varepsilon}^{(2)}(x)\Big)
^{\frac{4}{N-2}}dt.
\end{equation*}

\begin{Lem}\label{lem-A-1}
For any constant $0<\sigma\leq N-2$, there is a constant $C>0$, such that
 \begin{equation}\label{AB.2}
   \int_{\R^N}\frac{1}{|y-z|^{N-2}}\frac{1}{(1+|z|)^{2+\sigma}}dz\leq
   \begin{cases}
  {C\big(1+|y|\big)^{-\sigma}},~& \sigma< N-2,\\[1.5mm]
    {C \big |\ln |y| \big| \big(1+|y|\big)^{-\sigma}},~& \sigma=N-2.
   \end{cases}
 \end{equation}
\end{Lem}

\begin{proof}
See Lemma B.2 in \cite{Wei}.
\end{proof}

\begin{Prop}\label{prop3.1}
For $\xi_{\varepsilon}(x)$ defined by \eqref{3.1}, we have
\begin{equation}\label{3-3}
\int_{\Omega}{\xi}_\varepsilon(x)dx=O\Big(\frac{\ln \bar{\lambda}_\varepsilon}{\bar{\lambda}^{N-2}_\varepsilon}\Big)~\mbox{and}~ \xi_{\varepsilon}(x) =O\Big(\frac{\ln \bar{\lambda}_\varepsilon}{\bar{\lambda}^{N-2}_\varepsilon}\Big),~\mbox{in}~ \Omega\backslash\bigcup_{j=1}^k B_{d}(x_{j,\varepsilon}^{(1)}),
\end{equation}
where $d>0$ is any small fixed constant.
\end{Prop}
\begin{proof}
By the potential theory,  \eqref{3.2} and \eqref{AB.2},  we have
\begin{equation}\label{cc6}
\begin{split}
 {\xi}_\varepsilon(x)=& \int_{\Omega}G(y,x)
\big(C_{\varepsilon}(y)+\varepsilon\big)\xi_{\varepsilon}(y)dy\\=&
O\Big(\sum^k_{j=1}\sum^2_{l=1}\int_{\Omega}\frac{1}{|x-y|^{N-2}} U^{\frac{4}{N-2}}_{x^{(l)}_{j,\varepsilon},\lambda^{(l)}_{j,\varepsilon}}
(y)dy\Big)+O\big(\varepsilon\big)\\=&
O\Big(\sum^k_{j=1}\sum^2_{l=1} \frac{1}{\big(1+\lambda^{(l)}_{j,\varepsilon}|x-x^{(l)}_{j,\varepsilon}|\big)^{2}}\Big)+O\big(\varepsilon\big).
\end{split}
\end{equation}
Next repeating the above process, we know
\begin{equation*}
\begin{split}
{\xi}_\varepsilon(x) = &O\left(\int_{\Omega}\frac{1}{|x-y|^{N-2}}
\big(C_{\varepsilon}(y)+\varepsilon\big)\Big(\sum^k_{j=1}\sum^2_{l=1}
\frac{1}{\big(1+\lambda^{(l)}_{j,\varepsilon}|x-x^{(l)}_{j,\varepsilon}|\big)^{2}}
+\varepsilon\Big)dy\right)\\=&
O\Big(\sum^k_{j=1}\sum^2_{l=1}\frac{1}{\big(1+\lambda^{(l)}_{j,\varepsilon}
|x-x^{(l)}_{j,\varepsilon}|\big)^{4}}\Big)+
O\big(\varepsilon^2\big).
\end{split}
\end{equation*}
Then we can proceed as in the above argument for finite number of times to prove
\begin{equation}\label{l99}
{\xi}_\varepsilon(x) =
O\Big(\sum^k_{j=1}\sum^2_{l=1}\frac{\ln \bar{\lambda}_\varepsilon}{\big(1+\lambda^{(l)}_{j,\varepsilon}
|x-x^{(l)}_{j,\varepsilon}|\big)^{N-2}}\Big)+
O\big(\varepsilon^2\big).
\end{equation}
Hence \eqref{3-3} can be deduced by \eqref{l99}.
\end{proof}

\begin{Prop}\label{prop3-2}
For $N\geq 6$ and $j=1,\cdots, k$, let $\xi_{\varepsilon,j}(x)=\xi_{\varepsilon}(\frac{x}{\lambda^{(1)}_{j,\varepsilon}}+x_{j,\varepsilon}^{(1)})$. Then by taking
a subsequence if necessary, we have
\begin{equation}\label{cc7}
\Big|\xi_{\varepsilon,j}(x)-\sum_{i=0}^N c_{j,i}\psi_{i}(x)\Big|=o\big(\frac{1}{\bar \lambda_{\varepsilon}}\big),~\mbox{uniformly in}~C^1\big(B_R(0)\big) ~\mbox{for any}~R>0,
\end{equation}
 where $c_{j,i}$, $i=0,1,\cdots,N$ are some constants  and $$\psi_{0}(x)=\frac{\partial U_{0,\lambda}(x)}{\partial\lambda}\big|_{\lambda=1},~~\psi_{i}(x)=\frac{\partial U_{0,1}(x)}{\partial x_i},~i=1,\cdots,N.
$$
\end{Prop}
\begin{proof}
Since $\xi_{\varepsilon,j}(x)$ is bounded, by the regularity theory in \cite{Gilbarg}, we find
$$\xi_{\varepsilon,j}(x)\in {C^{1,\alpha}\big(B_r(0)\big)}~ \mbox{and}~  \|\xi_{\varepsilon,j}\|_{C^{1,\alpha}\big(B_r(0)\big)}\leq C,$$
for any fixed large $r$ and $\alpha \in (0,1)$ if $\varepsilon$ is small, where the constants $r$ and $C$ are independent of $\varepsilon$ and $j$.
So we may assume that $\xi_{\varepsilon,j}(x)\rightarrow\xi_{j}(x)$ in $C\big(B_r(0)\big)$. By direct calculations, we know
\begin{small}
\begin{equation}\label{tian41}
\begin{split}
-\Delta\xi_{\varepsilon,j}(x)=&-\frac{1}{(\lambda^{(1)}_{j,\varepsilon})^{2}}\Delta \xi_{\varepsilon}\big(\frac{x}{\lambda^{(1)}_{j,\varepsilon}}+x_{j,\varepsilon}^{(1)}\big)=
\frac{1}{(\lambda^{(1)}_{j,\varepsilon})^{2}}C_{\varepsilon}(\frac{x}{\lambda^{(1)}_{j,\varepsilon}}
+x_{j,\varepsilon}^{(1)})\xi_{\varepsilon,j}(x)+ \frac{\varepsilon }{(\lambda^{(1)}_{j,\varepsilon})^{2}}\xi_{\varepsilon,j}(x).
\end{split}
\end{equation}
\end{small}
Now, we estimate $C_{\varepsilon}(\varepsilon x+x_{j,\varepsilon}^{(1)})$. By \eqref{clp-03}, we have
\begin{equation*}
\begin{split}
U&_{x_{j,\varepsilon}^{(1)},\lambda_{j,\varepsilon}^{(1)}}(x)
-
U_{x_{j,\varepsilon}^{(2)},\lambda_{j,\varepsilon}^{(2)}}(x)
 \\=&
O\Big(\big|x_{j,\varepsilon}^{(1)}-x_{j,\varepsilon}^{(2)}\big|\cdot \big(\nabla_y U_{y,\lambda_{j,\varepsilon}^{(1)}}(x)|_{y=x_{j,\varepsilon}^{(1)}}\big)+
\big|\lambda_{j,\varepsilon}^{(1)}-\lambda_{j,\varepsilon}^{(2)}\big|\cdot\big( \nabla_\lambda U_{x_{j,\varepsilon}^{(1)},\lambda}(x) |_{\lambda=\lambda_{j,\varepsilon}^{(1)}}\big)
\Big)\\=&
O\Big(\lambda_{j,\varepsilon}^{(1)}\big|x_{j,\varepsilon}^{(1)}-x_{j,\varepsilon}^{(2)}\big| +
(\lambda_{j,\varepsilon}^{(1)})^{-1}|\lambda_{j,\varepsilon}^{(1)}-\lambda_{j,\varepsilon}^{(2)}| \Big) U_{x_{j,\varepsilon}^{(1)},\lambda_{j,\varepsilon}^{(1)}}(x)
=o\big(\frac{1}{\bar\lambda_{\varepsilon}}\big)U_{x_{j,\varepsilon}^{(1)},\lambda_{j,\varepsilon}^{(1)}}(x),
\end{split}
\end{equation*}
which means
\begin{equation}\label{luopeng2}
u_{\varepsilon}^{(1)}(x)-u_{\varepsilon}^{(2)}(x)=
o\big(\frac{1}{\bar\lambda_{\varepsilon}}\big)\Big(\sum_{j=1}^kU_{x_{j,\varepsilon}^{(1)},\lambda_{j,\varepsilon}^{(1)}}(x)\Big)
+O\Big(\sum^2_{l=1}|w_{\varepsilon}^{(l)}(x)|\Big).
\end{equation}
Then for a small fixed $d$ and $x\in B_{d}(x_{j,\varepsilon}^{(1)})$, by \eqref{A--1} and
\eqref{A---1}, we find
\begin{equation}\label{cc8}
\begin{split}
C_{\varepsilon}(x) =& \Big(\frac{N+2}{N-2}+
o\big(\frac{1}{\bar\lambda_{\varepsilon}}\big)\Big)U^{\frac{4}{N-2}}_{x_{j,\varepsilon}^{(1)},\lambda_{j,\varepsilon}^{(1)}}(x)
+O\Big(\frac{1}{\bar\lambda_{\varepsilon}^{2}}\Big)
+O\Big(\sum^2_{l=1}|w_{\varepsilon}^{(l)}(x)|
^{\frac{4}{N-2}}\Big).
\end{split}
\end{equation}
Next, for any given $\Phi(x)\in C_{0}^{\infty}(\R^N)$ with $supp~\Phi(x) \subset B_{\lambda^{(1)}_{j,\varepsilon}d}(x_{j,\varepsilon}^{(1)})
$, we have
\begin{equation}\label{tian42}
\begin{aligned}
\frac{1}{(\lambda^{(1)}_{j,\varepsilon})^{2}}\int_{B_{\lambda^{(1)}_{j,\varepsilon}d}(x_{j,\varepsilon}^{(1)})}
&C_{\varepsilon}\big(\frac{x}{\lambda^{(1)}_{j,\varepsilon}}+x_{j,\varepsilon}^{(1)})\xi_{\varepsilon,j}(x)\Phi(x) dx\\=&\Big(\frac{N+2}{N-2}\Big)
\int_{\R^N}U^{\frac{4}{N-2}}_{0,1}(x)\xi_{\varepsilon,j}(x)\Phi(x)dx+
o\big(\frac{1}{\bar\lambda_{\varepsilon}}\big).
\end{aligned}
\end{equation}
Also from the fact that $\|\xi_{\varepsilon}\|_{L^{\infty}(\Omega)}=1$, we know
\begin{equation}\label{tian43}
\begin{split}
\frac{\varepsilon}{(\lambda^{(1)}_{j,\varepsilon})^{2}}\int_{B_{\lambda^{(1)}_{j,\varepsilon}d}(x_{j,\varepsilon}^{(1)})}
\xi_{\varepsilon,j}(x)\Phi(x)dx=o\big(\frac{1}{\bar\lambda_{\varepsilon}}\big).
\end{split}
\end{equation}
Then \eqref{tian41}, \eqref{tian42} and \eqref{tian43} imply
\begin{equation}\label{3.11}
\begin{split}
-&\int_{B_{\lambda^{(1)}_{j,\varepsilon}d}(x_{j,\varepsilon}^{(1)})}
\Delta\xi_{\varepsilon,j}(x)\Phi(x)dx=
\Big(\frac{N+2}{N-2}\Big)\int_{\R^N}U^{\frac{4}{N-2}}_{0,1}(x)\xi_{\varepsilon,j}(x)\Phi(x)dx+o\big(\frac{1}{\bar\lambda_{\varepsilon}}\big).
\end{split}
\end{equation}
Letting $\varepsilon\rightarrow 0$ in \eqref{3.11} and using the elliptic regularity theory, we find that $\xi_{j}(x)$ satisfies
\begin{equation}\label{luoluo-1}
-\Delta\xi_{j}(x)=\Big(\frac{N+2}{N-2}\Big)U_{0,1}^{\frac{4}{N-2}}(x)\xi_{j}(x),~~\textrm{in~}\R^N,
\end{equation}
which gives $
\xi_{j}(x)=\displaystyle\sum_{i=0}^Nc_{j,i}\psi_i(x)$. Moreover combining \eqref{3.11} and \eqref{luoluo-1}, we find \eqref{cc7}.
\end{proof}
\begin{Prop}
Let $\xi_{\varepsilon}(x)$ be defined as  in  \eqref{3.1}. Then it holds
\begin{equation}\label{laa1}
\begin{split}
\xi_{\varepsilon}(x)=&\sum^k_{j=1}A_{\varepsilon,j}G(x^{(1)}_{j,\varepsilon},x)+
\sum^k_{j=1}\sum^N_{i=1}B_{\varepsilon,j,i}\partial_iG(x^{(1)}_{j,\varepsilon},x)\\&+
\begin{cases}
O\Big(\frac{\ln \bar{\lambda}_\varepsilon}{\bar{\lambda}_\varepsilon^4}\Big),~&N=5,\\[1.5mm]
O\Big(\frac{\ln \bar{\lambda}_\varepsilon}{\bar{\lambda}_\varepsilon^N}\Big),&N\geq 6,
\end{cases}~\mbox{in}~ C^1
\Big(\Omega\backslash\bigcup^k_{j=1}B_{2d}(x^{(1)}_{j,\varepsilon})\Big),
\end{split}\end{equation}
where $d>0$ is any small fixed constant, $\partial_iG(y,x)=\frac{\partial G(y,x)}{\partial y_i}$,
\begin{equation}\label{luo51}
A_{\varepsilon,j}=
\int_{B_d\big(x^{(1)}_{j,\varepsilon}\big)} C_\varepsilon(x)\xi_{\varepsilon}(x)dx~\mbox{and}~
B_{\varepsilon,j,i}= \int_{B_d\big(x^{(1)}_{j,\varepsilon}\big)}(x_i-x^{(1)}_{j,\varepsilon,i}) C_\varepsilon(x)\xi_{\varepsilon}(x)dx.
\end{equation}
\end{Prop}
\begin{proof}
By the potential theory and \eqref{3.2}, we have
\begin{equation}\label{tian44}
\begin{split}
 {\xi}_\varepsilon(x)=& \int_{\Omega}G(y,x)
 C_{\varepsilon}(y) \xi_{\varepsilon}(y)dy+  \varepsilon \int_{\Omega}G(y,x)
\xi_{\varepsilon}(y)dy.
\end{split}
\end{equation}
Next for $x\in \Omega\backslash\displaystyle\bigcup^k_{j=1}B_{2d}(x^{(1)}_{j,\varepsilon})$, by \eqref{3-3}  we find
\begin{equation}\label{tian48}
\begin{split}
 \int_{\Omega}G(y,x)\xi_{\varepsilon}(y)dy=&
 \sum^k_{j=1}\int_{B_d\big(x^{(1)}_{j,\varepsilon}\big)}G(y,x)
 \xi_{\varepsilon}(y)dy+
\int_{\Omega\backslash\bigcup^k_{j=1}B_d\big(x^{(1)}_{j,\varepsilon}\big)}G(y,x)
 \xi_{\varepsilon}(y)dy\\=&
 \sum^k_{j=1}\int_{B_d\big(x^{(1)}_{j,\varepsilon}\big)}
 \xi_{\varepsilon}(y)dy+O\Big(\frac{\ln \bar{\lambda}_\varepsilon}{\bar{\lambda}^{N-2}_\varepsilon}
\int_{\Omega }G(y,x)
 dy\Big)=O\Big(\frac{\ln \bar{\lambda}_\varepsilon}{\bar{\lambda}^{N-2}_\varepsilon}\Big).
\end{split}
\end{equation}
Also using   \eqref{3-3} and \eqref{cc7}, we can get
\begin{equation}\label{atian45}
\begin{split}
 \int_{\Omega}G&(y,x)
 C_{\varepsilon}(y) \xi_{\varepsilon}(y)dy \\=&
 \sum^k_{j=1}\int_{B_d\big(x^{(1)}_{j,\varepsilon}\big)}G(y,x)
 C_{\varepsilon}(y) \xi_{\varepsilon}(y)dy+
 \int_{\Omega\backslash\bigcup^k_{j=1}B_d\big(x^{(1)}_{j,\varepsilon}\big)}G(y,x)
 C_{\varepsilon}(y) \xi_{\varepsilon}(y)dy\\=&
 \sum^k_{j=1}A_{\varepsilon,j}G(x^{(1)}_{j,\varepsilon},x)
 +\sum^k_{j=1}\sum^N_{i=1}B_{\varepsilon,j,i}\partial_iG(x^{(1)}_{j,\varepsilon},x)\\&
 +O\Big(\sum^k_{j=1}\int_{B_d\big(x^{(1)}_{j,\varepsilon}\big)}|y-x^{(1)}_{j,\varepsilon}|^2
 C_{\varepsilon}(y) \xi_{\varepsilon}(y)dy\Big)+O\Big(\frac{\ln \bar{\lambda}_\varepsilon}{\bar\lambda^N_\varepsilon}
 \int_{\Omega\backslash\bigcup^k_{j=1}B_d\big(x^{(1)}_{j,\varepsilon}\big)}G(y,x)dy\Big)
 \\=&
  \sum^k_{j=1}A_{\varepsilon,j}G(x^{(1)}_{j,\varepsilon},x)
 + \sum^k_{j=1}\sum^N_{i=1}B_{\varepsilon,j,i}\partial_iG(x^{(1)}_{j,\varepsilon},x)
 +O\Big(\frac{\ln \bar{\lambda}_\varepsilon}{\bar\lambda^N_\varepsilon}\Big),
\end{split}
\end{equation}
where $A_{\varepsilon,j}$ and $B_{\varepsilon,j,i}$ are defined  in \eqref{luo51}.
Hence \eqref{4-16-1}, \eqref{tian44}, \eqref{tian48} and \eqref{atian45} imply
\begin{equation*}\begin{split}
\xi_{\varepsilon}(x)=&
  \sum^k_{j=1}A_{\varepsilon,j}G(x^{(1)}_{j,\varepsilon},x)
 + \sum^k_{j=1}\sum^N_{i=1}B_{\varepsilon,j,i}\partial_iG(x^{(1)}_{j,\varepsilon},x)
 +O\Big(\frac{\ln \bar{\lambda}_\varepsilon}{\bar\lambda^N_\varepsilon}+\frac{\varepsilon \ln \bar{\lambda}_\varepsilon}{\bar{\lambda}^{N-2}_\varepsilon}\Big)\\=&
\sum^k_{j=1}A_{\varepsilon,j}G(x^{(1)}_{j,\varepsilon},x)+
\sum^k_{j=1}\sum^N_{i=1}B_{\varepsilon,j,i}\partial_iG(x^{(1)}_{j,\varepsilon},x)\\&+
\begin{cases}
O\Big(\frac{\ln \bar{\lambda}_\varepsilon}{\bar{\lambda}_\varepsilon^4}\Big),~&N=5,\\[1.5mm]
O\Big(\frac{\ln \bar{\lambda}_\varepsilon}{\bar{\lambda}_\varepsilon^N}\Big),&N\geq 6,
\end{cases}~\mbox{for}~ x\in
 \Omega\backslash\bigcup^k_{j=1}B_{2d}(x^{(1)}_{j,\varepsilon}).
 \end{split}
\end{equation*}
On the other hand, from \eqref{tian44}, we know
\begin{equation*}
\begin{split}
\frac{\partial{\xi}_\varepsilon(x)}{\partial x_i}=& \int_{\Omega}D_{x_i}G(y,x)
 C_{\varepsilon}(y) \xi_{\varepsilon}(y)dy+  \varepsilon \int_{\Omega}D_{x_i}G(y,x)
\xi_{\varepsilon}(y)dy,~\mbox{for}~i=1,\cdots,N.
\end{split}
\end{equation*}
Then similar to the above estimates of $\xi_{\varepsilon}(x)$, we can  complete the proof of \eqref{laa1}.
\end{proof}
\section{Proofs of Theorem \ref{th1.1} and Theorem \ref{th1.2}}\label{s4}
\setcounter{equation}{0}

In this section we prove Theorem \ref{th1.1} and Theorem \ref{th1.2}. Theorem \ref{th1.1} will be verified by using an indirect argument to  show  $\xi_{\varepsilon}(x)\equiv 0$ for $\varepsilon$ small. To do this we
first show that $\xi_{\varepsilon}(x)$ is small both near and away from the points at which
 $u^{(l)}_{\varepsilon}(x)$ ($l=1,2$) concentrates. Therefore we need to obtain quantitative
  behaviors for $u^{(1)}_{\varepsilon}(x)$ , $u^{(2)}_{\varepsilon}(x)$ and  $\xi_{\varepsilon}(x)$.
\begin{Prop}
Let $u^{(l)}_{\varepsilon}(x)$ with $l=1,2$ be the solutions of $\eqref{1.1}$ satisfying \eqref{4-6-1}. Then for small fixed $d>0$, it holds
\begin{equation}\label{ccc}
u^{(l)}_{\varepsilon}(x)=A\Big(\sum^k_{j=1}\frac{G(x^{(1)}_{j,\varepsilon},x)}{
(\lambda^{(1)}_{j,\varepsilon})^{(N-2)/2}}\Big)
+\begin{cases}
 O\Big(\frac{1}{\bar\lambda_\varepsilon^{5/2}}\Big),~&N=5,\\[1.5mm]
 O\Big(\frac{1}{\bar\lambda_\varepsilon^{(N+2)/2}}\Big),~&N\geq 6,
\end{cases}~\mbox{in}~  C^1
\Big(\Omega\backslash\bigcup^k_{j=1}B_{2d}(x^{(1)}_{j,\varepsilon})\Big),
\end{equation}
where $A$ is the constant in \eqref{a1}.
\end{Prop}
\begin{proof}
First, \eqref{clp-2} implies that \eqref{ccc} holds for $l=1$
and
\begin{equation}\label{tian1}
u^{(2)}_{\varepsilon}(x)=A\Big(\sum^k_{j=1}\frac{G(x^{(2)}_{j,\varepsilon},x)}{
(\lambda^{(2)}_{j,\varepsilon})^{(N-2)/2}}\Big)
+\begin{cases}
 O\Big(\frac{1}{\bar\lambda_\varepsilon^{5/2}}\Big),~&N=5,\\[1.5mm]
 O\Big(\frac{1}{\bar\lambda_\varepsilon^{(N+2)/2}}\Big),~&N\geq 6,
\end{cases}~\mbox{in}~  C^1
\Big(\Omega\backslash\bigcup^k_{j=1}B_{2d}(x^{(2)}_{j,\varepsilon})\Big).
\end{equation}
Also we calculate
\begin{equation*}
\frac{G(x^{(2)}_{j,\varepsilon},x)}{
(\lambda^{(2)}_{j,\varepsilon})^{(N-2)/2}}=
\frac{G(x^{(1)}_{j,\varepsilon},x)}{
(\lambda^{(1)}_{j,\varepsilon})^{(N-2)/2}}+O\Big(\frac{
|x^{(1)}_{j,\varepsilon}-x^{(2)}_{j,\varepsilon}|}{\bar{\lambda}_\varepsilon^{(N-2)/2}}\Big)
+O\Big(\frac{|\lambda^{(1)}_{j,\varepsilon}-\lambda^{(2)}_{j,\varepsilon}|}{\bar{\lambda}_\varepsilon^{N/2}}\Big).
\end{equation*}
Since $B_{d}(x^{(1)}_{j,\varepsilon})\subset B_{2d}(x^{(2)}_{j,\varepsilon})$ for small $\varepsilon$,   we get \eqref{ccc} for $l=2$ from \eqref{clp-03} and \eqref{tian1}.
\end{proof}
\begin{Prop}
For $\xi_{\varepsilon}$ defined by \eqref{3.1}, we have the following local Pohozaev identities:
 \begin{equation}\label{a2}
 \begin{split}
-\int_{\partial \Omega'}&\frac{\partial \xi_\varepsilon}{\partial \nu}\frac{\partial u^{(1)}_\varepsilon}{\partial x_i}-
\int_{\partial \Omega'}\frac{\partial u^{(2)}_\varepsilon}{\partial \nu}\frac{\partial \xi_\varepsilon}{\partial x_i}
+\frac{1}{2}\int_{\partial \Omega'}\big\langle \nabla (u^{(1)}_\varepsilon+u^{(2)}_\varepsilon),\nabla \xi_\varepsilon \big\rangle \nu_i\\&
= \frac{N-2}{2N} \int_{\partial \Omega'}D_\varepsilon(x) \xi_\varepsilon \nu_i
+  \varepsilon\int_{\partial \Omega'} (u_{\varepsilon}^{(1)}+u_{\varepsilon}^{(2)})\xi_\varepsilon \nu_i,
\end{split}
\end{equation}
and
\begin{equation}\label{clp1}
\begin{split}
\frac{1}{2}&\int_{\partial \Omega'}
\big\langle \nabla (u_{\varepsilon}^{(1)}+u_{\varepsilon}^{(2)}),  \nabla \xi_{\varepsilon}\big\rangle
\big\langle x-x^{(1)}_{j,\varepsilon},\nu\big\rangle -\int_{\partial \Omega'}\frac{\partial\xi_\varepsilon}{\partial\nu}
\big\langle x-x^{(1)}_{j,\varepsilon},\nabla u_{\varepsilon}^{(1)}\big\rangle\\&
-\int_{\partial \Omega'}\frac{\partial u^{(2)}_\varepsilon}{\partial\nu}
\big\langle x-x^{(1)}_{j,\varepsilon},\nabla \xi_{\varepsilon}\big\rangle
+\frac{2-N}{2}\int_{\partial \Omega'}\frac{\partial\xi_\varepsilon}{\partial\nu}
  u_{\varepsilon}^{(1)}
 +\frac{2-N}{2}\int_{\partial \Omega'}\frac{\partial u_{\varepsilon}^{(2)}}{\partial\nu}\xi_\varepsilon   \\=&
  \int_{\partial \Omega'} D_{\varepsilon}(x)\xi_\varepsilon \big\langle x-x^{(1)}_{j,\varepsilon},\nu\big\rangle  +\varepsilon\int_{\partial \Omega'} (u_{\varepsilon}^{(1)}+u_{\varepsilon}^{(2)})\xi_\varepsilon \big\langle x-x^{(1)}_{j,\varepsilon},\nu\big\rangle
  -\varepsilon\int_{ \Omega'} (u_{\varepsilon}^{(1)}+u_{\varepsilon}^{(2)})\xi_\varepsilon ,
\end{split}
\end{equation}
where  $\Omega'\subset \Omega$ is a smooth domain, $\nu(x)=\big(\nu_{1}(x),\cdots,\nu_N(x)\big)$ is the outward unit normal of $\partial \Omega'$ and
\begin{equation*}
D_{\varepsilon}(x)= \int_{0}^1\Big(tu_{\varepsilon}^{(1)}(x)+(1-t)u_{\varepsilon}^{(2)}(x)\Big)^{\frac{N+2}{N-2}}
dt.
\end{equation*}
\end{Prop}
\begin{proof}
Taking $u_\varepsilon=u_{\varepsilon}^{(l)}$ with $l=1,2$ in \eqref{clp-1} and \eqref{clp-10}, and  then making a difference between those respectively, we can obtain \eqref{a2} and \eqref{clp1}.
\end{proof}

\begin{Prop}\label{prop-luo1}
For $N\geq 6$, it holds
\begin{equation}\label{luo-13}
c_{j,0}=0,~\mbox{for}~j=1,\cdots,k,
\end{equation}
where $c_{j,0}$ are the constants in Proposition \ref{prop3-2}.
\end{Prop}
\begin{proof}
First, we define the following quadric form
\begin{equation*}
\begin{split}
P_1(u,v)=&- \theta\int_{\partial B_\theta(x^{(1)}_{j,\varepsilon})}
\big\langle \nabla u ,\nu\big\rangle
\big\langle \nabla v,\nu\big\rangle
+ \frac{\theta}{2} \int_{\partial B_\theta(x^{(1)}_{j,\varepsilon})}
\big\langle \nabla u , \nabla v \big\rangle\\&
+\frac{2-N}{4}\int_{\partial B_\theta(x^{(1)}_{j,\varepsilon})}
\big\langle \nabla u ,  \nu \big\rangle v
+\frac{2-N}{4}\int_{\partial B_\theta(x^{(1)}_{j,\varepsilon})}
\big\langle \nabla v ,  \nu \big\rangle u.
\end{split}
\end{equation*}
Note that if $u$ and $v$ are harmonic in $ B_d\big(x^{(1)}_{j,\varepsilon}\big)\backslash \{x^{(1)}_{j,\varepsilon}\}$, then $P_1(u,v)$ is independent of $\theta\in (0,d]$. So using \eqref{laa1} and \eqref{ccc}, we get by taking $\Omega'=B_\theta\big(x^{(1)}_{j,\varepsilon}\big)$ in \eqref{clp1}, for $N\geq 6$,
\begin{equation}\label{tian61}
\begin{split}
 \text{LHS of \eqref{clp1}}=&\sum^k_{l=1}\sum^k_{m=1}\frac{2AA_{\varepsilon,l}}{(\lambda^{(1)}_{m,\varepsilon})^{(N-2)/2}}
P_1\Big(G(x^{(1)}_{m,\varepsilon},x),G(x^{(1)}_{l,\varepsilon},x)\Big)
\\&+\sum^k_{l=1}\sum^k_{m=1}\sum^N_{h=1}\frac{2AB_{\varepsilon,l,h}}{(\lambda^{(1)}_{m,\varepsilon})^{(N-2)/2}}
P_1\Big(G(x^{(1)}_{m,\varepsilon},x),\partial_hG(x^{(1)}_{l,\varepsilon},x)\Big)
+O\Big(\frac{1}{\bar{\lambda}_\varepsilon^{(3N-2)/2}}\Big).
\end{split}
\end{equation}
Similar to \eqref{luo2}, we have
\begin{equation}\label{1-1}
P_1\Big(G(x^{(1)}_{m,\varepsilon},x), G(x^{(1)}_{l,\varepsilon},x)\Big)=
\begin{cases}
-\frac{(N-2)R(x^{(1)}_{j,\varepsilon})}{2},~&\mbox{for}~l,m=j.\\[1mm]
\frac{(N-2)G(x^{(1)}_{j,\varepsilon},x^{(1)}_{l,\varepsilon})}{4},~
&\mbox{for}~m=j, ~l\neq j. \\[1mm]
\frac{(N-2)G(x^{(1)}_{j,\varepsilon},x^{(1)}_{m,\varepsilon})}{4},~&\mbox{for}~m\neq j, ~l=j.\\[1mm]
0, ~&\mbox{for}~l,m\neq j.
\end{cases}
\end{equation}
Also we have the following estimates for which the proof is postponed until Section 5:
 \begin{equation}\label{bb1-1}
P_1\Big(G(x^{(1)}_{m,\varepsilon},x),\partial_hG(x^{(1)}_{l,\varepsilon},x)\Big)=
\begin{cases}
\big(\frac{N-2}{4}+\frac{1-N}{2N}\big)\partial _h R\big(x^{(1)}_{j,\varepsilon}\big),~&\mbox{for}~l,m=j.\\[1mm]
\frac{(N-2)}{4}\partial_hG \big(x^{(1)}_{l,\varepsilon},x^{(1)}_{j,\varepsilon} \big),~
&\mbox{for}~m=j, ~l\neq j. \\[1mm]
\big(\frac{N-2}{4}+\frac{1-N}{N}\big)\partial _hG \big(x^{(1)}_{j,\varepsilon},x^{(1)}_{m,\varepsilon} \big),~&\mbox{for}~m\neq j, ~l=j.\\[1mm]
0, ~&\mbox{for}~l,m\neq j.
\end{cases}
\end{equation}
On the other hand, from  \eqref{3-3}, \eqref{cc7}, \eqref{cc8}, \eqref{luo51} and \eqref{lp21}, we deduce
\begin{equation}\label{luo31}
\begin{split}
A_{\varepsilon,j}=&\Big(\frac{N+2}{N-2}\Big)
\int_{B_d\big(x^{(1)}_{j,\varepsilon}\big)} U^{\frac{4}{N-2}}_{x^{(1)}_{j,\varepsilon},\lambda^{(1)}_{j,\varepsilon}}\xi_{\varepsilon}+
o\Big(\frac{1}{\bar{\lambda}_\varepsilon^{N-1}}\Big)=
-\frac{(N-2)Ac_{j,0}}{2(\lambda^{(1)}_{j,\varepsilon})^{N-2}}+
o\Big(\frac{1}{\bar{\lambda}_{\varepsilon}^{N-1}}\Big).
\end{split}
\end{equation}
Now let $
d_{j,\varepsilon}=\frac{(N-2)A^2c_{j,0}}{8(\lambda^{(1)}_{j,\varepsilon})^{(N-2)/2}}$, for $j=1,\cdots,k$.
Then \eqref{tian61}--\eqref{luo31} imply
\begin{small}
\begin{equation}\label{luo3}
\begin{split}
 \text{LHS of \eqref{clp1}}=&
 {2(N-2)d_{j,\varepsilon}} \Big(\frac{R(x^{(1)}_{j,\varepsilon})}{(\lambda^{(1)}_{j,\varepsilon})^{N-2}}-\sum^k_{l\neq j}\frac{G(x^{(1)}_{j,\varepsilon},x^{(1)}_{l,\varepsilon})}
 {(\lambda^{(1)}_{j,\varepsilon})^{(N-2)/2}(\lambda^{(1)}_{l,\varepsilon})^{(N-2)/2}}
\Big)\\&
+{2(N-2)} \Big(\frac{d_{j,\varepsilon} R(x^{(1)}_{j,\varepsilon})}{(\lambda^{(1)}_{j,\varepsilon})^{N-2}}-\sum^k_{l\neq j}\frac{d_{l,\varepsilon}G(x^{(1)}_{j,\varepsilon},x^{(1)}_{l,\varepsilon})}
 {(\lambda^{(1)}_{j,\varepsilon})^{(N-2)/2}(\lambda^{(1)}_{l,\varepsilon})^{(N-2)/2}}
\Big)\\&
+\frac{N-2}{2}\sum^N_{h=1} B_{\varepsilon,j,h} \Big(\frac{\partial_hR(x^{(1)}_{j,\varepsilon}) }{(\lambda^{(1)}_{j,\varepsilon})^{(N-2)/2}} -\sum^k_{l\neq j}
\frac{\partial_h G(x^{(1)}_{j,\varepsilon},x^{(1)}_{l,\varepsilon})}{(\lambda^{(1)}_{l,\varepsilon})^{(N-2)/2}}
\Big) \\&
-\frac{N-2}{2} \sum^N_{h=1}\sum^k_{l\neq j} \frac{B_{\varepsilon,l,h}}{(\lambda^{(1)}_{j,\varepsilon})^{(N-2)/2}}\partial_h G(x^{(1)}_{j,\varepsilon},x^{(1)}_{l,\varepsilon})
+o\Big(\frac{1}{\bar{\lambda}_\varepsilon^{(3N-4)/2}}\Big)\\&
+
\frac{(N-1)}{N} \sum^N_{h=1}B_{\varepsilon,j,h} \Big(\frac{ \partial_hR(x^{(1)}_{j,\varepsilon})}{(\lambda^{(1)}_{j,\varepsilon})^{N-2}}-2\sum^k_{l\neq j}\frac{ \partial_hG(x^{(1)}_{j,\varepsilon},x^{(1)}_{l,\varepsilon})}
 {(\lambda^{(1)}_{j,\varepsilon})^{(N-2)/2}(\lambda^{(1)}_{l,\varepsilon})^{(N-2)/2}}
\Big).
\end{split}
\end{equation}
\end{small}
Also, from \eqref{cc5}, \eqref{3-3}, \eqref{luopeng2}  and \eqref{lp21},  we know
\begin{equation}\label{luo4}
\begin{split}
 \text{RHS of \eqref{clp1}}=& -2\varepsilon\int_{ B_\theta\big(x^{(1)}_{j,\varepsilon}\big)} U_{x^{(1)}_{j,\varepsilon},\lambda^{(1)}_{j,\varepsilon}}(x)\xi_\varepsilon
+o\Big(\frac{1}{\bar{\lambda}_\varepsilon^{(3N-4)/2}}\Big)\\=&
\frac{2Bc_{j,0}\varepsilon}{(\lambda^{(1)}_{j,\varepsilon})^{(N+2)/2}}
+o\Big(\frac{1}{\bar{\lambda}_\varepsilon^{(3N-4)/2}}\Big)\\=&
8{d_{j,\varepsilon}} \Big(\frac{R(x^{(1)}_{j,\varepsilon})}{(\lambda^{(1)}_{j,\varepsilon})^{N-2}}-\sum^k_{l\neq j}\frac{G(x^{(1)}_{j,\varepsilon},x^{(1)}_{l,\varepsilon})}
 {(\lambda^{(1)}_{j,\varepsilon})^{(N-2)/2}(\lambda^{(1)}_{l,\varepsilon})^{(N-2)/2}}
\Big)+o\Big(\frac{1}{\bar{\lambda}_\varepsilon^{(3N-4)/2}}\Big).
\end{split}
\end{equation}
It follows from \eqref{luo-tian} that
\begin{equation}\label{luo-5}
\begin{split}
\sum^k_{h=1}&B_{\varepsilon,j,h} \Big(\frac{ \partial_hR(x^{(1)}_{j,\varepsilon})}{(\lambda^{(1)}_{j,\varepsilon})^{N-2}}-2\sum^k_{l\neq j}\frac{ \partial_hG(x^{(1)}_{j,\varepsilon},x^{(1)}_{l,\varepsilon})}
 {(\lambda^{(1)}_{j,\varepsilon})^{(N-2)/2}(\lambda^{(1)}_{l,\varepsilon})^{(N-2)/2}}
\Big) \\=&
O\Big(\sum^k_{h=1}\frac{B_{\varepsilon,j,h}}{(\lambda^{(1)}_{j,\varepsilon})^{N/2}} \big(\partial_h \Psi_k(a^k,\Lambda^k)
+o(1)\big)\Big)
=o\Big(\frac{1}{\bar{\lambda}_\varepsilon^{(3N-2)/2}}\Big).
\end{split}
\end{equation}
Hence \eqref{tian61} and \eqref{luo3}--\eqref{luo-5} imply, for $j=1,\cdots,k$,
\begin{equation}\label{luo5}
\begin{split}
 &{(N-6)d_{j,\varepsilon}} \Big(\frac{R(x^{(1)}_{j,\varepsilon})}{(\lambda^{(1)}_{j,\varepsilon})^{N-2}}-\sum^k_{l\neq j}\frac{G(x^{(1)}_{j,\varepsilon},x^{(1)}_{l,\varepsilon})}
 {(\lambda^{(1)}_{j,\varepsilon})^{(N-2)/2}(\lambda^{(1)}_{l,\varepsilon})^{(N-2)/2}}
\Big)\\&
+{(N-2)} \Big(\frac{d_{j,\varepsilon} R(x^{(1)}_{j,\varepsilon})}{(\lambda^{(1)}_{j,\varepsilon})^{N-2}}-\sum^k_{l\neq j}\frac{d_{l,\varepsilon}G(x^{(1)}_{j,\varepsilon},x^{(1)}_{l,\varepsilon})}
 {(\lambda^{(1)}_{j,\varepsilon})^{(N-2)/2}(\lambda^{(1)}_{l,\varepsilon})^{(N-2)/2}}
\Big)\\&
+\frac{N-2}{4}\sum^N_{h=1} B_{\varepsilon,j,h} \Big(\frac{\partial_hR(x^{(1)}_{j,\varepsilon}) }{(\lambda^{(1)}_{j,\varepsilon})^{(N-2)/2}} -\sum^k_{l\neq j}
\frac{\partial_h G(x^{(1)}_{j,\varepsilon},x^{(1)}_{l,\varepsilon})}{(\lambda^{(1)}_{l,\varepsilon})^{(N-2)/2}}
\Big) \\&
-\frac{N-2}{4} \sum^N_{h=1}\sum^k_{l\neq j} \frac{B_{\varepsilon,l,h}}{(\lambda^{(1)}_{j,\varepsilon})^{(N-2)/2}}\partial_h G(x^{(1)}_{j,\varepsilon},x^{(1)}_{l,\varepsilon})
=o\Big(\frac{1}{\bar{\lambda}_\varepsilon^{(3N-2)/2}}\Big).
\end{split}
\end{equation}
Let $a^k=(a_1,\cdots,a_k)\in \Omega^k$ and $a_j=(y_{(j-1)N+1},y_{(j-1)N+2},\cdots,y_{jN})\in \Omega$.
Since for any $i=\{1,\cdots,kN\}$, there exists some $j\in \{1,\cdots,k\}$ satisfying $i\in [(j-1)N+1,jN]\bigcap \N^+$. Then by direct calculation, we have
 \begin{equation*}
\frac{\partial^2\Psi_k(x,\lambda)}{\partial y_i\partial \lambda_m}=
 (N-2)\Big(\lambda_{j}^{N-3}
\frac{\partial R(x_j)}{\partial y_i}- \displaystyle\sum^k_{l\neq j}\lambda_{j}^{\frac{N-4}{2}}\lambda_{l}^{\frac{N-2}{2}}
\frac{\partial G(x_j,x_l)}{\partial y_i}\Big),~\mbox{if}~m\in [(j-1)N+1,jN]\bigcap \N^+,
\end{equation*}
and
 \begin{equation*}
\frac{\partial^2\Psi_k(x,\lambda)}{\partial y_i\partial \lambda_m}=
-(N-2)\lambda_{j}^{\frac{N-2}{2}}\lambda_{s}^{\frac{N-4}{2}}
\frac{\partial G(x_j,x_s)}{\partial y_i},~\mbox{if}~ m\in [(s-1)N+1,sN]\bigcap \N^+~\mbox{and}~s\neq j.
\end{equation*}
Now we rewrite \eqref{luo5} as follows:
\begin{equation}\label{tian62}
\begin{split}
&\overline{M}_{k,\varepsilon} \vec{D}_{k}+\frac{2}{A^2}
\widetilde{D}_{k}
\Big(D^2_{\lambda,x}\Psi_k(x,\lambda)\Big)_{(x,\lambda)=(a^k,\Lambda^k)}
\widetilde{B}_{\varepsilon,k} =\Big(o\big(\frac{1}{\bar{\lambda}_\varepsilon^{N-1}}\big),
\cdots,o\big(\frac{1}{\bar{\lambda}_\varepsilon^{N-1}}\big)\Big)^T,
\end{split}\end{equation}
with the vector $\vec{D}_{k}$, the matrix $\overline{M}_{k,\varepsilon}=\big(a_{i,j,\varepsilon}\big)_{1\leq i,j\leq k}$ defined by
\begin{equation*}
\vec{D}_{k}=\big(c_{1,0},\cdots,c_{k,0}\big)^T,~
\widetilde{D}_{k}=diag \big(\lambda^{\frac{4-N}{2}}_{1},\cdots,\lambda^{\frac{4-N}{2}}_{k}\big),
~\mbox{with}~ \lambda_{j}:=\lim_{\varepsilon\to 0}\Big(\varepsilon ^{\frac{1}{N-4}}\lambda_{j,\varepsilon}\Big)^{-1},
\end{equation*}
\begin{equation*}
~a_{i,j,\varepsilon}=
\begin{cases}
\frac{2(N-4)R(x^{(1)}_{j,\varepsilon})}{(\lambda^{(1)}_{j,\varepsilon})^{N-2}}-
 \sum^k_{l\neq j}\frac{(N-6)G(x^{(1)}_{j,\varepsilon},x^{(1)}_{l,\varepsilon})}
 {(\lambda^{(1)}_{j,\varepsilon})^{(N-2)/2}(\lambda^{(1)}_{l,\varepsilon})^{(N-2)/2}},& \mbox{for}~i=j,\\[1.5mm]
 -\frac{(N-2)G(x^{(1)}_{j,\varepsilon},x^{(1)}_{i,\varepsilon})}
 {(\lambda^{(1)}_{j,\varepsilon})^{(N-2)/2}(\lambda^{(1)}_{i,\varepsilon})^{(N-2)/2}},& \mbox{for}~i\neq j,
\end{cases}
\end{equation*}
and
\begin{equation}\label{7-30-1}
\widetilde{B}_{\varepsilon,k}=\big(\bar B_{\varepsilon,1},\cdots,\bar B_{\varepsilon,Nk}\big)^T,~\mbox{with}~
\bar B_{\varepsilon,m}=B_{\varepsilon,j,h},~ m=h+N(j-1).
\end{equation}
Since $\overline{M}_{k,\varepsilon}$ is the main diagonally dominant matrix if $N\geq 6$, we see that $\overline{M}_{k,\varepsilon}$ is invertible. So \eqref{tian62} means \eqref{luo-13}.
Then from \eqref{luo31}  and \eqref{tian62}, we obtain
\begin{equation}\label{luo13}
A_{\varepsilon,j}=
o\Big(\frac{1}{\bar{\lambda}_{\varepsilon}^{N-1}}\Big),
~\mbox{for}~j=1,\cdots,k~\mbox{and}~N\geq 6.
\end{equation}
Therefore we find
 \begin{equation}\label{dddtian62}
\begin{split}
\Big(D^2_{\lambda,x}\Psi_k(x,\lambda)\Big)_{(x,\lambda)=(a^k,\Lambda^k)}
\widetilde{B}_{\varepsilon,k} =\Big(o\big(\frac{1}{\bar{\lambda}_\varepsilon^{N-1}}\big),
\cdots,o\big(\frac{1}{\bar{\lambda}_\varepsilon^{N-1}}\big)\Big)^T.
\end{split}\end{equation}
\end{proof}
\begin{Prop}\label{prop-luo}
For $N\geq 6$, it holds
\begin{equation}\label{luo--13}
c_{j,i}=0,~\mbox{for}~j=1,\cdots,k~\mbox{and}~i=1,\cdots,N,
\end{equation}
where $c_{j,i}$ are the constants in Proposition \ref{prop3-2}.
\end{Prop}

\begin{proof}
First taking $\Omega'=B_\theta\big(x^{(1)}_{j,\varepsilon}\big)$ in \eqref{a2}, from \eqref{laa1} and \eqref{ccc}, we have
\begin{equation}\label{luo21}
\begin{split}
 \text{LHS of \eqref{a2}}=&\sum^k_{l=1}\sum^k_{m=1}\frac{AA_{\varepsilon,l}Q_1
 \Big(G(x^{(1)}_{m,\varepsilon},x),G(x^{(1)}_{l,\varepsilon},x)\Big)}{(\lambda^{(1)}_{m,\varepsilon})^{(N-2)/2}}
\\&+
\sum^k_{l=1}\sum^N_{h=1}\sum^k_{m=1}\frac{AB_{\varepsilon,l,h}Q_1
\Big(G(x^{(1)}_{m,\varepsilon},x),\partial_hG(x^{(1)}_{l,\varepsilon},x)\Big)}{(\lambda^{(1)}_{m,\varepsilon})^{(N-2)/2}}
+O\Big(\frac{\ln \bar{\lambda}_\varepsilon}{\bar{\lambda}_\varepsilon^{(3N-2)/2}}\Big),
\end{split}
\end{equation}
and
\begin{equation}\label{luo22}
\begin{split}
 \text{RHS of \eqref{a2}}= O\Big(\frac{1}{\bar{\lambda}_\varepsilon^{(3N-2)/2}}\Big).
\end{split}
\end{equation}
Then from \eqref{luo13}, \eqref{luo21} and \eqref{luo22}, we find
\begin{equation}\label{luo24}
\begin{split}
\sum^k_{l=1}\sum^N_{h=1}\sum^k_{m=1}\frac{B_{\varepsilon,l,h}Q_1
\Big(G(x^{(1)}_{m,\varepsilon},x),\partial_hG(x^{(1)}_{l,\varepsilon},x)\Big)}{(\lambda^{(1)}_{m,\varepsilon})^{(N-2)/2}}
=o\Big(\frac{1}{\bar{\lambda}_\varepsilon^{(3N-4)/2}}\Big).
\end{split}
\end{equation}
Also we have the following estimates for which the proof is postponed until Section 5:
\begin{equation}\label{luo41}
Q_1\Big(G(x^{(1)}_{m,\varepsilon},x),\partial_h G(x^{(1)}_{l,\varepsilon},x)\Big)=
\begin{cases}
-\partial^2_{x_ix_h}R(x^{(1)}_{j,\varepsilon}),~&\mbox{for}~l,m=j,\\[1mm]
D^2_{x_ix_h}G(x^{(1)}_{m,\varepsilon},x^{(1)}_{j,\varepsilon}),
 ~&\mbox{for}~m\neq j,~l=j,\\[1mm]
D_{x_i}\partial_hG(x^{(1)}_{l,\varepsilon},x^{(1)}_{j,\varepsilon}),
~&\mbox{for}~m= j,~l\neq j,\\[1mm]
0,~&\mbox{for}~l,m\neq j.
\end{cases}
\end{equation}
Then \eqref{luo24} and \eqref{luo41} imply
\begin{equation}\label{7-30-2}
\begin{split}
\sum^N_{h=1} B_{\varepsilon,j,h}&\Big(\frac{\partial^2_{x_ix_h}R(x^{(1)}_{j,\varepsilon})}
{(\lambda^{(1)}_{j,\varepsilon})^{(N-2)/2}}-\sum_{l\neq j}\frac{
 \partial^2_{x_i  x_h}G(x^{(1)}_{j,\varepsilon},x^{(1)}_{l,\varepsilon})}{{(\lambda^{(1)}_{l,\varepsilon})^{(N-2)/2}}}\Big)\\&
-
\sum^N_{h=1}\sum_{m\neq j}\frac{B_{\varepsilon,m,h}
D_{x_h}\partial_{x_i}G(x^{(1)}_{j,\varepsilon},x^{(1)}_{m,\varepsilon})}
{{(\lambda^{(1)}_{j,\varepsilon})^{(N-2)/2}}}=o\Big(\frac{1}{\bar{\lambda}_\varepsilon^{(3N-4)/2}}\Big).
\end{split}
\end{equation}
Since for any $i=\{1,\cdots,kN\}$, there exists some $j\in \{1,\cdots,k\}$ satisfying $i\in [(j-1)N+1,jN]\bigcap \N^+$, by direct calculations, we have
\begin{equation*}
\frac{\partial^2\Psi_k(x,\lambda)}{\partial y_i\partial y_m}=
\lambda_{j}^{N-2}
\frac{\partial^2R(x_j)}{\partial y_i\partial y_m}- \displaystyle\sum^k_{l\neq j}\lambda_{j}^{\frac{N-2}{2}}\lambda_{l}^{\frac{N-2}{2}}
\frac{\partial^2 G(x_j,x_l)}{\partial y_i\partial y_m},~\mbox{if}~m\in [(j-1)N+1,jN]\bigcap \N^+,
\end{equation*}
and
 \begin{equation*}
\frac{\partial^2\Psi_k(x,\lambda)}{\partial y_i\partial y_m}=
-\lambda_{j}^{\frac{N-2}{2}}\lambda_{s}^{\frac{N-2}{2}}
\frac{\partial^2 G(x_j,x_s)}{\partial y_i\partial y_m},~\mbox{if}~ m\in [(s-1)N+1,sN]\bigcap \N^+~\mbox{and}~s\neq j.
\end{equation*}
So from \eqref{7-30-2}, we can obtain
\begin{equation}\label{lll123}
\Big(D^2_{xx}\Psi_k(x,\lambda)\Big)_{(x,\lambda)=(a^k,\Lambda^k)}
\widetilde{B}_{\varepsilon,k}=\Big(o\big(\frac{1}{\bar{\lambda}_\varepsilon^{N-1}}\big),\cdots,
o\big(\frac{1}{\bar{\lambda}_\varepsilon^{N-1}}\big)\Big)^T,
\end{equation}
where
$\widetilde{B}_{\varepsilon,k}$ is the vector in \eqref{7-30-1}.
Noting that $(a^k, \Lambda^k)$ is a nondegenerate  critical point of $\Psi_{k}$, we see
\begin{equation}\label{adddtian62}
\mbox{Rank}~\Big(D^2_{(x,\lambda)x}\Psi_k(x,\lambda)\Big)_{(x,\lambda)=(a^k,\Lambda^k)}=Nk.
\end{equation}
Hence \eqref{dddtian62}, \eqref{lll123} and \eqref{adddtian62} imply  that
\begin{equation}\label{luo52}
B_{\varepsilon,j,h}=o\Big(\frac{1}{\bar{\lambda}_\varepsilon^{N-1}}\Big), ~\mbox{for}~ j=1,\cdots,k~\mbox{and}~
h=1,\cdots,N.
\end{equation}
On the other hand, from \eqref{cc7}, \eqref{cc8} and \eqref{luo51}, we find
 \begin{equation}\label{7-15-1}
\begin{split}
B_{\varepsilon,j,h}=&\int_{B_{\lambda^{(1)}_{j,\varepsilon}d}(0)} x_h C_\varepsilon(\frac{x}{\lambda^{(1)}_{j,\varepsilon}}+x^{(1)}_{j,\varepsilon})\xi_{\varepsilon,j}(x)dx
\\=&\frac{1}{\big(\lambda^{(1)}_{j,\varepsilon}\big)^{N-1}}\int_{\R^N} x_h U_{0,1}^{\frac{4}{N-2}}\Big(\sum^N_{l=1}c_{j,l} \psi_l(x)\Big)dx
+o\Big(\frac{1}{\bar{\lambda}_\varepsilon^{N-1}}\Big)
\\=&-\frac{N-2}{2N}\int_{\R^N}\frac{|x|^2}{(1+|x|^2)^{\frac{N}{2}}} \frac{c_{j,h}}{\big(\lambda^{(1)}_{j,\varepsilon}\big)^{N-1}}
+o\Big(\frac{1}{\bar{\lambda}_\varepsilon^{N-1}}\Big).\end{split}\end{equation}
Then \eqref{luo52} and \eqref{7-15-1} imply \eqref{luo--13}.
\end{proof}

We are now ready to show Theorem \ref{th1.1}.

\begin{proof}[\textbf{Proof of Theorem \ref{th1.1}}]
For any given $a^k=(a_1,\cdots,a_{k})$, since $M_k(a^k)$ is a positive matrix and
  $(a^k,\Lambda^k)$ is a nondegenerate critical point of $\Psi_k$, then from \cite{Musso1}, we find a  solution of \eqref{1.1} with \eqref{4-6-1}.
Next, we prove the local uniqueness of solutions to \eqref{1.1} with \eqref{4-6-1}.

From \eqref{cc6}, it holds
\begin{equation*}
|\xi_{\varepsilon}(x)|=O\Big(\frac{1}{R^2}\Big)+O(\varepsilon),~\mbox{for}~
x\in\Omega\backslash\bigcup_{j=1}^kB_{R(\lambda^{(1)}_{j,\varepsilon})^{-1}}(x_{j,\varepsilon}^{(1)}),
\end{equation*}
which implies that for any fixed $\gamma\in (0,1)$ and small $\varepsilon$, there exists $R_1>0$,
\begin{equation}\label{tian92}
|\xi_{\varepsilon}(x)|\leq \gamma,~ x\in\Omega\backslash\bigcup_{j=1}^kB_{R_1(\lambda^{(1)}_{j,\varepsilon})^{-1}}(x_{j,\varepsilon}^{(1)}).
\end{equation}
Also for the above fixed $R_1$, from \eqref{luo-13} and \eqref{luo--13}, we have
\begin{equation*}
\xi_{\varepsilon,j}(x)=o(1)~\mbox{in}~ B_{R_1}(0),~j=1,\cdots,k.
\end{equation*}
We know $\xi_{\varepsilon,j}(x)=\xi_{\varepsilon}(
\frac{x}{\lambda^{(1)}_{j,\varepsilon}}+x_{j,\varepsilon}^{(1)})$, so
\begin{equation}\label{tian91}
\xi_{\varepsilon}(x)=o(1),~x\in \bigcup_{j=1}^k B_{R_1(\lambda^{(1)}_{j,\varepsilon})^{-1}}(x_{j,\varepsilon}^{(1)}).
\end{equation}
Hence for any fixed $\gamma\in (0,1)$ and small $\varepsilon$, \eqref{tian92} and \eqref{tian91} imply
$|\xi_{\varepsilon}(x)|\leq \gamma$ for all $x\in \Omega$,
which is in contradiction with $\|\xi_{\varepsilon}\|_{L^{\infty}(\Omega)}=1$. Aa a result, $u^{(1)}_{\varepsilon}(x)\equiv u^{(2)}_{\varepsilon}(x)$ for small $\varepsilon$.
\end{proof}
\begin{Rem}\label{Rem-luo1}
Here we point out the reasons why  our methods are  unsuitable for $N=5$.
In fact, a first problem is that we can only get $\big|x^{(1)}_{j,\varepsilon}-x^{(2)}_{j,\varepsilon}\big|=O\Big(\frac{1}{\bar \lambda_{\varepsilon}}\Big)$ for $N=5$. However, to obtain
\begin{equation*}
\xi_{\varepsilon,j}(x)\to \sum_{i=0}^N c_{j,i}\psi_{i}(x),~\mbox{uniformly in}~C^1\big(B_R(0)\big) ~\mbox{for any}~R>0
\end{equation*}
as in Proposition \ref{prop3-2},
a necessary estimate is  $\big|x^{(1)}_{j,\varepsilon}-x^{(2)}_{j,\varepsilon}\big|=o\Big(\frac{1}{\bar \lambda_{\varepsilon}}\Big)$.
On the other hand, even if we would obtain  more precise estimate $\big|x^{(1)}_{j,\varepsilon}-x^{(2)}_{j,\varepsilon}\big|=o\Big(\frac{1}{\bar \lambda^2_{\varepsilon}}\Big)$ for $N=5$,
then similar to the estimate of \eqref{luo13}, we can find  $
A_{\varepsilon,j}=
O\Big(\frac{\ln \bar{\lambda}_{\varepsilon}}{\bar{\lambda}_{\varepsilon}^{4}}\Big)$,
 for $j=1,\cdots,k$,
which and \eqref{luo31} imply $c_{j,0}=0$, for $j=1,\cdots,k$.
But similar to the estimate of \eqref{luo52}, we can only find
 \begin{equation}\label{aaaluo52}
B_{\varepsilon,j,i}=O\Big(\frac{\ln \bar{\lambda}_\varepsilon}{\bar{\lambda}_\varepsilon^{4}}\Big), ~\mbox{for}~ j=1,\cdots,k~\mbox{and}~
i=1,\cdots,5.
\end{equation}
Then from \eqref{luo51} and \eqref{aaaluo52}, we can only get $c_{j,i}=O\big( \ln \bar{\lambda}_{\varepsilon}\big)$, for $j=1,\cdots,k$
 and $i=1,\cdots,5$.
 Why above phenomena occur is that the error estimate $w_\varepsilon$ is not enough for us in Proposition
 \ref{prop-A.2} when $N=5$. However
 the error estimate $w_\varepsilon$ in Proposition
 \ref{prop-A.2} is basic and cann't be improved.
\end{Rem}

Now we are in the position to show  Theorem \ref{th1.2}.

\begin{proof}[\textbf{Proof of Theorem \ref{th1.2}}]
First, \textbf{Assumption A} implies that the solution of \eqref{1.1} must blow up. In fact, if $u_\varepsilon(x)$ is a  solution of \eqref{1.1} which does not blow up, then letting $\varepsilon\rightarrow 0$, we can find a  nontrivial solution of \eqref{4-18-21}, which is a contradiction with  \textbf{Assumption A}.
Also \cite{Cerqueti} gives us that  all
blow-up points are isolated and it implies the uniform bound $\int_{\Omega}|\nabla u_\varepsilon|^2dx\leq C$, for some positive
constant $C$.  Now by the global compactness result in \cite{Struwe}, $u_\varepsilon(x)$ can be written as
\begin{equation*}
u_{\varepsilon}=u_0+\sum^k_{j=1} PU_{x_{j,\varepsilon}, \lambda_{j,\varepsilon}}+w_{\varepsilon}(x),
\end{equation*}
where
\begin{equation*}
 x _{j,\varepsilon}\rightarrow a_j,~\mbox{as}~\varepsilon\rightarrow  0, \lambda_{j,\varepsilon}\rightarrow \infty,~\|w_{\varepsilon}\|_\varepsilon=o(1),
\end{equation*}
and $u_0$ is a nonnegative solution of $-\Delta u= u^{\frac{N+2}{N-2}}$ in $\Omega$. By maximum principle and \textbf{Assumption A}, we get $u_0=0$. Then using Theorem \ref{th1.1}, the number of solutions to \eqref{1.1} with $k$ blow-up points is $\sharp T_k$. Since the number of the  blow-up points to \eqref{1.1} are finite,  we complete the proof of Theorem \ref{th1.2}.
\end{proof}
\section{Key estimates on Green's function}\label{s5}
In this section, we give  proofs of \eqref{luo2},  \eqref{luo1}, \eqref{bb1-1}  and \eqref{luo41} involving Green's function, which have been
used in sections \ref{s3} and \ref{s4}.
\begin{proof} [\underline{\textbf{Proof of \eqref{luo2}}}]
By the bilinearity of $P(u,v)$, we have
\begin{equation}\label{tian21}
\begin{split}
P\Big(G(x_{j,\varepsilon},x),G(x_{j,\varepsilon},x)\Big)=&
 P\Big(S(x_{j,\varepsilon},x),S(x_{j,\varepsilon},x)\Big)
 -2P\Big(S(x_{j,\varepsilon},x),H(x_{j,\varepsilon},x)\Big)\\&
 +P\Big(H(x_{j,\varepsilon},x),H(x_{j,\varepsilon},x)\Big).
 \end{split}
\end{equation}
After direct calculations, we know
\begin{equation}\label{laa}
\begin{split}
D_{x_i} S(x_{j,\varepsilon},x) =-\frac{x_i-x_{j,\varepsilon,i}}{\omega_N|x_{j,\varepsilon}-x|^{N}},~\mbox{and}~\nu_i=
\frac{x_i-x_{j,\varepsilon,i}}{|x_{j,\varepsilon}-x|}.
 \end{split}
\end{equation}
Putting \eqref{laa} in the term $P\Big(S(x_{j,\varepsilon},x),S(x_{j,\varepsilon},x)\Big)$, we get
\begin{equation}\label{tian22}
P\Big(S(x_{j,\varepsilon},x),S(x_{j,\varepsilon},x)\Big)=0.
\end{equation}
Now we calculate $P\Big(S(x_{j,\varepsilon},x),H(x_{j,\varepsilon},x)\Big)$.  Since $D_{\nu} H(x_{j,\varepsilon},x)$   is bounded in $B_d(x_{j,\varepsilon})$, we know
\begin{equation}\label{tian23}
\begin{split}
\theta\int_{\partial B_\theta(x_{j,\varepsilon})}& \big\langle D S(x_{j,\varepsilon},x),\nu\big\rangle \big\langle
D H(x_{j,\varepsilon},x),\nu \big\rangle=O\Big(\theta\int_{\partial B_\theta(x_{j,\varepsilon})}|x-x_{j,\varepsilon}|^{-(N-1)}\Big)=O(\theta),
\end{split}
\end{equation}
\begin{equation}\label{tian24}
\begin{split}
\theta\int_{\partial B_\theta(x_{j,\varepsilon})}& \big\langle D S(x_{j,\varepsilon},x),
D H(x_{j,\varepsilon},x) \big\rangle=O\Big(\theta\int_{\partial B_\theta(x_{j,\varepsilon})}|x-x_{j,\varepsilon}|^{-(N-1)}\Big)=O(\theta),
\end{split}
\end{equation}
and
\begin{equation}\label{tian25}
\begin{split}
 \int_{\partial B_\theta(x_{j,\varepsilon})}& \big\langle D H(x_{j,\varepsilon},x),\nu\big\rangle   S(x_{j,\varepsilon},x)  =O\Big( \int_{\partial B_\theta(x_{j,\varepsilon})}|x-x_{j,\varepsilon}|^{-(N-2)}\Big)=O(\theta).
\end{split}
\end{equation}
Next, we obtain
\begin{equation}\label{tian26}
\begin{split}
 \int_{\partial B_\theta(x_{j,\varepsilon})}& \big\langle D S(x_{j,\varepsilon},x),\nu\big\rangle   H(x_{j,\varepsilon},x)\\=&
  H(x_{j,\varepsilon},x_{j,\varepsilon})\int_{\partial B_\theta(x_{j,\varepsilon})} \big\langle D S(x_{j,\varepsilon},x),\nu\big\rangle +O(\theta)= -R(x_{j,\varepsilon}) +O(\theta).
\end{split}
\end{equation}
Then from \eqref{tian23}--\eqref{tian26}, we get
\begin{equation}\label{tian27}
P\Big(S(x_{j,\varepsilon},x),H(x_{j,\varepsilon},x)\Big)=\frac{N-2}{4}R(x_{j,\varepsilon})
+O(\theta).
\end{equation}
Also since $H(x_{j,\varepsilon},x)$ and $D_{\nu} H(x_{j,\varepsilon},x)$ are bounded in $B_d(x_{j,\varepsilon})$, it holds that
\begin{equation}\label{tian28}
P\Big(H(x_{j,\varepsilon},x),H(x_{j,\varepsilon},x)\Big)=O(\theta^{N-1}).
\end{equation}
Letting $\theta\rightarrow 0$, from \eqref{tian21}, \eqref{tian22}, \eqref{tian27} and \eqref{tian28}, we get
\begin{equation*}
\begin{split}
P\Big(G(x_{j,\varepsilon},x),G(x_{j,\varepsilon},x)\Big)=-\frac{N-2}{2}R(x_{j,\varepsilon}).
 \end{split}
\end{equation*}
Next, for $m\neq j$,
\begin{equation}\label{tian29}
\begin{split}
P\Big(G(x_{j,\varepsilon},x),G(x_{m,\varepsilon},x)\Big)=&
P\Big(S(x_{j,\varepsilon},x),G(x_{m,\varepsilon},x)\Big)
 - P\Big(H(x_{j,\varepsilon},x),G(x_{m,\varepsilon},x)\Big).
 \end{split}
\end{equation}
Since $D_{\nu} G(x_{m,\varepsilon},x)$   is bounded in $B_d(x_{j,\varepsilon})$, we know
\begin{equation}\label{tian30}
\begin{split}
\theta\int_{\partial B_\theta(x_{j,\varepsilon})}& \big\langle D S(x_{j,\varepsilon},x),\nu \big\rangle \big\langle
D G(x_{m,\varepsilon},x),\nu  \big\rangle=O\Big(\theta\int_{\partial B_\theta(x_{j,\varepsilon})}|x-x_{j,\varepsilon}|^{-(N-1)}\Big)=O \big(\theta \big),
\end{split}
\end{equation}
\begin{equation}\label{tian31}
\begin{split}
\theta\int_{\partial B_\theta(x_{j,\varepsilon})}&  \big\langle D S(x_{j,\varepsilon},x),
D G(x_{m,\varepsilon},x)  \big\rangle=O\Big(\theta\int_{\partial B_\theta(x_{j,\varepsilon})}|x-x_{j,\varepsilon}|^{-(N-1)}\Big)=O \big(\theta \big),
\end{split}
\end{equation}
\begin{equation}\label{tian32}
\begin{split}
 \int_{\partial B_\theta(x_{j,\varepsilon})}&  \big\langle D G(x_{m,\varepsilon},x),\nu \big\rangle   S(x_{j,\varepsilon},x)  =O\Big( \int_{\partial B_\theta(x_{j,\varepsilon})}|x-x_{j,\varepsilon}|^{-(N-2)}\Big)=O\big(\theta\big),
\end{split}
\end{equation}
and
\begin{equation}\label{tian33}
\begin{split}
\int_{\partial B_\theta(x_{j,\varepsilon})}& \big\langle D S(x_{j,\varepsilon},x),\nu \big\rangle   G(x_{m,\varepsilon},x)  \\
=&  G \big(x_{m,\varepsilon},x_{j,\varepsilon} \big) \int_{\partial B_\theta(x_{j,\varepsilon})}D_{\nu} S(x_{j,\varepsilon},x) +O \big(\theta \big)\int_{\partial B_\theta(x_{j,\varepsilon})} \big|D_{\nu} S(x_{j,\varepsilon},x)  \big|\\
=&- G \big(x_{m,\varepsilon},x_{j,\varepsilon} \big) +O \big(\theta \big).
\end{split}
\end{equation}
The combination of  \eqref{tian30}--\eqref{tian33} gives
\begin{equation}\label{tian34}
P\Big(S(x_{j,\varepsilon},x),G(x_{m,\varepsilon},x)\Big)=\frac{N-2}{4}G \big(x_{m,\varepsilon},x_{j,\varepsilon} \big)+O \big(\theta \big).
\end{equation}
Also since $H(x_{j,\varepsilon},x)$, $D_{\nu} H(x_{j,\varepsilon},x)$, $G(x_{m,\varepsilon},x)$ and $D_{\nu}G(x_{m,\varepsilon},x)$ are bounded in $B_d(x_{j,\varepsilon})$, it holds
\begin{equation}\label{tian35}
P\Big(H(x_{j,\varepsilon},x),G(x_{m,\varepsilon},x)\Big)=O \big(\theta^{N-1} \big).
\end{equation}
Letting $\theta\rightarrow 0$, from \eqref{tian29}, \eqref{tian34} and \eqref{tian35}, we know
\begin{equation}\label{la-1}
P\Big(G(x_{j,\varepsilon},x),G(x_{m,\varepsilon},x)\Big)= \frac{N-2}{4}G \big(x_{m,\varepsilon},x_{j,\varepsilon} \big),~\mbox{for}~m\neq j.
\end{equation}
By the symmetry of $P(u,v)$, \eqref{la-1} implies
\begin{equation*}
P\Big(G(x_{m,\varepsilon},x),G(x_{j,\varepsilon},x)\Big)= \frac{N-2}{4}G \big(x_{m,\varepsilon},x_{j,\varepsilon} \big),~\mbox{for}~m\neq j.
\end{equation*}

Finally, for $l,m\neq j$, the boundedness of $G(x_{m,\varepsilon},x)$, $D_{\nu}G(x_{m,\varepsilon},x)$, $G(x_{l,\varepsilon},x)$ and $D_{\nu}G(x_{l,\varepsilon},x)$ are bounded in $B_d(x_{j,\varepsilon})$ yields that $P\Big(G(x_{l,\varepsilon},x),G(x_{m,\varepsilon},x)\Big)=O \big(\theta^{N-1} \big)$.
Letting $\theta\rightarrow 0$, we know
\begin{equation*}
P\Big(G(x_{l,\varepsilon},x),G(x_{m,\varepsilon},x)\Big)=0,~\mbox{for}~l,m\neq j.
\end{equation*}
\end{proof}

\begin{proof} [\underline{\textbf{Proof of \eqref{luo1}}}]
First, by the bilinearity of $Q(u,v)$, we have
\begin{equation}\label{tian123}
\begin{split}
Q \Big(G(x_{j,\varepsilon},x),G(x_{j,\varepsilon},x)\Big)=&
 Q\Big(S(x_{j,\varepsilon},x),S(x_{j,\varepsilon},x)\Big)
 -2Q\Big(S(x_{j,\varepsilon},x),H(x_{j,\varepsilon},x)\Big)\\&
 +Q\Big(H(x_{j,\varepsilon},x),H(x_{j,\varepsilon},x)\Big).
 \end{split}
\end{equation}
Then for $Q\Big(S(x_{j,\varepsilon},x),S(x_{j,\varepsilon},x)\Big)$, the oddness of integrands  means
\begin{equation}\label{tian-7}
Q\Big(S(x_{j,\varepsilon},x),S(x_{j,\varepsilon},x)\Big)=0.
\end{equation}
Now we calculate the term $Q\Big(S(x_{j,\varepsilon},x),H(x_{j,\varepsilon},x)\Big)$. First, we know
\begin{equation}\label{tian-4}
\begin{split}
\int_{\partial B_\theta(x_{j,\varepsilon})}& D_{\nu} S(x_{j,\varepsilon},x)
D_{x_i}H(x_{j,\varepsilon},x) \\
=&D_{x_i} H \big(x_{j,\varepsilon},x_{j,\varepsilon} \big) \int_{\partial B_\theta(x_{j,\varepsilon})}D_{\nu} S(x_{j,\varepsilon},x) +O \big(\theta \big)\int_{\partial B_\theta(x_{j,\varepsilon})} \big|D_{\nu} S(x_{j,\varepsilon},x) \big|\\
=&-\frac{1}{2}\frac{\partial R(x_{j,\varepsilon})} {\partial x_i}+O \big(\theta \big),
\end{split}
\end{equation}
\begin{equation}\label{tian-5}
\begin{split}
\int_{\partial B_\theta(x_{j,\varepsilon})}&D_{\nu} H(x_{j,\varepsilon},x)
D_{x_i} S(x_{j,\varepsilon},x) \\
=&\sum^N_{l=1} D_{x_l} H \big(x_{j,\varepsilon},x_{j,\varepsilon} \big)\int_{\partial B_\theta(x_{j,\varepsilon})}D_{x_i} S(x_{j,\varepsilon},x) \nu_l+O \big(\theta \big)\int_{\partial B_\theta(x_{j,\varepsilon})} \big|D_{x_i} S(x_{j,\varepsilon},x) \big|\\
=&D_{x_i} H \big(x_{j,\varepsilon},x_{j,\varepsilon} \big) \int_{\partial B_\theta(x_{j,\varepsilon})}D_{x_i} S(x_{j,\varepsilon},x) \nu_i+O \big(\theta \big),
\end{split}
\end{equation}
and
\begin{equation}\label{tian-6}
\begin{split}
\int_{\partial B_\theta(x_{j,\varepsilon})}& \big\langle D S(x_{j,\varepsilon},x),D H(x_{j,\varepsilon},x) \big\rangle \nu_i\\
=&\sum^N_{l=1} D_{x_l} H \big(x_{j,\varepsilon},x_{j,\varepsilon} \big) \int_{\partial B_\theta(x_{j,\varepsilon})}D_{x_l} S(x_{j,\varepsilon},x) \nu_i+O \big(\theta \big)\int_{\partial B_\theta(x_{j,\varepsilon})} \big|D S(x_{j,\varepsilon},x) \big| \\
=&D_{x_i} H \big(x_{j,\varepsilon},x_{j,\varepsilon} \big) \int_{\partial B_\theta(x_{j,\varepsilon})}D_{x_i} S(x_{j,\varepsilon},x) \nu_i+O \big(\theta \big),
\end{split}
\end{equation}
which together imply
\begin{equation}\label{tian-8}
Q\Big(S(x_{j,\varepsilon},x),H(x_{j,\varepsilon},x)\Big)= \frac{1}{2}\frac{\partial R(x_{j,\varepsilon})} {\partial x_i}+O \big(\theta \big).
\end{equation}
Also since $D_{\nu} H(x_{j,\varepsilon},x)$ is bounded in $B_d(x_{j,\varepsilon})$, it holds
\begin{equation}\label{tian-9}
Q\Big(H(x_{j,\varepsilon},x),H(x_{j,\varepsilon},x)\Big)=O \big(\theta^{N-1} \big).
\end{equation}
Letting $\theta\rightarrow 0$, from \eqref{tian123}, \eqref{tian-7}, \eqref{tian-8} and \eqref{tian-9},  we get
\begin{equation*}
\begin{split}
Q\Big(G(x_{j,\varepsilon},x),G(x_{j,\varepsilon},x)\Big)=-\frac{\partial R(x_{j,\varepsilon})}{\partial x_i}.
 \end{split}
\end{equation*}
Next, for $m\neq j$,
\begin{equation}\label{tian1234}
\begin{split}
Q\Big(G(x_{j,\varepsilon},x),G(x_{m,\varepsilon},x)\Big)=&
 Q\Big(S(x_{j,\varepsilon},x),G(x_{m,\varepsilon},x)\Big)
 - Q\Big(H(x_{j,\varepsilon},x),G(x_{m,\varepsilon},x)\Big).
 \end{split}
\end{equation}
Similar to  \eqref{tian-4}--\eqref{tian-6}, we know
\begin{equation*}
\begin{split}
\int_{\partial B_\theta(x_{j,\varepsilon})}&D_{\nu} S(x_{j,\varepsilon},x)
D_{x_i} G(x_{m,\varepsilon},x) \\
=&D_{x_i} G\big(x_{m,\varepsilon},x_{j,\varepsilon}\big) \int_{\partial B_\theta(x_{j,\varepsilon})}D_{\nu} S(x_{j,\varepsilon},x) +O\big(\theta\big)\int_{\partial B_\theta(x_{j,\varepsilon})}|D_{\nu} S(x_{j,\varepsilon},x) |\\
=&-D_{x_i} G\big(x_{m,\varepsilon},x_{j,\varepsilon}\big) +O\big(\theta\big),
\end{split}
\end{equation*}
\begin{equation*}
\begin{split}
\int_{\partial B_\theta(x_{j,\varepsilon})}&D_{\nu}G(x_{m,\varepsilon},x)
D_{x_i}S(x_{j,\varepsilon},x) \\
=&\sum^N_{l=1} D_{x_l} G\big(x_{m,\varepsilon},x_{j,\varepsilon}\big) \int_{\partial B_\theta(x_{j,\varepsilon})}D_{x_i} S(x_{j,\varepsilon},x) \nu_l+O\big(\theta\big)\int_{\partial B_\theta(x_{j,\varepsilon})}\big|D_{x_i} S(x_{j,\varepsilon},x) \big| \\
=&D_{x_i}G\big(x_{m,\varepsilon},x_{j,\varepsilon}\big) \int_{\partial B_\theta(x_{j,\varepsilon})}D_{x_i} S(x_{j,\varepsilon},x) \nu_i+O\big(\theta\big),
\end{split}
\end{equation*}
and
\begin{equation*}
\begin{split}
\int_{\partial B_\theta(x_{j,\varepsilon})}&\big\langle \nabla S(x_{j,\varepsilon},x),\nabla H(x_{j,\varepsilon},x) \big\rangle \nu_i\\
=&\sum^N_{l=1}  D_{x_l} G\big(x_{m,\varepsilon},x_{j,\varepsilon}\big) \int_{\partial B_\theta(x_{j,\varepsilon})}D_{x_l} S(x_{j,\varepsilon},x) \nu_i+O\Big(\theta \int_{\partial B_\theta(x_{j,\varepsilon})}\big|D_{\nu} S(x_{j,\varepsilon},x) \big|\Big)\\
=& D_{x_i} G\big(x_{m,\varepsilon},x_{j,\varepsilon}\big) \int_{\partial B_\theta(x_{j,\varepsilon})}D_{x_i} S(x_{j,\varepsilon},x) +O\big(\theta\big),
\end{split}
\end{equation*}
which together imply
\begin{equation}\label{tian-13}
Q\Big(S(x_{j,\varepsilon},x),G(x_{m,\varepsilon},x)\Big)= D_{x_i} G\big(x_{m,\varepsilon},x_{j,\varepsilon}\big)+O\big(\theta\big).
\end{equation}
Since $D_{\nu} H(x_{j,\varepsilon},x)$ and $D_{\nu}G(x_{m,\varepsilon},x)$ are bounded in $B_d(x_{j,\varepsilon})$, it holds
\begin{equation}\label{tian-14}
Q\Big(H(x_{j,\varepsilon},x),G(x_{m,\varepsilon},x)\Big)=O\big(\theta^{N-1}\big).
\end{equation}
Letting $\theta\rightarrow 0$, from \eqref{tian1234}--\eqref{tian-14}, we know
\begin{equation}\label{la1}
Q\Big(G(x_{j,\varepsilon},x),G(x_{m,\varepsilon},x)\Big)= D_{x_i} G\big(x_{m,\varepsilon},x_{j,\varepsilon}\big),~\mbox{for}~m\neq j.
\end{equation}
By the symmetry of $Q(u,v)$, \eqref{la1} imply
$$
Q\Big(G(x_{m,\varepsilon},x),G(x_{j,\varepsilon},x)\Big)= D_{x_i} G\big(x_{m,\varepsilon},x_{j,\varepsilon}\big), ~\mbox{for}~m\neq j.$$

Finally, since $D_{\nu}G(x_{l,\varepsilon},x)$ and $D_{\nu}G(x_{m,\varepsilon},x)$ are bounded in $B_d(x_{j,\varepsilon})$ for $l,m\neq j$, it holds that  $
Q\Big(G(x_{l,\varepsilon},x),G(x_{m,\varepsilon},x)\Big)=O\big(\theta^{N-1}\big)$.
So letting $\theta\rightarrow 0$, we know
\begin{equation*}
Q\Big(G(x_{l,\varepsilon},x),G(x_{m,\varepsilon},x)\Big)=0,~\mbox{for}~l,m\neq j.
\end{equation*}
\end{proof}

\begin{proof} [\underline{\textbf{Proof of \eqref{bb1-1}}}]
By the bilinearity of $P(u,v)$, we have
\begin{equation}\label{ddtian21}
\begin{split}
P_1\Big(G(x^{(1)}_{j,\varepsilon},x),\partial_hG(x^{(1)}_{j,\varepsilon},x)\Big)=&
 P_1\Big(S(x^{(1)}_{j,\varepsilon},x),\partial_hS(x^{(1)}_{j,\varepsilon},x)\Big)
 - P_1\Big(S(x^{(1)}_{j,\varepsilon},x),\partial_h H(x^{(1)}_{j,\varepsilon},x)\Big)\\&
 - P_1\Big(H(x^{(1)}_{j,\varepsilon},x),\partial_h S(x^{(1)}_{j,\varepsilon},x)\Big)+P_1\Big(H(x^{(1)}_{j,\varepsilon},x),\partial_h H(x^{(1)}_{j,\varepsilon},x)\Big).
 \end{split}
\end{equation}
For $P_1\Big(S(x^{(1)}_{j,\varepsilon},x),\partial_hS(x^{(1)}_{j,\varepsilon},x)\Big)$, the oddness of the integrands yields
\begin{equation}\label{ddtian22}
 P_1\Big(S(x^{(1)}_{j,\varepsilon},x),\partial_hS(x^{(1)}_{j,\varepsilon},x)\Big)=0.
\end{equation}
Now we calculate $P_1\Big(S(x^{(1)}_{j,\varepsilon},x),\partial_h H(x^{(1)}_{j,\varepsilon},x)\Big)$.  Since $\partial_hD_{\nu} H(x^{(1)}_{j,\varepsilon},x)$   is bounded in $B_d(x^{(1)}_{j,\varepsilon})$, we know
\begin{equation}\label{ddtian23}
\begin{split}
\theta\int_{\partial B_\theta(x^{(1)}_{j,\varepsilon})}& \big\langle D S(x^{(1)}_{j,\varepsilon},x),\nu\big\rangle \big\langle
\partial_h D H(x^{(1)}_{j,\varepsilon},x),\nu \big\rangle=O\Big(\theta\int_{\partial B_\theta(x^{(1)}_{j,\varepsilon})}|x-x^{(1)}_{j,\varepsilon}|^{-(N-1)}\Big)=O(\theta),
\end{split}
\end{equation}
\begin{equation}\label{ddtian24}
\begin{split}
\theta\int_{\partial B_\theta(x^{(1)}_{j,\varepsilon})}& \big\langle D S(x^{(1)}_{j,\varepsilon},x),
\partial_h D H(x^{(1)}_{j,\varepsilon},x) \big\rangle=O\Big(\theta\int_{\partial B_\theta(x^{(1)}_{j,\varepsilon})}|x-x^{(1)}_{j,\varepsilon}|^{-(N-1)}\Big)=O(\theta),
\end{split}
\end{equation}
and
\begin{equation}\label{ddtian25}
\begin{split}
 \int_{\partial B_\theta(x^{(1)}_{j,\varepsilon})}& \big\langle \partial_h D H(x^{(1)}_{j,\varepsilon},x),\nu\big\rangle   S(x^{(1)}_{j,\varepsilon},x)  =O\Big( \int_{\partial B_\theta(x^{(1)}_{j,\varepsilon})}|x-x^{(1)}_{j,\varepsilon}|^{-(N-2)}\Big)=O(\theta).
\end{split}
\end{equation}
Next, we obtain
\begin{equation}\label{ddtian26}
\begin{split}
 \int_{\partial B_\theta(x^{(1)}_{j,\varepsilon})}& \big\langle D S(x^{(1)}_{j,\varepsilon},x),\nu\big\rangle   \partial_h H(x^{(1)}_{j,\varepsilon},x)\\=&
   \partial_h  H(x^{(1)}_{j,\varepsilon},x^{(1)}_{j,\varepsilon})\int_{\partial B_\theta(x^{(1)}_{j,\varepsilon})} \big\langle D S(x^{(1)}_{j,\varepsilon},x),\nu\big\rangle +O(\theta)= -\frac{1}{2}  \partial_hR(x^{(1)}_{j,\varepsilon}) +O(\theta).
\end{split}
\end{equation}
Then from \eqref{ddtian23}--\eqref{ddtian26}, we get
\begin{equation}\label{ddtian27}
P_1\Big(S(x^{(1)}_{j,\varepsilon},x),H(x^{(1)}_{j,\varepsilon},x)\Big)=\frac{N-2}{8}\partial_hR(x^{(1)}_{j,\varepsilon})
+O(\theta).
\end{equation}
Next, we calculate $P_1\Big(H(x^{(1)}_{j,\varepsilon},x),\partial_h S(x^{(1)}_{j,\varepsilon},x)\Big)$.
First, let  $y=x-x^{(1)}_{j,\varepsilon}$, then we get
\begin{equation}\label{8-17-1}
\partial_{h} S\big(x^{(1)}_{j,\varepsilon},x\big)=-\frac{y_h}{\omega_N|y|^{N}}, ~
\big\langle D \partial_{h} S\big(x^{(1)}_{j,\varepsilon},x\big),\nu \big\rangle =\frac{(1-N)y_h}{\omega_N|y|^{N+1}},
\end{equation}
and
\begin{equation}\label{8-17-2}
D_{x_l}\partial_{h} S\big(x^{(1)}_{j,\varepsilon},x\big) =\frac{\delta_{hl}}{\omega_N|y|^{N}}-
\frac{N y_hy_l}
{\omega_N|y|^{N+2}}.
\end{equation}
Then we know
\begin{equation*}
\begin{split}
\theta\int_{\partial B_\theta(x^{(1)}_{j,\varepsilon})}& \big\langle \partial_h D S(x^{(1)}_{j,\varepsilon},x),\nu\big\rangle \big\langle
 D H(x^{(1)}_{j,\varepsilon},x),\nu \big\rangle \\
=&\theta D_{h} H\big(x^{(1)}_{j,\varepsilon},x^{(1)}_{j,\varepsilon}\big) \int_{\partial B_\theta(0)}\frac{(1-N)y_h^2}{\omega_N|y|^{N+2}} +O\big(\theta\big)=\frac{1-N}{2N} \partial _h R\big(x^{(1)}_{j,\varepsilon}\big)+O\big(\theta\big),
\end{split}
\end{equation*}
\begin{equation*}
\begin{split}
\theta\int_{\partial B_\theta(x^{(1)}_{j,\varepsilon})}& \big\langle \partial_h  D S(x^{(1)}_{j,\varepsilon},x),
D H(x^{(1)}_{j,\varepsilon},x) \big\rangle \\
=& \theta D_{h}H\big(x^{(1)}_{j,\varepsilon},x^{(1)}_{j,\varepsilon}\big)\int_{\partial B_\theta(0)}
\big(\frac{1}{\omega_N|y|^{N}}-
\frac{N y^2_h}
{\omega_N|y|^{N+2}}\big)
+O\big(\theta\big)=O\big(\theta\big),
\end{split}
\end{equation*}
\begin{equation*}
\begin{split}
 \int_{\partial B_\theta(x^{(1)}_{j,\varepsilon})}& \big\langle D H(x^{(1)}_{j,\varepsilon},x),\nu\big\rangle   \partial_h   S(x^{(1)}_{j,\varepsilon},x)  \\
=& -\partial _h H\big(x^{(1)}_{j,\varepsilon},x^{(1)}_{j,\varepsilon}\big)
 \int_{\partial B_\theta(0)} \frac{y^2_h}{\omega_N|y|^{N+1}}+O\big(\theta\big)
 =-\frac{1}{2N} \partial _h R\big(x^{(1)}_{j,\varepsilon}\big)+O\big(\theta\big),
\end{split}
\end{equation*}
and
\begin{equation*}
\begin{split}
 \int_{\partial B_\theta(x^{(1)}_{j,\varepsilon})}& \big\langle \partial_h D S(x^{(1)}_{j,\varepsilon},x),\nu\big\rangle   H(x^{(1)}_{j,\varepsilon},x)\\=&
   \partial_h  H(x^{(1)}_{j,\varepsilon},x_{j,\varepsilon})\int_{\partial B_\theta(x^{(1)}_{j,\varepsilon})}\frac{(1-N)y^2_h}{\omega_N|y|^{N+1}} +O(\theta)= \frac{1-N}{2N}  \partial_hR(x^{(1)}_{j,\varepsilon})+O(\theta),
\end{split}
\end{equation*}
which together imply
\begin{equation}\label{ddll-2}
P_1\Big(H\big(x^{(1)}_{j,\varepsilon},x\big),\partial_h S\big(x^{(1)}_{j,\varepsilon},x\big)\Big)
=\big(\frac{N-2}{8}+\frac{1-N}{2N}\big)\partial _h R\big(x^{(1)}_{j,\varepsilon}\big)+O\big(\theta\big).
\end{equation}
Also since $H(x^{(1)}_{j,\varepsilon},x)$ and $D_{\nu} H(x^{(1)}_{j,\varepsilon},x)$ are bounded in $B_d(x^{(1)}_{j,\varepsilon})$, it holds that
\begin{equation}\label{ddtian28}
P_1\Big(H(x^{(1)}_{j,\varepsilon},x),\partial_h H(x^{(1)}_{j,\varepsilon},x)\Big)=O(\theta^{N-1}).
\end{equation}
Letting $\theta\rightarrow 0$, from \eqref{ddtian21}, \eqref{ddtian22}, \eqref{ddtian27},
\eqref{ddll-2} and \eqref{ddtian28}, we get
\begin{equation*}
\begin{split}
P_1\Big(G(x^{(1)}_{j,\varepsilon},x),\partial_h G(x^{(1)}_{j,\varepsilon},x)\Big)=\Big(\frac{N-2}{4}+\frac{1-N}{2N}\Big)\partial _h R\big(x^{(1)}_{j,\varepsilon}\big).
 \end{split}
\end{equation*}
Next, for $m\neq j$,
\begin{equation}\label{ddtian29}
\begin{split}
P_1\Big(G(x^{(1)}_{j,\varepsilon},x),\partial_hG(x^{(1)}_{m,\varepsilon},x)\Big)=&
P_1\Big(S(x^{(1)}_{j,\varepsilon},x),\partial_hG(x^{(1)}_{m,\varepsilon},x)\Big)
 -P_1\Big(H(x^{(1)}_{j,\varepsilon},x),\partial_hG(x^{(1)}_{m,\varepsilon},x)\Big).
 \end{split}
\end{equation}
Since $\partial_hD_{\nu} G(x^{(1)}_{m,\varepsilon},x)$   is bounded in $B_d(x^{(1)}_{j,\varepsilon})$, we know
\begin{equation}\label{ddtian30}
\begin{split}
\theta\int_{\partial B_\theta(x^{(1)}_{j,\varepsilon})}& \big\langle D S(x^{(1)}_{j,\varepsilon},x),\nu \big\rangle \big\langle
\partial_hD G(x^{(1)}_{m,\varepsilon},x),\nu  \big\rangle=O\Big(\theta\int_{\partial B_\theta(0)}\frac{1}{|y|^{N-1}}\Big)=O \big(\theta \big),
\end{split}
\end{equation}
\begin{equation}\label{ddtian31}
\begin{split}
\theta\int_{\partial B_\theta(x^{(1)}_{j,\varepsilon})}&  \big\langle D S(x^{(1)}_{j,\varepsilon},x),
\partial_hD G(x^{(1)}_{m,\varepsilon},x)  \big\rangle=O\Big(\theta\int_{\partial B_\theta(0)}\frac{1}{|y|^{N-1}}\Big)=O \big(\theta \big),
\end{split}
\end{equation}
\begin{equation}\label{ddtian32}
\begin{split}
 \int_{\partial B_\theta(x^{(1)}_{j,\varepsilon})}&  \big\langle \partial_hD G(x^{(1)}_{m,\varepsilon},x),\nu \big\rangle   S(x^{(1)}_{j,\varepsilon},x)  =O\Big( \int_{\partial B_\theta(0)}\frac{1}{|y|^{N-2}}\Big)=O\big(\theta\big),
\end{split}
\end{equation}
and
\begin{equation}\label{ddtian33}
\begin{split}
\int_{\partial B_\theta(x^{(1)}_{j,\varepsilon})}& \big\langle D S(x^{(1)}_{j,\varepsilon},x),\nu \big\rangle   \partial_hG(x^{(1)}_{m,\varepsilon},x)  \\
=&  \partial_hG \big(x^{(1)}_{m,\varepsilon},x^{(1)}_{j,\varepsilon} \big) \int_{\partial B_\theta(x^{(1)}_{j,\varepsilon})}D_{\nu} S(x^{(1)}_{j,\varepsilon},x) +O \big(\theta \big)=-\partial_hG \big(x^{(1)}_{m,\varepsilon},x^{(1)}_{j,\varepsilon} \big) +O \big(\theta \big).
\end{split}
\end{equation}
From \eqref{ddtian30}--\eqref{ddtian33}, we get
\begin{equation}\label{ddtian34}
P_1\Big(S(x^{(1)}_{j,\varepsilon},x),G(x^{(1)}_{m,\varepsilon},x)\Big)=\frac{N-2}{4}\partial_hG \big(x^{(1)}_{m,\varepsilon},x^{(1)}_{j,\varepsilon} \big)+O \big(\theta \big).
\end{equation}
Also since $H(x^{(1)}_{j,\varepsilon},x)$, $D_{\nu} H(x^{(1)}_{j,\varepsilon},x)$, $\partial_hG(x^{(1)}_{m,\varepsilon},x)$ and $\partial_hD_{\nu}G(x^{(1)}_{m,\varepsilon},x)$ are bounded in $B_d(x^{(1)}_{j,\varepsilon})$, it holds that
\begin{equation}\label{ddtian35}
P_1\Big(H(x^{(1)}_{j,\varepsilon},x),G(x^{(1)}_{m,\varepsilon},x)\Big)=O \big(\theta^{N-1} \big).
\end{equation}
Letting $\theta\rightarrow 0$, from \eqref{ddtian29}, \eqref{ddtian34} and \eqref{ddtian35}, we know
\begin{equation*}
P_1\Big(G(x^{(1)}_{j,\varepsilon},x),\partial_hG(x^{(1)}_{m,\varepsilon},x)\Big)= \frac{N-2}{4}\partial_hG \big(x^{(1)}_{m,\varepsilon},x^{(1)}_{j,\varepsilon} \big),~\mbox{for}~m\neq j.
\end{equation*}
Next, we calculate  the term $
P_1\Big(G(x^{(1)}_{m,\varepsilon},x),\partial_hG(x^{(1)}_{j,\varepsilon},x)\Big)$. Similar to the estimate of \eqref{ddll-2}, we find
\begin{equation*}
P_1\Big(G(x^{(1)}_{m,\varepsilon},x),\partial_hG(x^{(1)}_{j,\varepsilon},x)\Big)=
\big(\frac{N-2}{4}+\frac{1-N}{N}\big)\partial _hG \big(x^{(1)}_{j,\varepsilon},x^{(1)}_{m,\varepsilon} \big),~\mbox{for}~m\neq j.
\end{equation*}

Finally, for $l,m\neq j$, since $G(x^{(1)}_{m,\varepsilon},x)$, $D_{\nu}G(x^{(1)}_{m,\varepsilon},x)$, $\partial_hG(x^{(1)}_{l,\varepsilon},x)$ and $\partial_hD_{\nu}G(x^{(1)}_{l,\varepsilon},x)$ are bounded in $B_d(x^{(1)}_{j,\varepsilon})$, we conclude  $P_1\Big(G(x^{(1)}_{l,\varepsilon},x),\partial_hG(x^{(1)}_{m,\varepsilon},x)\Big)=O \big(\theta^{N-1} \big)$.
So letting $\theta\rightarrow 0$, we know
\begin{equation*}
P_1\Big(G(x^{(1)}_{l,\varepsilon},x),\partial_hG(x^{(1)}_{m,\varepsilon},x)\Big)=0,~\mbox{for}~l,m\neq j.
\end{equation*}
\end{proof}
\begin{proof} [\underline{\textbf{Proof of \eqref{luo41}}}]
By the bilinearity of $Q_1(u,v)$, we have

\begin{small}
\begin{equation}\label{tian71}
\begin{split}
Q_1\Big(G(x^{(1)}_{j,\varepsilon},x),\partial _hG(x^{(1)}_{j,\varepsilon},x)\Big)=&
Q_1\Big(S(x^{(1)}_{j,\varepsilon},x),\partial _h S(x^{(1)}_{j,\varepsilon},x)\Big)
  - Q_1\Big(H(x^{(1)}_{j,\varepsilon},x),\partial _hS
  (x^{(1)}_{j,\varepsilon},x)\Big)  \\&- Q_1\Big(S(x^{(1)}_{j,\varepsilon},x),\partial _hH(x^{(1)}_{j,\varepsilon},x)\Big)+Q_1\Big(H(x^{(1)}_{j,\varepsilon},x),\partial_h H(x^{(1)}_{j,\varepsilon},x)\Big).
 \end{split}
\end{equation}
\end{small}
Then for $Q_1\Big(S(x^{(1)}_{j,\varepsilon},x),\partial _h S(x^{(1)}_{j,\varepsilon},x)\Big)$, the integrand is odd which means
\begin{equation}\label{tian72}
Q_1\Big(S(x^{(1)}_{j,\varepsilon},x),\partial _h S(x^{(1)}_{j,\varepsilon},x)\Big)=0.
\end{equation}
Now we calculate the term $Q_1\Big(S(x^{(1)}_{j,\varepsilon},x),\partial_h H(x^{(1)}_{j,\varepsilon},x)\Big)$. First, we know
\begin{equation*}
\begin{split}
\int_{\partial B_\theta(x^{(1)}_{j,\varepsilon})}& D_{\nu} S\big(x^{(1)}_{j,\varepsilon},x\big)
D_{x_i}\partial _h H\big(x^{(1)}_{j,\varepsilon},x\big) \\
=&D_{x_i}\partial _h H\big(x^{(1)}_{j,\varepsilon},x^{(1)}_{j,\varepsilon}\big) \int_{\partial B_\theta(x^{(1)}_{j,\varepsilon})}D_{\nu} S\big(x^{(1)}_{j,\varepsilon},x\big) +O\big(\theta\big)\int_{\partial B_\theta(x^{(1)}_{j,\varepsilon})}\big|D_{\nu} S(x^{(1)}_{j,\varepsilon},x)\big|\\
=&-D_{x_i}\partial _h H\big(x^{(1)}_{j,\varepsilon},x^{(1)}_{j,\varepsilon}\big)+O\big(\theta\big),
\end{split}
\end{equation*}
\begin{equation*}
\begin{split}
\int_{\partial B_\theta(x^{(1)}_{j,\varepsilon})}&D_{\nu} \partial _hH\big(x^{(1)}_{j,\varepsilon},x\big)
D_{x_i} S\big(x^{(1)}_{j,\varepsilon},x\big) \\
=&\sum^N_{l=1} D_{x_l} \partial _hH\big(x^{(1)}_{j,\varepsilon},x^{(1)}_{j,\varepsilon}\big)\int_{\partial B_\theta(x_{j,\varepsilon})}D_{x_i} S\big(x^{(1)}_{j,\varepsilon},x\big) \nu_l+O\big(\theta\big)\int_{\partial B_\theta(x^{(1)}_{j,\varepsilon})}\big|D S(x^{(1)}_{j,\varepsilon},x)\big|\\
=&D_{x_i} \partial _hH\big(x^{(1)}_{j,\varepsilon},x^{(1)}_{j,\varepsilon}\big) \int_{\partial B_\theta(x^{(1)}_{j,\varepsilon})}D_{x_i} S\big(x^{(1)}_{j,\varepsilon},x\big) \nu_i+O\big(\theta\big),
\end{split}
\end{equation*}
and
\begin{equation*}
\begin{split}
\int_{\partial B_\theta(x^{(1)}_{j,\varepsilon})}&\big\langle D S(x^{(1)}_{j,\varepsilon},x),D \partial _hH\big(x^{(1)}_{j,\varepsilon},x\big) \big\rangle \nu_i\\
=&\sum^N_{l=1}D_{x_l} \partial _h H\big(x^{(1)}_{j,\varepsilon},x^{(1)}_{j,\varepsilon}\big) \int_{\partial B_\theta(x^{(1)}_{j,\varepsilon})}D_{x_l} S\big(x^{(1)}_{j,\varepsilon},x\big) \nu_i+O\big(\theta\big)\int_{\partial B_\theta(x^{(1)}_{j,\varepsilon})}\big|D S\big(x^{(1)}_{j,\varepsilon},x\big)\big|\\
=&D_{x_i}\partial _h H\big(x^{(1)}_{j,\varepsilon},x^{(1)}_{j,\varepsilon}\big) \int_{\partial B_\theta(x^{(1)}_{j,\varepsilon})}D_{x_i} S\big(x^{(1)}_{j,\varepsilon},x\big) \nu_i+O\big(\theta\big),
\end{split}
\end{equation*}
which together imply
\begin{equation}\label{ll-2}
Q_1\Big(S\big(x^{(1)}_{j,\varepsilon},x\big),\partial_hH\big(x^{(1)}_{j,\varepsilon},x\big)\Big)
=D_{x_i}\partial _h H\big(x^{(1)}_{j,\varepsilon},x^{(1)}_{j,\varepsilon}\big)+O\big(\theta\big).
\end{equation}
Next  we calculate the term $Q_1\Big(\partial_h S(x^{(1)}_{j,\varepsilon},x), H(x^{(1)}_{j,\varepsilon},x)\Big)$. Using \eqref{8-17-1} and \eqref{8-17-2}, we find
\begin{equation}\label{tian76}
\begin{split}
\int_{\partial B_\theta\big(x^{(1)}_{j,\varepsilon}\big)}& D_{\nu} \partial _h  S\big(x^{(1)}_{j,\varepsilon},x\big)
D_{x_i}H\big(x^{(1)}_{j,\varepsilon},x\big) \\
=&\int_{\partial B_\theta(x^{(1)}_{j,\varepsilon})}D_{\nu} \partial _h  S\big(x^{(1)}_{j,\varepsilon},x\big)
\big\langle D D_{x_i}H\big(x^{(1)}_{j,\varepsilon},x^{(1)}_{j,\varepsilon}\big), x-x^{(1)}_{j,\varepsilon}\big\rangle+O\big(\theta\big)\\=&
\sum^N_{l=1}(1-N)\omega_N^{-1}D^2_{x_ix_l}H\big(x^{(1)}_{j,\varepsilon},x^{(1)}_{j,\varepsilon}\big)
\int_{|y|=\theta} \frac{y_hy_l}{|y|^{N+1}}+O\big(\theta\big)  \\
=&\frac{1-N}{N}D^2_{x_ix_h} H\big(x^{(1)}_{j,\varepsilon},x^{(1)}_{j,\varepsilon}\big)+O\big(\theta\big),
\end{split}
\end{equation}
\begin{equation}\label{tian77}
\begin{split}
\int_{\partial B_\theta(x^{(1)}_{j,\varepsilon})}&D_{\nu} H\big(x^{(1)}_{j,\varepsilon},x\big)
D_{x_i}\partial _h S\big(x^{(1)}_{j,\varepsilon},x\big) \\
=&\int_{\partial B_\theta(x^{(1)}_{j,\varepsilon})}D_{x_i} \partial _h  S\big(x^{(1)}_{j,\varepsilon},x\big)
\big\langle D^2H\big(x^{(1)}_{j,\varepsilon},x^{(1)}_{j,\varepsilon}\big)\big(x-x^{(1)}_{j,\varepsilon}\big),
\nu\big\rangle+O\big(\theta\big)\\
=&\sum^N_{l=1}\sum^N_{t=1} \omega_N^{-1}D^2_{x_tx_l}H\big(x^{(1)}_{j,\varepsilon},x^{(1)}_{j,\varepsilon}\big)\int_{|y|=\theta} \frac{y_ty_l}{|y|}
\big(\frac{\delta_{hi}}{|y|^{N}}-
\frac{N y_hy_i}{|y|^{N+2}}\big)+O\big(\theta\big)\\
=&\begin{cases}
-\frac{2}{N}D^2_{x_ix_h} H\big(x^{(1)}_{j,\varepsilon},x^{(1)}_{j,\varepsilon}\big)+O\big(\theta\big), ~& \mbox{for}~i\neq h,\\[1mm]
O\big(\theta\big), ~& \mbox{for}~i= h,
\end{cases}
\end{split}
\end{equation}
and
\begin{equation}\label{tian78}
\begin{split}
\int_{\partial B_\theta(x^{(1)}_{j,\varepsilon})}&\big\langle D\partial _h S\big(x^{(1)}_{j,\varepsilon},x\big),D H\big(x^{(1)}_{j,\varepsilon},x\big) \big\rangle \nu_i\\
=&\int_{\partial B_\theta(x^{(1)}_{j,\varepsilon})}
\langle D^2H\big(x^{(1)}_{j,\varepsilon},x^{(1)}_{j,\varepsilon}\big)\big(x-x^{(1)}_{j,\varepsilon}\big),D \partial _h  S\big(x^{(1)}_{j,\varepsilon},x\big)\big\rangle \nu_i+O\big(\theta\big)\\
=&\sum^N_{l=1}\sum^N_{t=1} \omega_N^{-1}D^2_{x_tx_l}H\big(x^{(1)}_{j,\varepsilon},x^{(1)}_{j,\varepsilon}\big)\int_{|y|=\theta}  y_t
\big(\frac{\delta_{hl}}{|y|^{N}}-
\frac{N y_hy_l}{|y|^{N+2}}\big)\frac{y_i}{|y|}+O\big(\theta\big)\\
=&\begin{cases}
-\frac{1}{N}D^2_{x_ix_h} H\big(x^{(1)}_{j,\varepsilon},x^{(1)}_{j,\varepsilon}\big)+O\big(\theta\big), ~& \mbox{for}~i\neq h,\\[1mm]
\frac{1}{N}D^2_{x_ix_h} H\big(x^{(1)}_{j,\varepsilon},x^{(1)}_{j,\varepsilon}\big)+O\big(\theta\big), ~& \mbox{for}~i= h.
\end{cases}
\end{split}
\end{equation}
From \eqref{tian76}--\eqref{tian78}, we get
\begin{equation}\label{ll-1}
Q_1\Big(S\big(x^{(1)}_{j,\varepsilon},x\big),\partial_hH\big(x^{(1)}_{j,\varepsilon},x\big)\Big)= D^2_{x_ix_h} H\big(x^{(1)}_{j,\varepsilon},x^{(1)}_{j,\varepsilon}\big)+O\big(\theta\big).
\end{equation}
Here the last two equalities hold by the fact $\Delta G\big(x,x^{(1)}_{j,\varepsilon}\big)=0$ for $x\in \Omega\backslash B_{\theta}(x^{(1)}_{j,\varepsilon})$.
Also since $H\big(x^{(1)}_{j,\varepsilon},x\big)$ and $D_{\nu} H\big(x^{(1)}_{j,\varepsilon},x\big)$ are bounded in $B_d\big(x^{(1)}_{j,\varepsilon}\big)$, it holds that
\begin{equation}\label{tian79}
Q_1\Big(H\big(x^{(1)}_{j,\varepsilon},x\big),\partial _h H\big(x^{(1)}_{j,\varepsilon},x\big)\Big)=O\big(\theta^{N-1}\big).
\end{equation}
Letting $\theta\rightarrow 0$, from \eqref{tian71}--\eqref{ll-2}, \eqref{ll-1} and \eqref{tian79}, we get
\begin{equation*}
\begin{split}
Q_1\Big(G\big(x^{(1)}_{j,\varepsilon},x\big), \partial _hG\big(x^{(1)}_{j,\varepsilon},x\big)\Big)=-D_{x_i}\partial _h H\big(x^{(1)}_{j,\varepsilon},x^{(1)}_{j,\varepsilon}\big)-D^2_{x_ix_h} H\big(x^{(1)}_{j,\varepsilon},x^{(1)}_{j,\varepsilon}\big)=-\frac{\partial^2 R(x^{(1)}_{j,\varepsilon})}{\partial x_i\partial x_h}.
 \end{split}
\end{equation*}
Next, for $m\neq j$,
\begin{equation}\label{tian81}
\begin{split}
Q_1 &\Big(G\big(x^{(1)}_{j,\varepsilon},x\big),\partial _h G\big(x^{(1)}_{m,\varepsilon},x\big)\Big) \\=&
 Q_1\Big(S\big(x^{(1)}_{j,\varepsilon},x\big),\partial _hG\big(x^{(1)}_{m,\varepsilon},x\big)\Big)
 - Q_1\Big(H\big(x^{(1)}_{j,\varepsilon},x\big),\partial _hG\big(x^{(1)}_{m,\varepsilon},x\big)\Big).
 \end{split}
\end{equation}
Similar to the estimates of \eqref{ll-2}, we know
\begin{equation}\label{tian82}
Q_1\Big(S\big(x^{(1)}_{j,\varepsilon},x\big),\partial _hG\big(x^{(1)}_{m,\varepsilon},x\big)\Big)=D_{x_i}\partial _h G\big(x^{(1)}_{m,\varepsilon},x^{(1)}_{j,\varepsilon}\big)+O\big(\theta\big).
\end{equation}
Obviously,
\begin{equation}\label{tian83}
\begin{split}
Q_1\Big(H\big(x^{(1)}_{j,\varepsilon},x\big),\partial _hG\big(x^{(1)}_{m,\varepsilon},x\big)\Big)=O\big(\theta^{N-1}\big).
 \end{split}
\end{equation}
Letting $\theta\rightarrow 0$, from \eqref{tian81}--\eqref{tian83},  we obtain
\begin{equation*}
\begin{split}
Q_1\Big(G\big(x^{(1)}_{j,\varepsilon},x\big), \partial _hG\big(x^{(1)}_{m,\varepsilon},x\big)\Big)= D_{x_i}\partial _h G\big(x^{(1)}_{m,\varepsilon},x^{(1)}_{j,\varepsilon}\big),~\mbox{for}~m\neq j.
 \end{split}
\end{equation*}
Also, for $m\neq j$,
\begin{equation}\label{tian84}
\begin{split}
Q_1 &\Big(G\big(x^{(1)}_{m,\varepsilon},x\big),\partial _h G\big(x^{(1)}_{j,\varepsilon},x\big)\Big) \\=&
 Q_1\Big(S\big(x^{(1)}_{m,\varepsilon},x\big),\partial _hG\big(x^{(1)}_{j,\varepsilon},x\big)\Big)
 - Q_1\Big(H\big(x^{(1)}_{m,\varepsilon},x\big),\partial _hG\big(x^{(1)}_{j,\varepsilon},x\big)\Big).
\end{split}\end{equation}
Similar to the estimates of \eqref{ll-1}, we know
\begin{equation}\label{tian85}
Q_1\Big(S(x^{(1)}_{m,\varepsilon},x),\partial _hG(x^{(1)}_{j,\varepsilon},x)\Big)=
D^2_{x_ix_h} G\big(x^{(1)}_{m,\varepsilon},x^{(1)}_{j,\varepsilon}\big)+O\big(\theta\big).
\end{equation}
Obviously,
\begin{equation}\label{tian86}
\begin{split}
Q_1\Big(H\big(x^{(1)}_{m,\varepsilon},x\big),\partial _hG\big(x^{(1)}_{j,\varepsilon},x\big)\Big)=O\big(\theta^{N-1}\big).
 \end{split}
\end{equation}
Letting $\theta\rightarrow 0$, from \eqref{tian84}--\eqref{tian86},  we obtain
\begin{equation*}
\begin{split}
Q_1\Big(G\big(x^{(1)}_{m,\varepsilon},x\big), \partial _hG\big(x^{(1)}_{j,\varepsilon},x\big)\Big)=
D^2_{x_ix_h} G\big(x^{(1)}_{m,\varepsilon},x^{(1)}_{j,\varepsilon}\big),~\mbox{for}~m\neq j.
 \end{split}
\end{equation*}
Finally,  since $G\big(x^{(1)}_{l,\varepsilon},x\big)$ and $D_{\nu} H\big(x^{(1)}_{l,\varepsilon},x\big)$ are bounded in $B_d\big(x^{(1)}_{j,\varepsilon}\big)$ for $l\neq j$, it holds that
\begin{equation*}
Q_1\Big(G\big(x^{(1)}_{m,\varepsilon},x\big),\partial _h G\big(x^{(1)}_{l,\varepsilon},x\big)\Big)=0,~\mbox{for}~m,l\neq j.
\end{equation*}
\end{proof}

\appendix
\renewcommand{\theequation}{A.\arabic{equation}}

\setcounter{equation}{0}

\section{Some basic estimates}\label{App-A}
\setcounter{equation}{0}
In this appendix, we give estimates that have been  used in the previous sections.
\begin{Lem}
For any small fixed $d>0$ and  $j=1,2\cdots,k$, it holds
\begin{equation}\label{A--1}
PU_{x_{j,\varepsilon},\lambda_{j,\varepsilon}}(x)
=O\Big(\frac{1}{\lambda^{(N-2)/2}_{j,\varepsilon}}\Big), ~\mbox{in}~ C^1
\Big(\Omega \backslash B_d(x_{j,\varepsilon})\Big).
\end{equation}
Moreover,
if we define  $\varphi_{x_{j,\varepsilon},\lambda_{j,\varepsilon}}(x)=
U_{x_{j,\varepsilon},\lambda_{j,\varepsilon}}(x)-PU_{x_{j,\varepsilon},\lambda_{j,\varepsilon}}(x)$, then it holds
\begin{equation}\label{A---1}
\varphi_{x_{j,\varepsilon},\lambda_{j,\varepsilon}}(x)=O\Big(\frac{1}{\lambda^{(N-2)/2}_{j,\varepsilon}}\Big), ~\mbox{in}~ C^1
\big(\Omega \big).
\end{equation}
\end{Lem}

\begin{proof}
See $\big[$\cite{Rey2}, Proposition 1$\big]$.
\end{proof}

\begin{Prop}\label{prop-A.2}
Let $N\geq 5$ and $u_{\varepsilon}(x)$ be a  solution of \eqref{1.1} with \eqref{4-6-1}. Then
\begin{equation}\label{lp21}
\|w_{\varepsilon}\|=
\begin{cases}
O\Big(\frac{1}{\lambda_\varepsilon^3}+
\frac{\varepsilon}{\lambda_\varepsilon^{3/2}}\Big), & \mbox{if}~N=5,\\[1.5mm]
O\Big(\frac{(\log \lambda_\varepsilon)^{2/3}}{\lambda_\varepsilon^{4}}+
  \frac{\varepsilon(\log \lambda_\varepsilon)^{2/3}}{\lambda^2_\varepsilon}\Big), & \mbox{if}~N=6,\\[1.5mm]
O\Big(\frac{1}{\lambda_\varepsilon^{(N+2)/2}}+
  \frac{\varepsilon}{\lambda^2_\varepsilon}\Big), & \mbox{if}~N>6,
\end{cases}
  \end{equation}
where $
  \lambda_\varepsilon:
  =\min\big\{\lambda_{1,\varepsilon},\cdots,\lambda_{k,\varepsilon}\big\}$.
\end{Prop}
\begin{proof}
See Proposition 4 in \cite{Rey2}.
\end{proof}

\begin{Prop}
Let $N\geq 5$ and $u_{\varepsilon}(x)$ be  a  solution of \eqref{1.1} with \eqref{4-6-1}. Then
\begin{equation}\label{at3}
u_{\varepsilon}(x)=
O\Big(\frac{1}{\lambda_\varepsilon^{(N-2)/2}}\Big), \mbox{in}~~C^1\Big(\Omega\backslash\bigcup^k_{j=1}B_{2d}(x_{j,\varepsilon})\Big).
\end{equation}
\end{Prop}
\begin{proof}
First, by Theorem 8.17 in \cite{Gilbarg}, for any $y\in \Omega\backslash\displaystyle\bigcup^k_{j=1}B_{2d}(x_{j,\varepsilon})$,  we find
\begin{flalign*}
\sup_{B_{d/2}(y)} u_\varepsilon(x)\leq C\Big( \|u_{\varepsilon}\|_{L^2\big(B_{d}(y)\big)}+\|f_\varepsilon\|_{L^{q/2}\big(B_{d}(y)\big)}\Big).
\end{flalign*}
where $f_\varepsilon= u_\varepsilon^{\frac{N+2}{N-2}}+\varepsilon u_\varepsilon$ and some $q>N$.

Since $u_{\varepsilon}(x)$ is a  solution of \eqref{1.1} with \eqref{4-6-1}, then $u_{\varepsilon}(x)$ is uniformly bounded on $\Omega\backslash\displaystyle\bigcup^k_{j=1}B_{2d}(x_{j,\varepsilon})$. From this, we can get $$\|f_\varepsilon\|_{L^{q/2}\big(B_{d}(y)\big)}=o(1)\sup_{B_{d}(y)} u_\varepsilon(x).$$
Next, by Moser iteration, we find \eqref{at3}.
\end{proof}
\begin{Lem}
It holds
\begin{equation}\label{luo--1}
\begin{split}
 \int_{B_d(x_{j,\varepsilon})}& u_{\varepsilon}^{\frac{N+2}{N-2}}(y) dy=\frac{A}{
(\lambda_{j,\varepsilon})^{(N-2)/2}}
+O\Big(\frac{1}{
 \lambda_{ \varepsilon} ^{(N+2)/2}}\Big) +o\Big(\frac{\varepsilon}{\lambda_\varepsilon^{2}}\Big),
\end{split}
\end{equation}
where $A$ is the constant in \eqref{a1}.
\end{Lem}
\begin{proof}
By H\"older's inequality, we calculate
\begin{equation*}
\int_{B_d(x_{j,\varepsilon})} PU^{\frac{4}{N-2}}_{x_{j,\varepsilon}, \lambda_{j,\varepsilon}}w_\varepsilon(x)=o\Big(\frac{\varepsilon}{\lambda_\varepsilon^{2}}\Big),~\mbox{and}~
\int_{B_d(x_{j,\varepsilon})} PU^{\frac{1}{3}}_{x_{j,\varepsilon}, \lambda_{j,\varepsilon}}w^2_\varepsilon(x)
=o\Big(\frac{\varepsilon}{\lambda_\varepsilon^{2}}\Big).
\end{equation*}
Also, we know
\begin{equation*}
\begin{split}
\int_{B_d(x_{j,\varepsilon})}&\Big(\displaystyle\sum_{l=1,l\neq j}^kPU_{x_{l,\varepsilon},\lambda_{l,\varepsilon}} +w_{\varepsilon}\Big)^{\frac{N+2}{N-2}}
=O\Big(\frac{1}{\lambda^{(N+2)/2}_\varepsilon}\Big)+o\Big(\frac{\varepsilon}{\lambda_\varepsilon^{2}}\Big).
\end{split}
\end{equation*}
On the other hand, for any $a,b\in \R^+$ and $p>1$, we have the following inequality:
\begin{equation*}
(a+b)^p-a^p-pa^{p-1}b=
 O(b^p+a^{p-p^*}b^{p^*}),~\mbox{with}~p^*=\min \{2,p\}.
\end{equation*}
Combining the above estimates, we get \eqref{luo--1}.
\end{proof}

Similar to the proof of \eqref{luo--1}, we can find following estimates.
\begin{Lem}
For any $j=1,\cdots,N$,
it holds
\begin{equation}\label{tian-2}
\int_{B_d(x_{j,\varepsilon})}\big|y-x_{j,\varepsilon}\big|u_{\varepsilon}^{\frac{N+2}{N-2}}(y) dy=
O\Big(\frac{1}{
 \lambda_{\varepsilon}^{(N+2)/2}}\Big)+o\Big(\frac{\varepsilon}{\lambda_\varepsilon^{2}}\Big),
\end{equation}
and
\begin{equation}\label{luopeng1}
\int_{B_d(x_{j,\varepsilon})} \big|y-x_{j,\varepsilon}\big|^3 u_{\varepsilon}^{\frac{N+2}{N-2}}(y) dy
=O\Big(\frac{1}{
 \lambda_{ \varepsilon} ^{(N+2)/2}}\Big)+o\Big(\frac{\varepsilon}{\lambda_\varepsilon^{2}}\Big).
\end{equation}
Also for any $i,j,m=1,\cdots,N$, it holds
\begin{equation}\label{8-17-6}
\int_{B_d(x_{j,\varepsilon})}\big(y_i-x_{j,\varepsilon,i}\big)
\big(y_m-x_{j,\varepsilon,m}\big)  u_{\varepsilon}^{\frac{N+2}{N-2}}(y) dy
=\frac{\delta_{im}\log \lambda_{i,\varepsilon}}{
 \lambda_{i,\varepsilon}^{(N+2)/2}} + O\Big(\frac{1}{
 \lambda_{\varepsilon}^{(N+2)/2}}\Big)+o\Big(\frac{\varepsilon}{\lambda_\varepsilon^{2}}\Big),
\end{equation}
where $\delta_{im}=1$ if $i=m$ and $\delta_{im}=0$ if $i\neq m$.
\end{Lem}

\noindent\textbf{Acknowledgments} The authors would like to thank the referees for their careful reading of this paper and for their valuable suggestions to improve the presentation and the style of the paper.
Cao, Luo and Peng were supported  by the Key Project of NSFC (No.11831009). Cao was partially supported by NSFC grants (No.11771469). Luo  was partially supported by NSFC grants (No.11701204) and the China Scholarship Council.
\renewcommand\refname{References}
\renewenvironment{thebibliography}[1]{%
\section*{\refname}
\list{{\arabic{enumi}}}{\def\makelabel##1{\hss{##1}}\topsep=0mm
\parsep=0mm
\partopsep=0mm\itemsep=0mm
\labelsep=1ex\itemindent=0mm
\settowidth\labelwidth{\small[#1]}%
\leftmargin\labelwidth \advance\leftmargin\labelsep
\advance\leftmargin -\itemindent
\usecounter{enumi}}\small
\def\newblock{\ }
\sloppy\clubpenalty4000\widowpenalty4000
\sfcode`\.=1000\relax}{\endlist}
\bibliographystyle{model1b-num-names}

\end{document}